\documentclass[10pt,reqno]{amsart}

\usepackage{fancyhdr}
\usepackage[nonewpage]{imakeidx}
\makeindex[title=Index of definitions,columns=2]
\makeindex[name=nota, title=Index of notation, columns=3]

\usepackage{amscd}
\usepackage{amsfonts}
\usepackage{amsmath}
\usepackage{amssymb}
\usepackage{amsthm}
\usepackage{enumitem}
\usepackage{float}
\usepackage{mathtools}
\usepackage{mathrsfs}
\usepackage{needspace}
\usepackage{setspace}
\usepackage{tikz}
\usepackage{version}

\usepackage[colorlinks=true, linkcolor=blue, urlcolor=blue, citecolor=blue,%
pagebackref=true]{hyperref}

\onehalfspace

\theoremstyle{plain}
\newtheorem{lemma}{Lemma}[section]
  \newtheorem{proposition}[lemma]{Proposition}
  \newtheorem{corollary}[lemma]{Corollary}
  \newtheorem{theorem}[lemma]{Theorem}
  \newtheorem*{theorem*}{Theorem}
  
  \newtheorem{observation}[lemma]{Observation}
  
  \newtheorem*{fact*}{Fact}
  \newtheorem*{claim*}{Claim}
  \newtheorem{claim}[lemma]{Claim}

  \newtheorem{slogan}[lemma]{Slogan}

\theoremstyle{definition}
  \newtheorem{definition}[lemma]{Definition}

  \newtheorem{remark}[lemma]{Remark}
  \newtheorem{example}[lemma]{Example}

\theoremstyle{remark}
  \newtheorem{case}{Case}

\setenumerate{label=\textup{(}\roman*\textup{)}}

\makeatletter
\newsavebox{\@brx}
\newcommand{\llangle}[1][]{\savebox{\@brx}{\(\m@th{#1\langle}\)}%
  \mathopen{\copy\@brx\kern-0.5\wd\@brx\usebox{\@brx}}}
\newcommand{\rrangle}[1][]{\savebox{\@brx}{\(\m@th{#1\rangle}\)}%
  \mathclose{\copy\@brx\kern-0.5\wd\@brx\usebox{\@brx}}}
\makeatother

\newcommand{\wt}{\widetilde}

\usepackage[alphabetic,abbrev,msc-links,nobysame]{amsrefs}

\title{The structure theory of Nilspaces I}
\author{Yonatan Gutman, Freddie Manners and P\'{e}ter P. Varj\'{u}}

\address{Yonatan Gutman, Institute of Mathematics, Polish Academy of Sciences,
ul. \'{S}niadeckich~8, 00-656 Warszawa, Poland.}
\email{y.gutman@impan.pl}

\address{Freddie Manners, Mathematical Institute, Radcliffe Observatory Quarter, Woodstock Road, Oxford OX2 6GG}
\email{Frederick.Manners@maths.ox.ac.uk}

\address{P\'{e}ter P. Varj\'{u}, Centre for Mathematical Sciences,
Wilberforce Road, Cambridge CB3 0WA,
UK}
\email{pv270@dpmms.cam.ac.uk}

\thanks{YG was partially supported by the ERC Grant \emph{Approximate Algebraic Structures and Applications} and the NCN (National Science Center, Poland) grant 2016/22/E/ST1/00448. PPV was supported by the Royal Society.}


\date{\today}

\begin{document}

\newcommand*{\singlesquare}[4]{
  \draw[thin,black] (0,0) node[below left] {#1} -- (1, 0) node[below right] {#2} -- (1, 1) node[above right] {#4} -- (0, 1) node[above left] {#3} -- (0,0);
}

\newcommand*{\doublesquare}[7]{
  \begin{scope}[xscale=#7]
    \draw[thin,black] (0,0) node[below left] {#4} -- (1, 0) node[below] {#5} -- (2, 0) node[below right] {#6} -- (2, 1) node[above right] {#3} -- (1, 1) node[above] {#2} -- (0, 1) node[above left] {#1} -- (0,0);
    \draw[thin,black] (1, 0) -- (1, 1);
  \end{scope}
}

\newcommand*{\threecube}[9]{
  \begin{scope}[x={(#9, 0)}, y={(0, 1)}, z={(0.352, 0.317)}, scale=2]
    \draw (0,0,0) node[below left] {#1} -- (0,0,1) node[below right] {#3} -- (0,1,1) node[above] {#7} -- (0,1,0) node[above left] {#5} -- cycle;
    \draw (0,0,0) -- (1,0,0) node[below right] {#2} -- (1,0,1) node[right] {#4} -- (0,0,1) -- cycle;
    \draw (1,1,1) node[above right] {#8} -- (0,1,1) -- (0,0,1) -- (1,0,1) -- cycle;

    \draw (0,0,0) -- (1,0,0) -- (1,1,0) node [above left] {#6} -- (0,1,0) -- cycle;
    \draw (1,1,1) -- (1,1,0) -- (1,0,0) -- (1,0,1) -- cycle;
    \draw (1,1,1) -- (0,1,1) -- (0,1,0) -- (1,1,0) -- cycle;
  \end{scope}
}

\newcommand*{\inlinetikz}[1]{
  \begin{figure}[H]
    \centering
    \begin{tikzpicture}
      #1
    \end{tikzpicture}
  \end{figure} \noindent
}
\usetikzlibrary{patterns}

\newcommand{\eps}[0]{\varepsilon}

\newcommand{\AAA}[0]{\mathbb{A}}
\newcommand{\CC}[0]{\mathbb{C}}
\newcommand{\EE}[0]{\mathbb{E}}
\newcommand{\FF}[0]{\mathbb{F}}
\newcommand{\NN}[0]{\mathbb{N}}
\newcommand{\PP}[0]{\mathbb{P}}
\newcommand{\QQ}[0]{\mathbb{Q}}
\newcommand{\RR}[0]{\mathbb{R}}
\newcommand{\TT}[0]{\mathbb{T}}
\newcommand{\ZZ}[0]{\mathbb{Z}}

\newcommand{\cA}[0]{\mathcal{A}}
\newcommand{\cB}[0]{\mathcal{B}}
\newcommand{\cC}[0]{\mathcal{C}}
\newcommand{\cD}[0]{\mathcal{D}}
\newcommand{\cE}[0]{\mathcal{E}}
\newcommand{\cF}[0]{\mathcal{F}}
\newcommand{\cH}[0]{\mathcal{H}}
\newcommand{\cG}[0]{\mathcal{G}}
\newcommand{\cK}[0]{\mathcal{K}}
\newcommand{\cM}[0]{\mathcal{M}}
\newcommand{\cN}[0]{\mathcal{N}}
\newcommand{\cP}[0]{\mathcal{P}}
\newcommand{\cR}[0]{\mathcal{S}}
\newcommand{\cS}[0]{\mathcal{S}}
\newcommand{\cT}[0]{\mathcal{S}}
\newcommand{\cU}[0]{\mathcal{U}}
\newcommand{\cW}[0]{\mathcal{W}}
\newcommand{\cX}[0]{\mathcal{X}}
\newcommand{\cY}[0]{\mathcal{Y}}
\newcommand{\cZ}[0]{\mathcal{Z}}

\newcommand{\fg}[0]{\mathfrak{g}}
\newcommand{\fk}[0]{\mathfrak{k}}
\newcommand{\fZ}[0]{\mathfrak{Z}}

\newcommand{\bmu}[0]{\boldsymbol\mu}

\newcommand{\AUT}[0]{\mathbf{Aut}}
\newcommand{\Aut}[0]{\operatorname{Aut}}
\newcommand{\Frob}[0]{\operatorname{Frob}}
\newcommand{\GI}[0]{\operatorname{GI}}
\newcommand{\HK}[0]{\operatorname{HK}}
\newcommand{\HOM}[0]{\mathbf{Hom}}
\newcommand{\Hom}[0]{\operatorname{Hom}}
\newcommand{\Ind}[0]{\operatorname{Ind}}
\newcommand{\Lip}[0]{\operatorname{Lip}}
\newcommand{\LHS}[0]{\operatorname{LHS}}
\newcommand{\RHS}[0]{\operatorname{RHS}}
\newcommand{\Sub}[0]{\operatorname{Sub}}
\newcommand{\id}[0]{\operatorname{id}}
\newcommand{\image}[0]{\operatorname{Im}}
\newcommand{\poly}[0]{\operatorname{poly}}
\newcommand{\trace}[0]{\operatorname{Tr}}
\newcommand{\sig}[0]{\ensuremath{\tilde{\cS}}}
\newcommand{\psig}[0]{\ensuremath{\cP\tilde{\cS}}}
\newcommand{\metap}[0]{\operatorname{Mp}}
\newcommand{\symp}[0]{\operatorname{Sp}}
\newcommand{\dist}[0]{\operatorname{dist}}
\newcommand{\stab}[0]{\operatorname{Stab}}
\newcommand{\HCF}[0]{\operatorname{hcf}}
\newcommand{\LCM}[0]{\operatorname{lcm}}
\newcommand{\SL}[0]{\operatorname{SL}}
\newcommand{\GL}[0]{\operatorname{GL}}
\newcommand{\rk}[0]{\operatorname{rk}}
\newcommand{\sgn}[0]{\operatorname{sgn}}
\newcommand{\uag}[0]{\operatorname{UAG}}
\newcommand{\freiman}[0]{Fre\u{\i}man}
\newcommand{\tf}[0]{\operatorname{tf}}
\newcommand{\codim}[0]{\operatorname{codim}}

\newcommand{\Conv}[0]{\mathop{\scalebox{1.5}{\raisebox{-0.2ex}{$\ast$}}}}
\newcommand{\bs}[0]{\backslash}

\newcommand{\heis}[3]{ \left(\begin{smallmatrix} 1 & \hfill #1 & \hfill #3 \\ 0 & \hfill 1 & \hfill #2 \\ 0 & \hfill 0 & \hfill 1 \end{smallmatrix}\right)  }

\newcommand{\uppar}[1]{\textup{(}#1\textup{)}}

\newcommand{\pol}[0]{\l}

\begin{abstract}
This paper forms the first part of a series by the authors \cites{GMV2,GMV3} concerning the structure theory of \emph{nilspaces} of Antol\'\i n Camarena and Szegedy. A nilspace is a compact space $X$ together with closed collections of \emph{cubes} $C^n(X)\subseteq X^{2^n}$, $n=1,2,\ldots$ satisfying some natural axioms. Antol\'\i n Camarena and Szegedy proved that from these axioms it follows that (certain) nilspaces are isomorphic (in a strong sense) to an inverse limit of nilmanifolds. The aim of our project is to provide a new self-contained treatment of this theory and give new applications to topological dynamics.

This paper provides an introduction to the project from the point of view of applications to
higher order Fourier analysis. We define and explain the basic definitions and constructions related to cubespaces and nilspaces and develop the \emph{weak structure theory}, which is the first stage of the proof of the main structure theorem for nilspaces. Vaguely speaking, this asserts that a nilspace can be built as a finite tower of extensions where each of the successive fibers is a compact abelian group.

We also make some modest innovations and extensions to this theory.  In particular, we consider a class of maps that we term \emph{fibrations}, which are essentially equivalent to what are termed \emph{fiber-surjective morphisms} by Anatol\'{\i}n Camarena and Szegedy; and we formulate and prove a relative analogue of the weak structure theory alluded to above for these maps.  These results find applications elsewhere in the project.
\end{abstract}

\maketitle{}

\tableofcontents{}

\section{An introduction to the project}

This is the first in a series of three papers by the authors concerned with the structure theory of \emph{cubespaces}, the others being \cite{GMV2} and \cite{GMV3}.  Informally, a cubespace is a compact metric space $X$, together with some notion of when a collection of $2^k$ points of $X$ form a ``$k$-cube'', subject to certain further axioms.

The study of cubespaces as axiomatic objects is becoming established as a major theme in the nascent area of higher order Fourier analysis.  This programme has its origins in work of Host and Kra \cite{HK08}, where these objects appeared under the name of  ``parallelepiped structures''.  The study of these objects was furthered by Antol\'\i n Camarena and Szegedy \cite{CS12}, who in the same work formulated a strong structure theorem for cubespaces, subject to certain further hypotheses.  The gist of the structure theorem is that, subject to these extra assumptions, all cubespaces arise in some sense from \emph{nilmanifolds} $X = G/\Gamma$; i.e.~they come from compact homogeneous spaces of nilpotent Lie groups.

The papers of Candela \cites{Can1,Can2} expand on \cite{CS12}, providing more detailed proofs. He also includes
several additional results implicit in \cite{CS12}, particularly about continuous systems of measures.

This structure theory has applications in two broad areas: additive combinatorics, and in particular the inverse theorem for the Gowers norms; and topological dynamics and ergodic theory.  In this paper, we will approach this project from the point of view of someone interested primarily in understanding the former, and in particular Szegedy's proof \cite{S12} of the inverse theorem for the Gowers norms, which relies heavily on this structural result.  The third paper in the series \cite{GMV3} will contain an introduction to the project focussed instead on applications to topological dynamics.  Although the results proved in each paper are strongly relevant to the other, a reader more interested in the dynamical perspective might prefer to start there and refer to this paper only later.

The main purpose of our entire project is to provide a self-contained proof of the main structural result of \cite{CS12}.  In many places our approach will follow that of \cite{CS12}, or of previous work \cites{HK05,HM07,HK08,HKM10,GT10}, very closely.  At other times, we give notably different arguments: sometimes because they are arguably simpler; sometimes to avoid certain technical difficulties (although perhaps at the expense of introducing different ones); and sometimes to obtain sharper conclusions.  Furthermore, we do obtain new results in particular in the dynamical setting, and some of our proofs are optimized so as to prove these concurrently.

Our primary goal, however, is to obtain a fuller understanding of cubespaces and related structures.  In our view, this understanding will continue to find new relevance as the field of higher order Fourier analysis matures.

\subsection{Obstructions to Gowers uniformity}

It is not immediately clear why the structure theory of cubespaces should be relevant to proving an inverse theorem for the Gowers norms: indeed, the deduction of the latter from the former is by no means straightforward \cite{S12}.  To provide some motivation, we will instead sketch a proof of a kind of converse.  That is, we will argue that cubespaces (subject to certain additional hypotheses) are obstructions to Gowers uniformity, and that therefore the inverse theorem itself implies that cubespaces are somehow related to nilmanifolds.

We will not recall in full all the relevant definitions (of the uniformity norms, nilmanifolds, polynomial maps, nilsequences etc.), referring the reader to \cites{GT10,ben-book} but will informally sketch the set-up to motivate our discussion.  For notational simplicity we will focus only on the $U^3$ norm, although these remarks apply more generally.
\begin{enumerate}
  \item Given a function $f \colon \ZZ/N\ZZ \to \CC$ for $N$ a prime (say), the uniformity norm $\|f\|_{U^3}$ is defined in terms of an average over \emph{cube} or \emph{parallelepiped} configurations in $\ZZ/N\ZZ$, e.g.
    \[
      \|f\|_{U^3}^8 = \EE_{c \in C^3(\ZZ/N\ZZ)} f(c_{000}) \overline{f(c_{001}) f(c_{010})} f(c_{011}) \overline{f(c_{100})} f(c_{101}) f(c_{110}) \overline{f(c_{111})}
    \]
    where $C^3(\ZZ/N\ZZ) \subseteq (\ZZ/N\ZZ)^8$ consists of all tuples
    \[
      (c_\omega)_{\omega \in \{0,1\}^3} = x + \omega_1 h_1 + \omega_2 h_2 + \omega_3 h_3
    \]
    for $x, h_1, h_2, h_3 \in \ZZ/N\ZZ$.  It is informative to think of these tuples as elements written on the vertices of a $3$-dimensional cube, as shown.

    \inlinetikz{\threecube{$x$}{$x+h_1$}{$x+h_2$}{$x+h_1+h_2$}{$x+h_3$}{$x+h_1+h_3$}{$x+h_2+h_3$}{$x+h_1+h_2+h_3$}{1.6}}

  \item Given a nilmanifold $G/\Gamma$ (with suitable additional structure) there is also a notion of cubes on $G/\Gamma$, given by a construction due to Host and Kra \cites{HK05,HK08}.  Specifically, suppose $G$ is a $2$-step nilpotent Lie group and $\Gamma$ a discrete co-compact subgroup; then there is a closed subset $C^3(G/\Gamma) \subseteq (G/\Gamma)^8$ somewhat analogous to the parallelepipeds in an abelian group.
  \item There is a plentiful supply of maps $p \colon \ZZ/N\ZZ \to G/\Gamma$ which send cubes to cubes; that is, $p(c) \in C^3(G/\Gamma)$ for any $c \in C^3(\ZZ/N\ZZ)$ (with $p$ applied pointwise).
  \item The cubes on $G/\Gamma$ satisfy a \emph{corner constraint}: given $c \in C^3(G/\Gamma)$, if we know $c_{000}, \dots, c_{110}$ then the last vertex $c_{111}$ is uniquely determined by the others.
  \item By a \emph{nilsequence} on $\ZZ/N\ZZ$ we mean a function of the form $\phi = F \circ p$ where $p$ is as above and $F \colon G/\Gamma \to \CC$ is Lipschitz.
\end{enumerate}

The inverse theorem states roughly that (for very large $N$) if $f \colon \ZZ/N\ZZ \to \CC$, $|f| \le 1$ has $\|f\|_{U^3} \ge \delta$ then $|\langle f, \phi \rangle| \gtrsim_\delta 1$ for some nilsequence $\phi$ whose ``complexity'' is bounded in terms of $\delta$. One can say that nilsequences are the only ``obstructions to Gowers uniformity'': the only reason for a function to have large Gowers norm is if it correlates with a nilsequence.

All known proofs of this statement are significantly hard.  By contrast, the converse statement -- if a function correlates with a nilsequence then it has large Gowers norm -- is quite straightforward.  We will sketch a result along these lines, following the argument from \cite{green-tao-u3}*{Proposition 12.6}. Although the details of the proof are not logically required in what follows, it is useful to record them to motivate Observation \ref{obs:obstruct} below.

\begin{claim*}
  Suppose $\phi = F \circ p$ is a nilsequence in the sense of (v) with $\|F\|_{\infty} \le 1$. Let $f \colon \ZZ/N\ZZ \to \CC$ be such that $|f| \le 1$ and $|\langle f, \phi\rangle| \ge \delta$.  Then $\|f\|_{U^3} = \Omega_{F, \delta}(1)$.
\end{claim*}
The key point is that the lower bound on $\|f\|_{U^3}$ depends only on the choice of $F$ (and so implicitly of $G/\Gamma$) and on $\delta$; \emph{not} on $N$ or $p$.
\begin{proof}[Proof sketch]
  By property (iv), there is a closed subset $Y \subseteq (G/\Gamma)^7$ and a function $\tau \colon Y \to (G/\Gamma)$ such that
  \[
    C^3(G/\Gamma) = \{( \tau(y), y) \colon y \in Y \} \ .
  \]
  Hence we get a continuous function $F \circ \tau$ on $Y$.  By Tietze's extension theorem, we can extend this to a bounded continuous function $H$ on $(G/\Gamma)^7$.  Any continuous function on a product space can be approximated (up to a small error in $L^\infty$) by a finite sum of products of functions on the factors: that is, we can decompose
  \[
    H(x_1,\dots,x_7) = \sum_{i=1}^k R_1^{(i)}(x_1) \dots R_7^{(i)}(x_7) + H_{\text{err}}
  \]
  for some bounded continuous functions $R_j^{(i)} \colon G/\Gamma \to \CC$, and some bounded continuous $H_{\text{err}} \colon (G/\Gamma)^7 \to \CC$ such that $\|H_{\text{err}}\|_\infty = o_{F;k \to \infty}(1)$.

  Now, for any $x, h_1, h_2, h_3$ in $\ZZ/N\ZZ$ we have that $(p(x), p(x+h_1), p(x+h_2), p(x+h_1+h_2), \dots)$ is in $C^3(G/\Gamma)$, and so
  \begin{align*}
    \phi(x) &= F(p(x)) \\
            &= F(\tau(p(x+h_1), p(x+h_2), p(x+h_1+h_2), \dots, p(x+h_1+h_2+h_3))) \\
            &= \sum_{i=1}^k R_1^{(i)}(p(x+h_1)) R_2^{(i)}(p(x+h_2)) \dots R_7^{(i)}(p(x+h_1+h_2+h_3)) + o_{k\to\infty}(1) \ .
  \end{align*}
  Since $|\langle f, \phi \rangle|$ is bounded away from zero, we deduce that
  \[
    \Bigl| \EE_{x,h_1,h_2,h_3} f(x) \overline{R_1^{(i)}(p(x+h_1)) R_2^{(i)}(p(x+h_2)) \dots R_7^{(i)}(p(x+h_1+h_2+h_3))} \Bigr |
  \]
  is bounded away from zero for some $i \in \{1,\dots,k\}$ (after choosing $k$ appropriately in terms on $F$ and $\delta$).  But this expression is a ``Gowers inner product'' of eight functions, and by the Gowers--Cauchy--Schwarz inequality (essentially multiple applications of Cauchy--Schwarz), this quantity is bounded above by
  \[
    \|f\|_{U^3} \left\|R_1^{(i)} \circ p\right\|_{U^3} \dots \left\|R_7^{(i)} \circ p\right\|_{U^3}
  \]
  and noting that $\left\|R_j^{(i)} \circ p \right\|_{U^3} \le \left\|R_j^{(i)}\right\|_\infty$ which is bounded, we get a lower bound on $\|f\|_{U^3}$ as required.
\end{proof}

The key point is that the only properties of nilmanifolds, nilsequences etc.~that we have used are those described in (ii), (iii) and (iv) above.  So we have in fact shown:

\begin{observation}
  \label{obs:obstruct}
  If $X$ is \emph{any} compact metric space equipped with some suitable notion of ``cubes'' as in (ii), having an abundance of cube-preserving maps $\ZZ/N\ZZ \to X$ as in (iii), and satisfying a corner constraint as in (iv), then functions of the form $F \circ p$ as in (v) obstruct Gowers uniformity on $\ZZ/N\ZZ$ in the sense of the above claim.
\end{observation}

Let us refer to such a space informally for now as a ``nil-object'' (the formal notion of a \emph{nilspace} will be introduced later).  Then the above observation can be summarized as follows.

\begin{slogan}
  Any ``nil-object'' is an obstruction to Gowers uniformity.
\end{slogan}

But now, the inverse theorem for the $U^3$ norm tells us that any function $F \circ p$ coming from this construction must have something to do with a genuine nilsequence.  Given some technical hypotheses, one can deduce that any such nil-object $X$ must be very closely related to an actual nilmanifold $G/\Gamma$.

The conclusion of work of Szegedy \cite{S12} is that it is possible to go in the other direction.  He argues that all functions $f$ with $\|f\|_{U^3}$ somewhat large correlate with something of the form $F \circ p$ where $p \colon \ZZ/N\ZZ \to X$ and $F \colon X \to \CC$ is continuous, for \emph{some} space $X$ equipped with a notion of cubes, and some cube-preserving $p$, obeying some fairly reasonable additional axioms.  Moreover, this proof is essentially purely analytic in nature, making no attempt to say anything about the structure of $X$.

Assuming this, we conclude:

\begin{slogan}
  The class of all ``nil-objects'' corresponds \emph{precisely} to the obstructions to Gowers uniformity.
\end{slogan}

Hence, the inverse theorem for the Gowers norms is essentially \emph{equivalent} to classifying nil-objects $X$, showing in effect that they are all -- essentially -- nilmanifolds.  This structural result is the goal of \cite{CS12}.

\subsection{An outline of this paper}

The formal notion capturing properties (ii)-(iv) above and replacing the informal concept of a ``nil-object'', is what we term a \emph{nilspace}. A nilspace is a compact topological space satisfying certain ``nilspace axioms'', which are both very abstract (e.g.~they do not explicitly mention any group structure) but simultaneously strong enough for a strong structural result to hold.

The remaining tasks of this paper are as follows:

\begin{enumerate}
  \item to explain the nilspace axioms formally;
  \item to state a precise version of the structure theorem;
  \item to outline the stages in the proof of this theorem; and
  \item to prove some weaker structural results that constitute the first stage of the proof.
\end{enumerate}

The remaining parts of the proof of the structure theorem appear in the companion papers \cites{GMV2, GMV3}, and we will provide pointers to the relevant sections of these works in the outline.

Most of our discussion will be of an expository or heuristic nature.  The only parts that constitute steps in the rigorous proof of the structure theorem are Section \ref{sec:elementary-v2}, and in a way the definitions in Section \ref{sec:cubespace-defns}.  A reader already familiar with this whole approach and seeking only a complete proof could therefore read only these, together with the companion papers \cites{GMV2, GMV3}; needless to say we do not recommend this strategy.

Most results proved in this paper are due to Antol\'\i n Camarena and Szegedy \cite{CS12}, and the majority of the ideas we will discuss originate in \cite{HK08} or \cite{CS12}.
The primary exception is the ``relative'' version of this ``weak structure theory'' treated in Section \ref{sec:elementary-v2}.
This generalization is fairly mechanical when stated in the language of \emph{fibrations}, which are a class of maps between cubespaces satisfying certain additional hypotheses.  This notion is almost equivalent to that of a \emph{fiber-surjective morphism} appearing in \cite{CS12}; indeed, the latter is only defined for maps between two nilspaces, and in that case the definitions are equivalent.  Our reasons for introducing the former are twofold: the definition of a fibration makes sense in greater generality, which we do actually use; and in our view, the alternative definition makes the analogy between relative and non-relative versions more transparent.

Before embarking on (i), we recall a small amount of background about Host--Kra cube groups, which puts the abstract definitions in some context.  This is done in Section \ref{sec:hk}.

The definitions themselves are then expounded in Section \ref{sec:cubespace-defns}.

With these in place, we are in a position to state the full structural result in Section \ref{sec:structural-thm}.  Then, we give a very heuristic outline of the high-level stages of the proof.

This outline is expanded upon in Section \ref{sec:remaining-summary-v2}, which gives a detailed overview of each stage of the proof.

The final two sections, Section \ref{sec:weak-overview} and Section \ref{sec:elementary-v2}, are concerned with the first of these stages, which we have been calling the ``weak structure theory''.  Section \ref{sec:weak-overview} deals with the ``standard'' theory as it appears in \cite{CS12}.  In Section \ref{sec:elementary-v2}, we consider \emph{relative} versions of essentially all the elementary theory, which generalize the earlier results and which will be needed elsewhere in this project.

In the Appendix, we resume the discussion of Host--Kra cubes from Section \ref{sec:hk}, and make a detailed study of cubes on nilmanifolds in some example cases.

Finally, we draw the reader's attention to the two indices at the end of the paper, which can be used to locate the original definitions of the many terms and symbols introduced throughout.

\subsection{Acknowledgments}

First and foremost we owe gratitude to Bernard Host who introduced us to the subject and to Omar Antol\'\i n Camarena and Bal\'azs Szegedy
whose groundbreaking work \cite{CS12} was a constant inspiration
for us.

We would like to thank Emmanuel Breuillard, J\'er\^ome Buzzi,
Yves de Cornulier, Sylvain Crovisier, Eli Glasner, Ben Green, Bernard Host, Micha\pol{} Rams, Bal\'azs Szegedy, Anatoly Vershik and Benjamin
Weiss for helpful discussions. We are grateful to Pablo Candela and Bryna Kra for a careful reading of a preliminary version.

We are grateful to the referee for her or his careful reading of our paper and for her or his many
helpful comments, which greatly improved the presentation of the paper.
\section{Nilmanifolds and their Host--Kra cubes}
\label{sec:hk}

The key motivating examples of nilspaces (being in some sense the only examples) are nilmanifolds $G/\Gamma$ equipped with their Host--Kra cube structures. Certainly anything we define or prove about general nilspaces should hold true for these spaces, and they provide a good source of intuition to guide the definitions and arguments in the abstract setting.

With this in mind, we will now briefly recall the relevant constructions and a few properties.  A much more substantial exposition of this theory, including a number of examples explored in depth, is given in Appendix \ref{app:hk}; this also includes omitted proofs from this section.  The reader unfamiliar with this area might wish to read this exposition before continuing with the bulk of the paper.

The notions and ideas presented in this section originate from \cites{HK05,HK08, GT10}.

We first recall the notion of a filtration.

\begin{definition}
  Let $G$ be a topological group.  By a \emph{filtration of degree $s$}\index{filtration} on $G$, we mean a sequence $G = G_0 \supseteq G_1 \supseteq \dots \supseteq G_{s+1} = \{\id\}$ of closed subgroups, with the property that $[G_i, G_j] \subseteq G_{i+j}$ for all $i,j \ge 0$.  (Here $[\cdot,\cdot]$ denotes commutation, and we use the convention $G_i = \{\id\}$ for all $i \ge s+1$.)

The filtration is called {\em proper}\index{filtration!proper} if $G_0=G_1$.
\end{definition}

For most of the theory we will only consider proper filtrations, but we find it useful to permit $G_0\neq G_1$ in
the definition.
Note that  if $G$ admits a proper filtration of degree $s$, then $G$ is necessarily nilpotent, with nilpotency class at most $s$.

We write $G_\bullet$ in place of $G$ if we wish to emphasize that a group is equipped with a particular filtration.

The standard example of a filtration is the \emph{lower central series}\index{lower central series} filtration on a group, given by $G_0 = G_1 = G$ and $G_{i+1} = [G, G_i]$ for each $i \ge 1$ (for a proof that this is indeed a filtration see \cite{MKS66}*{Theorem 5.3}).  This is the minimal proper filtration on $G$, in the sense that every proper filtration contains it termwise.

The fundamental construction concerning filtered groups is that of the \emph{Host--Kra cube groups}.   These are designed to be the appropriate analogues of parallelepipeds in abelian groups, in the setting of general filtered groups $G_\bullet$.  To describe them, we will first set up some notation.

For a set $X$, we use the notation $X^{\{0,1\}^k}$ to mean the space of all functions $\{0,1\}^k \to X$.  Concretely, this is just $X^{2^k}$ but the reader should always imagine the elements written at the vertices of the discrete cube $\{0,1\}^k$, as in:

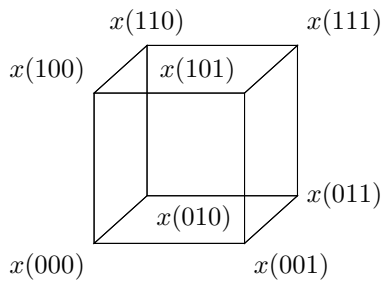
\inlinetikz{\threecube{$x(000)$}{$x(001)$}{$x(010)$}{$x(011)$}{$x(100)$}{$x(101)$}{$x(110)$}{$x(111)$}{1}}

We denote by $[k]$\index[nota]{$[k]$} the set $\{1,\ldots, k\}$.
We find it convenient to identify subsets of $[k]$
with vertices of the discrete cube $\{0,1\}^k$, mapping sets to their indicator functions.
In particular, we write $\omega_1\subseteq\omega_2$ for two vertices
if $\omega_1(j)\le\omega_2(j)$ for all $j\in [k]$.

By a \emph{face}\index{face} of the discrete cube $\{0,1\}^k$ we mean a sub-cube obtained by fixing some subset of the coordinates.  An \emph{upper face}\index{face!upper} is one obtained by fixing some subset of the coordinates to equal $1$; we can write this as $\{ \omega \in \{0,1\}^k \colon \omega \supseteq S \}$ for some $S \subseteq [k]$.  For instance, in
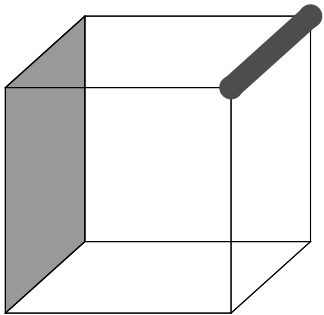
\inlinetikz{
  \begin{scope}[x={(1, 0)}, y={(0, 1)}, z={(0.352, 0.317)}, scale=3]
    \draw[fill=black!40] (0,0,0) -- (0,0,1) -- (0,1,1) -- (0,1,0) -- cycle;
    \draw[thin] (0,0,0) -- (1,0,0) -- (1,0,1) -- (0,0,1) -- cycle;
    \draw[thin] (1,1,1) -- (0,1,1) -- (0,0,1) -- (1,0,1) -- cycle;
    \draw[thin] (0,0,0) -- (1,0,0) -- (1,1,0) -- (0,1,0) -- cycle;
    \draw[thin] (1,1,1) -- (1,1,0) -- (1,0,0) -- (1,0,1) -- cycle;
    \draw[thin] (1,1,1) -- (0,1,1) -- (0,1,0) -- (1,1,0) -- cycle;
    \draw[thin] (0,0,0) -- (0,0,1) -- (0,1,1) -- (0,1,0) -- cycle;
    \draw[line width=0.3cm,black!70] (1,1,0) -- (1,1,1);
    \draw[fill=black!70,black!70] (1,1,1) circle[radius=0.05cm];
    \draw[fill=black!70,black!70] (1,1,0) circle[radius=0.05cm];
  \end{scope}
}
an upper face of codimension two and a face of codimension one are indicated.

\begin{definition}
  \label{def:cube-group}
  Let $G_\bullet$ be a filtered topological group, and let $k \ge 0$ be an integer.  The $k$-th Host--Kra cube group\index{Host--Kra cube group}, denoted $\HK^k(G_\bullet)$\index[nota]{$\HK^k(G_\bullet)$}, is a sub-group of $G^{\{0,1\}^k}$ defined as follows.

  For a subset $S \subseteq [k]$, let $F_S \subseteq \{0,1\}^k$\index[nota]{$F_S$} denote the upper face corresponding to $S$, i.e.~$F_S = \{ \omega \in \{0,1\}^k \colon \omega \supseteq S \}$.  For any $x \in G$ and any face $F$ of $\{0,1\}^k$, let $[x]_F$\index[nota]{$[x]_F$} donote the configuration $\{0,1\}^k \to G$ given by
  \[
  [x]_F(\omega) = \begin{cases} x &\colon \omega \in F , \\ \id &\colon \text{otherwise} \ . \end{cases}
  \]
  Then $\HK^k(G_\bullet)$ is the group generated by elements of the form $[x]_{F_S}$ where $S \subseteq \{0,1\}^n$ and $x \in G_{|S|}=G_{\codim(F_S)}$.
\end{definition}

We will illustrate this in the case $k=3$.  First, taking $S = \emptyset$ we are free to take any constant configuration:
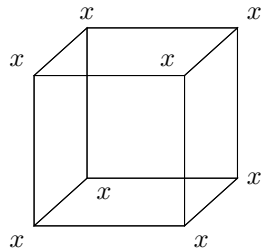
\inlinetikz{\threecube{$x$}{$x$}{$x$}{$x$}{$x$}{$x$}{$x$}{$x$}{1}}
for any $x \in G$.

Taking $S$ of size $1$ allows elements that are the identity on a lower face and equal to $x$ on the corresponding upper face for some fixed $x \in G_1$, e.g.:
\inlinetikz{\threecube{$\id$}{$\id$}{$\id$}{$\id$}{$x$}{$x$}{$x$}{$x$}{1}}

Similarly, for $|S| = 2$ we get a configuration equal to $x$ on some upper face of codimension $2$, and the identity elsewhere, where now $x \in G_2$ is fixed, e.g.:
\inlinetikz{\threecube{$\id$}{$\id$}{$\id$}{$\id$}{$\id$}{$\id$}{$x$}{$x$}{1}}

Finally, we can put any element of $G_3$ on the topmost vertex.
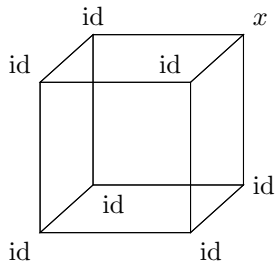
\inlinetikz{\threecube{$\id$}{$\id$}{$\id$}{$\id$}{$\id$}{$\id$}{$\id$}{$x$}{1}}

\begin{remark}
  At first glance, these definitions are inherently asymmetric in the sense of treating upper faces preferentially to other faces.  However, it is easy to see that $[x]_F$ is in $\HK^k(G_\bullet)$ for any face $F$ of codimension $r$, provided $x \in G_r$.

  For instance, if $S = \{i\}$ for some $i \in [k]$, and let $F$ denote the lower face $\{\omega_i = 0\}$.  Then
  \[
    [x]_F = [x^{-1}]_{F_S} [x]_{F_{\emptyset}}
  \]
  and this is clearly in $\HK^k(G_\bullet)$.  In general one can argue by induction on the co-dimension.
\end{remark}

Ultimately, though, we are not interested in nilpotent groups but in nilmanifolds $G/\Gamma$, their compact homogeneous spaces.  The notion of Host--Kra cubes over $G$ goes over directly to a notion on $G/\Gamma$.

\begin{definition}
  \label{defn:nilmanifold}
  Let $G_\bullet$ be a degree $s$ filtered Lie group.  Suppose for simplicity that $G$ is connected.\footnote{In \cites{GMV2,GMV3} there are good reasons to relax this assumption, but they do not apply for now.}  Also, let $\Gamma$ be a discrete and co-compact subgroup of $G$ (the latter meaning the quotient $G/\Gamma$ is compact).  Under these hypotheses, the quotient space $G/\Gamma$ (which need not be a group) is termed a \emph{nilmanifold}.\index{nilmanifold}

  Suppose furthermore that $\Gamma \cap G_i$ is discrete and co-compact in $G_i$ for each $i \ge 0$. If this property holds, we say that $\Gamma$ is {\em compatible}\index{compatible (with filtration)} with the filtration. Then for each $k \ge 0$, we define the Host--Kra cubes\index{Host--Kra cubes} $\HK^k(G_\bullet)/\Gamma$ to be the subset of $(G/\Gamma)^{\{0,1\}^k}$ given by the image of $\HK^k(G_\bullet)$ under the quotient map
  \begin{align*}
    \pi \colon G^{\{0,1\}^k} &\to (G/\Gamma)^{\{0,1\}^k} \\
              (g_\omega)_{\omega \in \{0,1\}^k}  &\mapsto (g_\omega \Gamma)_{\omega \in \{0,1\}^k} \ .
  \end{align*}
  Note that we abuse notation to let $\HK^k(G_\bullet)/\Gamma$ denote the \emph{pointwise} quotient by $\Gamma$ rather than a conventional quotient of groups: this notation is not meant to identify $\Gamma$ with a subgroup of $\HK^k(G_\bullet)$.  Equivalently, this quotient may be identified with $\HK^k(G_\bullet) / \left(\Gamma^{\{0,1\}^k}\cap\HK^k(G_\bullet) \right)$.
\end{definition}

The topological conditions are chosen to allow the following conclusion.

\begin{proposition}
  For $G_\bullet$ and $\Gamma$ as in the above definition, the space $\HK^k(G_\bullet)/\Gamma$ is a compact subset of $(G/\Gamma)^{\{0,1\}^k}$ for all $k \ge 0$.
\end{proposition}
See \cite{GT10}*{Lemma E.10} for a proof.

Finally, we summarize some properties of Host--Kra cubes on nilmanifolds that will be of significance in the next section.  The proofs of these properties appear in Appendix \ref{app:hk}, or follow easily from those results.

\begin{proposition}
  \label{prop:nilmanifold-props}
  Let $G_\bullet$ and $\Gamma \subseteq G$ be as in Definition \ref{defn:nilmanifold}, and suppose in particular $G_\bullet$ has degree $s$.  Then the following hold.
  \begin{enumerate}[label=(\roman*)]
    \item (Symmetries) The space $\HK^k(G_\bullet)/\Gamma$ is invariant under
a permutation of the $k$ coordinate axes, reflecting in a coordinate axis, or any combination of these.
    \item (Compatibility) If $c \in \HK^k(G_\bullet)/\Gamma$ then any face of $c$ of dimension $\ell$ (or more generally any ``subcube'', allowing some diagonal slicing) is an element of $\HK^\ell(G_\bullet)/\Gamma$.
    \item (Corner constraint) Suppose $c, c' \in \HK^{s+1}(G_\bullet)/\Gamma$ and $c(\omega) = c'(\omega)$ for all $\omega \ne \vec{1}$.  Then $c = c'$.
    \item (Corner completion) Suppose $\lambda \colon \{0,1\}^k \setminus \{\vec{1}\} \to G/\Gamma$ is a configuration such that for every lower face of $\{0,1\}^k$ of codimension $1$, the restriction of $\lambda$ to that face is in $\HK^{k-1}(G_\bullet)/\Gamma$.  Then $\lambda$ can be extended to an element of $\HK^k(G_\bullet)/\Gamma$.
  \end{enumerate}
\end{proposition}
Here, and throughout, $\vec{1}$\index[nota]{$\vec{1}$} denotes the element $(1,1,\dots,1) \in \{0,1\}^k$.
\begin{proof}
  Parts (i) and (ii) follow directly from the definition.  For (iii) and (iv), see Proposition \ref{prop:nilmanifold-uniqueness} and Proposition \ref{prop:nilmanifold-completion}.
\end{proof}

\section{Cubespaces and nilspaces}
\label{sec:cubespace-defns}

\subsection{Definitions}

We now give the formal definition and axioms of \emph{nilspaces}, and related notions.  In fact, the notion of a nilspace captures several distinct hypotheses of differing strength.  We will outline these individually, starting with the weakest, before amalgamating them into a final definition. We follow \cites{HK08, CS12} closely, although our terminology differs.

The very weakest structure we will wish to consider is termed a \emph{cubespace}.  Informally, this is just a topological space equipped with some notion of when $2^k$ points form a cube,\footnote{Following \cite{CS12}, we use the term ``cube'' throughout to refer to these distinguished collections of $2^k$ points.  These objects in fact seldom resemble geometric cubes, and the term ``parallelepiped'' used by  Host and Kra \cite{HK08} is more accurate.  However, ``cube'' has a significant advantage in brevity.} and satisfying several fairly basic conditions.

To define these conditions, we need to define some nomenclature for certain maps on the discrete cube $\{0,1\}^k$.

\begin{definition}
  A map $\rho \colon \{0,1\}^k \to \{0,1\}^\ell$ is called a \emph{morphism of discrete cubes}\index{morphism of discrete cubes} if it is the restriction to $\{0,1\}^k$ of an affine-linear map $\ZZ^k \to \ZZ^\ell$.

  Equivalently, this holds if and only if $\rho$ has the form
  \[
    \rho(x_1, \dots, x_k) = (\sigma_1(x_1, \dots, x_k), \dots, \sigma_\ell(x_1, \dots, x_k))
  \]
  where each function $\sigma_i$ is one of:
  \begin{itemize}
    \item identically $0$;
    \item identically $1$;
    \item equal to $x_j$ for some $j$;
    \item equal to $1 - x_j$ for some $j$.
  \end{itemize}
\end{definition}
It is straightforward to see that these definitions are indeed equivalent.

Informally, these morphisms of the discrete cube correspond to the fairly natural operations:
\begin{itemize}
  \item permute the coordinates of $\{0,1\}^k$;
  \item reflect in any coordinate axis;
  \item embed $\{0,1\}^k$ as a ``slice'' in $\{0,1\}^\ell$ for $\ell > k$;
  \item ``project'' $\{0,1\}^k$ onto $\{0,1\}^\ell$ for $\ell < k$ by deleting a coordinate;
\end{itemize}
and functions obtained from these by composition.

\begin{definition}
  A \emph{cubespace}\index{cubespace} is a metric space $X$, together with closed sets $C^k \subseteq X^{\{0,1\}^k}$ of \emph{$k$-cubes} for each $k \ge 0$, satisfying the following condition.  Suppose $\rho \colon \{0,1\}^k \to \{0,1\}^\ell$ is a morphism of discrete cubes and $c \colon \{0,1\}^\ell \to X$ is in $C^\ell$.  Then $c \circ \rho \colon \{0,1\}^k \to X$ is in $C^k$.

  Furthermore, it is always assumed that $C^0 = X$.
\end{definition}

We will typically abuse notation to allow $X$ to refer either to the underlying topological space, or to the full cubespace structure $(X, C^k)$.  We write $C^k(X)$\index[nota]{$C^k(X)$} in place of $C^k$ whenever there is ambiguity about which cubespace we are referring to.

Let us unpack what this means in terms of the elementary operations above.
\begin{itemize}
  \item If $c \in X^{\{0,1\}^k}$ is a $k$-cube, then permuting the $k$ coordinates gives another $k$-cube; e.g.:
    \inlinetikz
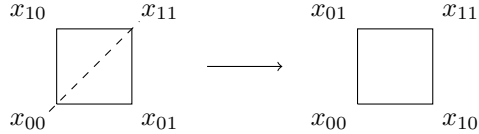
{
      \begin{scope}[shift={(-2,0)}]
        \singlesquare{$x_{00}$}{$x_{01}$}{$x_{10}$}{$x_{11}$}
        \draw[dashed] (-0.1, -0.1) -- (1.1, 1.1);
      \end{scope}
      \draw[->] (0,0.5) -- (1,0.5);
      \begin{scope}[shift={(2,0)}]
        \singlesquare{$x_{00}$}{$x_{10}$}{$x_{01}$}{$x_{11}$}
      \end{scope}
    }
  \item Similarly, reflecting in any coordinate axis gives another $k$-cube; e.g.:
    \inlinetikz
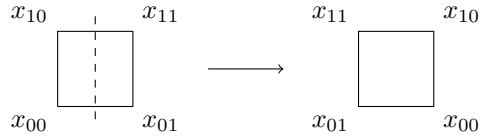
{
      \begin{scope}[shift={(-2,0)}]
        \singlesquare{$x_{00}$}{$x_{01}$}{$x_{10}$}{$x_{11}$}
        \draw[dashed] (0.5, 1.2) -- (0.5, -0.2);
      \end{scope}
      \draw[->] (0,0.5) -- (1,0.5);
      \begin{scope}[shift={(2,0)}]
        \singlesquare{$x_{01}$}{$x_{00}$}{$x_{11}$}{$x_{10}$}
      \end{scope}
    }
  \item If $c \in X^{\{0,1\}^k}$ is a $k$-cube, then restricting to an $\ell$-dimensional ``slice'' gives an $\ell$-cube:
    \inlinetikz
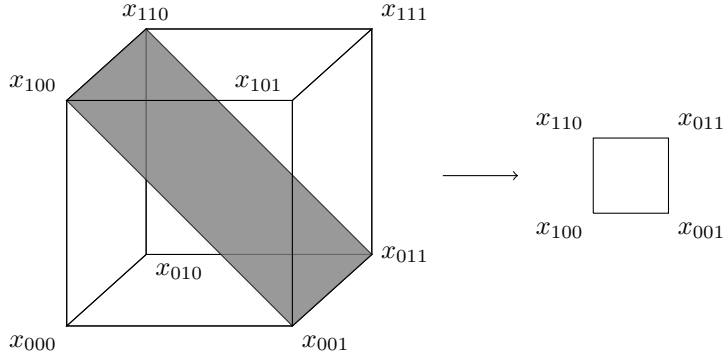
{
      \begin{scope}[shift={(-5,0)}]
        \begin{scope}[x={(1, 0)}, y={(0, 1)}, z={(0.352, 0.317)}, scale=3]
          \draw (0,0,0) node[below left] {$x_{000}$} -- (0,0,1) node[below right] {$x_{010}$} -- (0,1,1) node[above] {$x_{110}$} -- (0,1,0) node[above left] {$x_{100}$} -- cycle;
          \draw (0,0,0) -- (1,0,0) node[below right] {$x_{001}$} -- (1,0,1) node[right] {$x_{011}$} -- (0,0,1) -- cycle;
          \draw (1,1,1) node[above right] {$x_{111}$} -- (0,1,1) -- (0,0,1) -- (1,0,1) -- cycle;
          \draw (1,1,1) -- (1,1,0) -- (1,0,0) -- (1,0,1) -- cycle;
          \draw (1,1,1) -- (0,1,1) -- (0,1,0) -- (1,1,0) -- cycle;
          \draw[fill=black!50, opacity=.8] (0,1,0) -- (0,1,1) -- (1,0,1) -- (1,0,0) -- cycle;
          \draw (0,0,0) -- (1,0,0) -- (1,1,0) node [above left] {$x_{101}$} -- (0,1,0) -- cycle;
        \end{scope}
      \end{scope}
      \draw[->] (0,2) -- (1,2);
      \begin{scope}[shift={(2,1.5)}]
        \singlesquare{$x_{100}$}{$x_{001}$}{$x_{110}$}{$x_{011}$}
      \end{scope}
    }
  \item If $c \in X^{\{0,1\}^k}$ is a $k$-cube, the configuration obtained by placing two copies of $c$ adjacent to each other is a $(k+1)$-cube, and so on:
    \inlinetikz
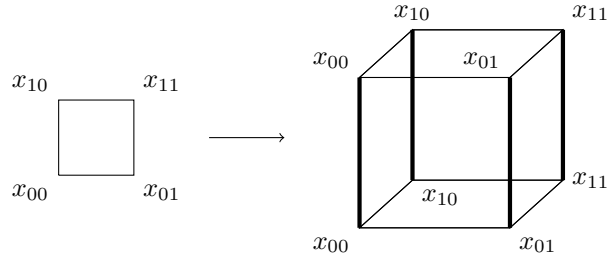
{
      \begin{scope}[shift={(-2,0.7)}]
        \singlesquare{$x_{00}$}{$x_{01}$}{$x_{10}$}{$x_{11}$}
      \end{scope}
      \draw[->] (0,1.2) -- (1,1.2);
      \begin{scope}[shift={(2,0)}]
        \threecube{$x_{00}$}{$x_{01}$}{$x_{10}$}{$x_{11}$}{$x_{00}$}{$x_{01}$}{$x_{10}$}{$x_{11}$}{1}
        \begin{scope}[x={(1, 0)}, y={(0, 1)}, z={(0.352, 0.317)}, scale=2]
          \draw[ultra thick] (0,0,0) -- (0,1,0);
          \draw[ultra thick] (1,0,0) -- (1,1,0);
          \draw[ultra thick] (0,0,1) -- (0,1,1);
          \draw[ultra thick] (1,0,1) -- (1,1,1);
        \end{scope}
      \end{scope}
    }
\end{itemize}

\begin{remark}
  \label{rem:constant-cube}
  Note that applying this last point repeatedly, we deduce that any constant configuration (i.e.~$c(\omega) = x$ for all $\omega$) is automatically a cube.
\end{remark}

It is intuitively fairly reasonable that any well-behaved notion of ``cube'' analogous to parallelepipeds or the Host--Kra construction on nilmanifolds should at the very least obey these properties.

We now consider the ``corner constraint'' discussed previously.  Recall we said informally that, for some specific dimension of cube, all but one of the vertices should determine the last one.

\begin{definition}
  We say a cubespace $X$ has \emph{$k$-uniqueness}\index{$k$-uniqueness} if the following holds: whenever $c, c' \in C^k(X)$ and $c(\omega) = c'(\omega)$ for all $\omega \in \{0,1\}^k \setminus \{\vec{1}\}$ then $c = c'$.
\end{definition}

The ``dual'' of this property is equally important but harder to motivate: if all but one vertex of a $k$-cube is specified in a consistent way, then there is \emph{at least one} way to complete the missing vertex to give a cube.

\begin{definition}
  \label{defn:k-completion}
  We say a cubespace $X$ has \emph{$k$-completion}\index{$k$-completion} if the following holds.  Suppose $\lambda \colon \{0,1\}^k \setminus \{\vec{1}\} \to X$ has the property that every ``lower face'' is a $(k-1)$-cube, i.e.~for each $1 \le i \le k$ the map
  \begin{align*}
    \{0,1\}^{k-1} &\to X \\
    (\omega_1, \dots, \omega_{k-1}) &\mapsto \lambda(\omega_1, \dots, \omega_{i-1}, 0, \omega_i, \dots, \omega_{k-1})
  \end{align*}
  is in $C^{k-1}(X)$.  Then there exists some $x \in X$ such that
  \begin{align*}
    c \colon \{0,1\}^k &\to X \\
                  \omega &\mapsto \begin{cases} \lambda(\omega) &\colon \omega \ne \vec{1} \\ x &\colon \omega = \vec{1} \end{cases}
  \end{align*}
  is in $C^k(X)$.
\end{definition}
\inlinetikz
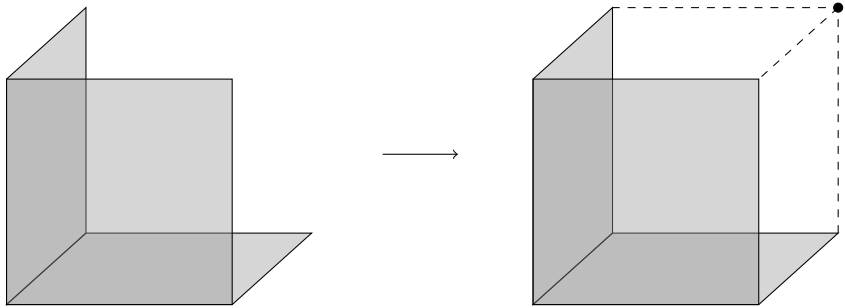
{
  \begin{scope}{scale=0.67}
    \begin{scope}[shift={(-5,0)}]
      \begin{scope}[x={(1, 0)}, y={(0, 1)}, z={(0.352, 0.317)}, scale=3]
        \draw[fill=black!40,fill opacity=.4] (0,0,0) -- (0,0,1) -- (0,1,1) -- (0,1,0) -- cycle;
        \draw[fill=black!40,fill opacity=.4]  (0,0,0) -- (1,0,0) -- (1,0,1) -- (0,0,1) -- cycle;
        \draw[fill=black!40,fill opacity=.4]  (0,0,0) -- (1,0,0) -- (1,1,0) -- (0,1,0) -- cycle;
      \end{scope}
    \end{scope}
    \draw[->] (0, 2) -- (1, 2);
    \begin{scope}[shift={(2,0)}]
      \begin{scope}[x={(1, 0)}, y={(0, 1)}, z={(0.352, 0.317)}, scale=3]
        \draw[fill=black!40,fill opacity=.4] (0,0,0) -- (0,0,1) -- (0,1,1) -- (0,1,0) -- cycle;
        \draw[fill=black!40,fill opacity=.4]  (0,0,0) -- (1,0,0) -- (1,0,1) -- (0,0,1) -- cycle;
        \draw[fill=black!40,fill opacity=.4]  (0,0,0) -- (1,0,0) -- (1,1,0) -- (0,1,0) -- cycle;
        \draw[dashed] (1,0,1) -- (1,1,1) -- (0,1,1);
        \draw[dashed] (1,1,1) -- (1,1,0);
        \draw[fill=black] (1,1,1) circle[radius=0.02cm];
      \end{scope}
    \end{scope}
  \end{scope}
}

We call a configuration of the form $\lambda$ (with the same properties) a \emph{$k$-corner}.\index{$k$-corner}

While uniqueness will typically be specified for one particular value of $k$, a well-behaved space will have $k$-completion for \emph{all} $k \ge 0$.

\begin{definition}
  We say a cubespace $X$ is a \emph{nilspace of degree $s$} if $s \ge 0$ is the smallest nonnegative integer such that $X$ has $(s+1)$-uniqueness, and if $X$ has $k$-completion for all $k$.  We say it is simply a \emph{nilspace}\index{nilspace} if it is a nilspace of degree $s$ for some $s$.
\end{definition}

\begin{remark}
One way of motivating the $k$-completion hypothesis is that it guarantees that $X$ is not missing any points, in the following sense.

Suppose $X$ is a compact nilspace of degree $s$, and $S \subseteq X$ an arbitrary closed subset.  We can make $S$ into a cubespace by taking $C^k(S) = C^k(X) \cap S^{\{0,1\}^k}$, i.e.~just taking those cubes of $X$ whose vertices lie in $S$.
We call this cubespace the subcubespace of $X$ {\em induced}\index{induced subcubespace} by $S$.

  With this cubespace structure, it is easy to verify that $(S, C^k(S))$ is still a compact cubespace with $(s+1)$-uniqueness; i.e., it satisfies all the hypotheses of being a (compact) nilspace of degree $s$ except for the $k$-completion one.  However, $(S, C^k(S))$ will typically \emph{not} have $k$-completion for any $k > 1$, unless $S$ takes a very particular form.

  Similarly, given $r \ge 1$ and a closed subset $S' \subseteq C^r(X)$ (obeying some mild conditions) we could form an induced cubespace structure on $X$ with $C^r = S'$ that will have the same uniqueness properties but typically not have $k$-completion for all $k$.

  So, if we were to omit the $k$-completion axiom from the definition of a nilspace, it would become fairly hopeless to ask for a rigid classification theorem for nilspaces, since there are a huge variety of examples of this type.  By contrast, if we do impose $k$-completion, we will see that this rules out such examples in general and imposes significantly more rigidity on the class of nilspaces.
  \footnote{It is natural to ask what less rigid classification one could nonetheless hope for if the $k$-completion condition is omitted.  A natural categorial formulation is whether there is such a thing as canonical \emph{nilspace completion}: given a compact cubespace $X$ with $(s+1)$-uniqueness, is there a degree $s$ nilspace $Y$ and a cubespace morphism (see below) $i \colon X \to Y$ with the property that every cubespace morphism $X \to Y'$, where $Y$ is any nilspace, factors uniquely through $Y$ (and if so, what properties does $i$ have).  In general the answer may be negative, a troubling example being a set $\{x \in \FF_3^n \colon P(x) = 0\}$ where $P$ is a quadratic polynomial.  It seems interesting questions remain here, although the authors are not sure how exactly they should be stated.}
\end{remark}

We will be interested not only in cubespaces and nilspaces in isolation, but also in maps between them.  Informally, a map $X \to Y$ preserves the relevant structures if it takes any cube of $X$ to a cube of $Y$.

\begin{definition}
  Suppose $X$, $Y$ are cubespaces with underlying spaces $X$, $Y$ respectively.  Let $\phi \colon X \to Y$ be a continuous map.  We say $\phi$ is a \emph{cubespace morphism}\index{cubespace!morphism} or just \emph{morphism} if $\phi(C^k(X)) \subseteq C^k(Y)$ for all $k \ge 0$.
\end{definition}

We will need one further technical definition.
\begin{definition}
  A cubespace $X$ is called \emph{ergodic}\index{cubespace!ergodic} if $C^1(X) = X^2$, i.e.~if every pair $(x,x')$ is a $1$-cube.  Furthermore we say it is $k$-ergodic if $C^k(X) = X^{\{0,1\}^k}$, i.e.~if every configuration is a $k$-cube.
\end{definition}

Note that our terminology differs slightly from that in \cite{CS12}, in that we do not insist that a nilspace be ergodic: this is built into the definition of a nilspace in that work.  As a consequence, the phrase ``ergodic nilspace'' will occur frequently in our statements.

It can be shown (although we will not need to do so) that a suitably nice non-ergodic space -- in particular, a non-ergodic nilspace -- decomposes as a disjoint union of ergodic components, whose cubespace structures essentially do not interact at all.  Hence one loses almost no generality by working in the ergodic setting.

Also, we remark that \cite{CS12} has the separate notions of ``abstract'' nilspaces, which have no topology, and ``compact'' ones, which have a topology but are always assumed to be compact.  By contrast, we assume all nilspaces have a topology (though it could be the discrete topology), but do not insist they be compact.  This discrepancy has no particular significance, except that it allows us to consider e.g.~nilpotent Lie groups as nilspaces with their natural topologies.  Again, the phrase ``compact, (ergodic) nilspace'' will occur often in what follows.

\subsection{Host--Kra nilmanifolds are nilspaces}

Having motivated these nilspace axioms (and other definitions) in terms of Host--Kra cubespaces and nilmanifolds, it is reasonable to want to check that the latter are actually instances of the former.

We state some facts of this nature now.  Most aspects of the proofs are straightforward; some, surprisingly, are much less so, and these are deferred to Appendix \ref{app:hk} which expounds the theory of Host--Kra cubespaces more fully than we have done so far.

\begin{proposition}
  \label{prop:hk-cubespace-properties}
  Let $G/\Gamma$ be a (filtered) nilmanifold of degree $s$, in the sense of the previous section, equipped with its Host--Kra cubes.  Then it is a nilspace of degree $s$.

  If the filtration $G_\bullet$ is \emph{proper} (i.e.~$G_0 = G_1$) then $G/\Gamma$ is ergodic.

  If $G'/\Gamma'$ is another nilmanifold, and $\psi \colon G \to G'$ a group homomorphism such that $\psi(G_i) \subseteq G'_i$ for all $i$ and $\psi(\Gamma) \subseteq \Gamma'$, then the induced map $\phi \colon G/\Gamma \to G'/\Gamma'$ is a cubespace morphism.

  Similarly, for any $x \in G_0$, $g \in G_1$, the map $n \mapsto g^n x\, \Gamma$ is a cubespace morphism $\ZZ \to G/\Gamma$.
\end{proposition}

Here, as throughout these papers, $\ZZ$ is considered as a filtered group with filtration $\ZZ_0 = \ZZ_1 \supseteq \{0\}$, and as a cubespace by setting $C^k(\ZZ) = \HK^k(\ZZ_\bullet)$.

\begin{proof}
  The fact that $G/\Gamma$ is a nilspace is essentially a restatement of Proposition \ref{prop:nilmanifold-props}.

  The ergodicity statement is immediate from the definition of $\HK^1(G_\bullet)$.

  The fact that a group homomorphism $\psi \colon G \to G'$ taking $G_i$ into $G_i'$ maps $\HK^k(G_\bullet)$ into $\HK^k(G_\bullet')$ is straightforward by considering the images of the generators.  Recovering the corresponding statement for the quotients $G/\Gamma$, $G'/\Gamma'$ is routine.

  Finally, we note that for any $x \in G_0$, $g \in G_1$, $a \in \ZZ$ and $h \in \ZZ^k$, the configuration
  \begin{align*}
    \{0,1\}^k &\to G \\
    \omega &\mapsto g^{a + \omega \cdot h} x
  \end{align*}
  is in $\HK^k(G_\bullet)$, since it is a product of $(k+1)$ generators:
\[
[g^{h_1}]_{F_{\{1\}}}\cdots [g^{h_k}]_{F_{\{k\}}}\cdot[g^ax]_{F_\varnothing},
\]
and the last part follows.
\end{proof}

We have been light on detail in this proof: partly because this result is not logically necessary in what follows.  However, working through the details is an excellent exercise for the reader unfamiliar with either these definitions or the machinery of Host--Kra cubes (or both).

\subsection{High-dimensional cubes}

One slightly unsatisfactory feature of our definitions is that the data of a cubespace involves spaces $C^k(X)$ for infinitely many $k \ge 0$.  We are used to considering cubes or parallelepipeds of only bounded dimension in any given problem (e.g.~of dimension $3$ when considering the $U^3$ norm).

In fact, under good hypotheses we can see that for sufficiently large $k$ the data of $C^k(X)$ contains no new information.

As one further reassuring sanity-check, and to see these definitions in action, we state this now.

\begin{proposition}
  \label{prop:high-cubes-boring}
  Let $s \ge 1$ be an integer, $X$ a nilspace of degree $s$, and $k \ge s+1$.  Then a configuration $c \colon \{0,1\}^k \to X$ is in $C^k(X)$, if and only if every face of $c$ of dimension $(s+1)$ is in $C^{s+1}(X)$.
\end{proposition}
\begin{proof}
  The ``only if'' direction is direct from the cubespace axioms.  For the ``if'' direction, we argue by induction on $k$, the base case $k=s+1$ being trivial.

  Given such a $c \colon \{0,1\}^k \to X$, consider the restriction $c|_{\{0,1\}^k \setminus \{\vec{1}\}}$.  By inductive hypothesis it is a $k$-corner.  Hence, there is a cube $c' \in C^k(X)$ such that $c(\omega) = c'(\omega)$ for all $\omega \ne \vec{1}$.  But now, considering any upper face $c'$ of dimension $(s+1)$ and the corresponding one of $c$, and invoking $(s+1)$-uniqueness, we see that $c(\vec{1}) = c'(\vec{1})$ and hence $c$ is a $k$-cube.
\end{proof}

\section{The structure theorem}
\label{sec:structural-thm}

\subsection{Statement}

We are now in a position to state the main structural result.
This result is essentially due to Antol\'\i n Camarena and Szegedy \cite{CS12}*{Theorems 4 and 7};
however, our formulation is somewhat stronger.
We point out the differences below.

\begin{theorem}
  \label{main-structural-v1}
  Let $X = (X, C^k)$ be a compact ergodic nilspace of degree $s$.  Further suppose that the spaces $C^k(X)$ are connected for each $k \ge 0$; so in particular $X = C^0(X)$ is connected.

  Then $X$ is an ``inverse limit of nilmanifolds'' in the following sense. There exists:
  \begin{itemize}
    \item a sequence $G^{(n)}$ of connected Lie groups equipped with filtrations $G^{(n)}_\bullet$ of degree at most $s$ (with $G^{(n)}_i$ also connected for each $i$);
    \item discrete co-compact subgroups $\Gamma^{(n)}$ of $G^{(n)}$ such that $\Gamma^{(n)} \cap G^{(n)}_i$ is discrete and co-compact in $G^{(n)}_i$ for each $i$; and
    \item surjective group homomorphisms $\phi_{n,m} \colon G^{(n)} \to G^{(m)}$ for each $n \ge m$, such that $\phi_{n,m}\left(G^{(n)}_i\right) = G^{(m)}_i$ for each $i \ge 0$ and also $\phi_{n,m}\left(\Gamma^{(n)}\right) \subseteq \Gamma^{(m)}$, so that we get a map of nilmanifolds $G^{(n)}/\Gamma^{(n)} \to G^{(m)}/\Gamma^{(m)}$ that sends cubes (surjectively) to cubes;
  \end{itemize}
  such that
  \begin{itemize}
    \item $X = \varprojlim \left(G^{(n)}/\Gamma^{(n)}\right)$ as a topological space, and
    \item the cubes $C^k(X)$ coincide with the inverse limit of Host--Kra cubes; i.e., writing $\pi_n \colon X \to \left(G^{(n)}/\Gamma^{(n)}\right)$ for the projection map arising from the above inverse limit, we have
      \[
        C^k(X) = \bigcap_n \pi_n^{-1} \left(C^k\left(G^{(n)} / \Gamma^{(n)}\right)\right)  \ .
      \]
  \end{itemize}
\end{theorem}

\begin{definition}
  We call a cubespace satisfying the hypothesis that $C^k(X)$ is connected for all $k$ \emph{strongly connected}\index{cubespace!strongly connected}.
\end{definition}

\begin{remark}
  \label{rem:inverse-limit}
  The last two conditions, together with the fact that $\phi_{n,m}$ are cubespace morphisms between nilmanifolds (see Proposition \ref{prop:hk-cubespace-properties}), say in some sense that $X$ is an ``inverse limit of nilmanifolds, in the category of cubespaces''.
  This is roughly the formulation one obtains from \cite{CS12}, combining \cite{CS12}*{Theorems 4 and 7}. (In fact the results from \cite{CS12} give slightly more detail than that, specifically that the maps $\phi_{n,m}$ are ``fiber-surjective'', or in our terminology ``fibrations''; see Definition \ref{def:fibration} and Remark \ref{rem:fiber-surjective}.  This same condition comes out of \cite{GMV3}*{Theorem 1.26}, which is equivalent to \cite{CS12}*{Theorem 4}.)

  However, the conclusion of Theorem \ref{main-structural-v1} is stronger than this, in that it gives even more information about the maps $\phi_{n,m}$: namely, that they come from group homomorphisms $G^{(n)} \to G^{(m)}$, and moreover that these group homomorphisms are surjective on the filtrations in the obvious sense.  Neither of these properties holds in general for a cubespace morphism between nilmanifolds.
\end{remark}

\begin{remark}\label{rem:connectivity}
  Without some similar connectivity hypothesis, no complete structure theorem is currently available in general.

  Using Theorem \ref{thm:invlim} (see below) or equivalently \cite{CS12}*{Theorem 4}, one can get some partial information by identifying a compact nilspace with an inverse limit of ``finite rank'' nilspaces, which are in particular e.g.~topological manifolds of finite dimension.  However, the structure of such spaces remains unclear: they need not necessarily be nilmanifolds.  Some structural results in the special case of finite nilspaces appear in \cite{szegedy2010structure}.

  The analogous result to Theorem \ref{main-structural-v1} implied by \cite{CS12} has this strong connectivity hypothesis replaced by an assumption on the ``structure groups'' of $X$ (see below).  That assumption easily implies strong connectivity (see \cite{GMV2}*{Proposition 2.4}), 
but for the converse, we need the full force of our structure theorem.
\end{remark}

For clarity, we summarize the dictionary between statements in these papers and those in \cite{CS12} in Section \ref{sec:comparison} below, once we have introduced the relevant concepts.

\begin{remark}
  It is reasonably straightforward to argue that any cubespace $X$ that is an inverse limit of nilmanifolds, subject to some further technical hypotheses on this inverse limit, satisfies the hypotheses of Theorem \ref{main-structural-v1}.  Hence in some sense this result is best possible.

  One can certainly not remove the inverse limit, as can be seen by considering any compact connected abelian group that is not a Lie group: such an object is a nilspace of degree $1$ but is not a degree $1$ nilmanifold, as the latter are all just tori.

  It is also possible to find an example of a nilspace satisfying the hypotheses for $s=2$, which is not a homogeneous space of any group.  Such an example is due to Rudolph \cite{R95}.  It seems likely that a sharper description of such spaces is possible, but we will not pursue such issues here.
\end{remark}

\subsection{The case $s=1$}\label{sc:1step}

To provide some motivation for this structural result, we briefly comment on the case $s=1$.  In this setting, Theorem \ref{main-structural-v1} is much easier.

\begin{proposition}
  A compact ergodic nilspace $X$ of degree $1$ naturally has the structure of a compact abelian group.
\end{proposition}
\begin{proof}[Proof sketch]
  We give only a partial account of the proof, since these ideas are covered in detail in Section \ref{sec:structure-groups}.

  Essentially our task is to use the data we have -- namely, unique corner-completion of $2$-cubes -- to recover a group operation on $X$.

  Fix any element $e \in X$, which will be our identity element.  Given $x, y \in X$ consider the $2$-corner:
  \inlinetikz{
    \draw[thin,black] (0,0) node[below left] {$e$} -- (1, 0) node[below right] {$x$};
    \draw[thin,black] (0,0) -- (0, 1) node[above left] {$y$};
    \draw[dashed] (1,0) -- (1,1) node[above right] {$\ast$};
    \draw[dashed] (0,1) -- (1,1);
  }
  and define $x+y$ to be the element obtained by completing it, i.e.~$x+y=\ast$.  We can define inverses by a similar process, and define the identity to be $e$.

  It now suffices to check that this defines a topological abelian group operation.  Many of the corresponding properties are immediate from the nilspace axioms.  The hardest part is checking associativity, which requires the completion property on $3$-dimensional cubes; we omit this part.
\end{proof}

The remaining content of Theorem \ref{main-structural-v1} in this case is that any connected compact abelian group is an inverse limit of tori, which is a classical result.

\subsection{Steps in the proof}

The proof of Theorem \ref{main-structural-v1} splits into three parts.

The first is a weaker, fairly elementary structure theorem.  Roughly speaking, this says that a nilspace of degree $s$ has the structure of a tower of extensions $X \to X_{s-1} \to \dots \to X_0 = \{\ast\}$ where each of the fibers is a compact abelian group.  We refer to this as the \emph{weak structure theorem}.

The second says that, under some strong technical assumptions, a compact ergodic strongly connected nilspace of degree $s$ actually \emph{is} a nilmanifold $G/\Gamma$ with its usual Host--Kra cubes (and a degree $s$ filtration).  These technical assumptions are phrased in terms of the compact abelian groups appearing in the weak structure theorem; specifically, that they are (abelian) Lie groups, i.e.~$(\RR/\ZZ)^d \times K$ for some $d \ge 0$ and $K$ finite.  This is the step where a non-abelian group operation is recovered from just the cubespace structure, the key step being to consider the automorphism group of the cubespace $X$, and to show that it is large enough in a certain sense.

The final step requires us to show that a cubespace satisfying the hypotheses of Theorem \ref{main-structural-v1} can be written as an inverse limit of ones which satisfy the strong technical assumptions alluded to in part two.  Given the previous discussion, this is very closely related to the statement that a connected compact abelian group is an inverse limit of tori, and this fact is a key ingredient in the proof.  However, bolting everything together is one of the more technically difficult aspects of the whole argument.

Only the first part of this plan will be proven fully in this paper, although we will give rather fuller outlines of the latter two below.  These are proven fully in two further papers by the authors, the second part in \cite{GMV2} and the third in \cite{GMV3}.

\section{A more detailed summary of the argument}
\label{sec:remaining-summary-v2}

We will now expand on the outline given above of each of the three stages in the proof of Theorem \ref{main-structural-v1}.

\subsection{An outline of the weak structure theory}\label{sc:outline-weak}

The simplest examples of nilspaces are compact abelian groups $A$ together with their Host--Kra cubes (with the degree $1$ filtration $A = A_0 = A_1 \supseteq A_2 = \{0\}$).  The $k$-cubes of these spaces are precisely the $k$-dimensional parallelepipeds on $A$ in the usual sense:
\begin{align*}
  \{0,1\}^k &\to A \\
  \omega &\mapsto x_0 + \sum_{i=1}^k \omega_i x_i
\end{align*}
where $x_0, \dots, x_k \in A$.

The simplest higher-degree examples come from the Host--Kra construction on $A$ (still a compact abelian group) given the degree $s$ filtration
\[
  A = A_0 = A_1 = \dots = A_s \supseteq A_{s+1} = \{0\}
\]
where cubes are given by $C^k(A) = \HK^k(A_\bullet)$.  We denote this cubespace by $\cD_s(A)$\index[nota]{$\cD_s(A)$}.  It is easy to check (or to deduce from general considerations) that $\cD_s(A)$ is a compact, ergodic (and in fact $s$-ergodic) nilspace of degree $s$.

In fact there are several equivalent ways to define $\cD_s(A)$.  We briefly note some equivalent characterizations.

\begin{proposition}\label{pr:Ds}
  If $c \colon \{0,1\}^{s+1} \to A$ is a configuration, then $c \in C^{s+1}(\cD_s(A))$ if and only if
  \[
    \sum_{\omega \in \{0,1\}^{s+1}} (-1)^{|\omega|} c(\omega) = 0 \ .
  \]
  Moreover, for any $k \ge 0$ and $c \colon \{0,1\}^k \to A$, we have $c \in C^k(\cD_s(A))$ if and only if every face of $c$ of dimension $(s+1)$ is a cube; or, if and only if $c \circ \eta \in C^{s+1}(\cD_s(A))$ for every morphism of discrete cubes $\eta \colon \{0,1\}^{s+1} \to \{0,1\}^k$.
\end{proposition}
\begin{proof}
  The first claim is proved in \ref{ex:Ds}.

  It is clear from the cubespace axioms that if $c \in C^k(\cD_s(A))$ and $\eta \colon \{0,1\}^{s+1} \to \{0,1\}^k$ is a morphism of discrete cubes then $c \circ \eta \in C^{s+1}(\cD_s(A))$, and this includes the case of faces of $c$.

  Conversely, by Proposition \ref{prop:nilmanifold-props}, we know that $\cD_s(A)$ is a nilspace of degree $s$. So if $c \colon \{0,1\}^k \to A$ and every face of $c$ of dimension $(s+1)$ is a cube, then $c$ is a cube by Proposition \ref{prop:high-cubes-boring}.
\end{proof}

The weak structure theorem states that, while these are certainly not the only examples of nilspaces, any (compact, ergodic) nilspace $X$ of degree $s$ can be \emph{built up} from spaces of the form $\cD_k(A_k)$ for some compact abelian groups $A_k$ for $1 \le k \le s$.
Indeed, $X$  can be realized as a tower of extensions
\[
X\to \pi_{s-1}(X)\to\pi_{s-2}(X)\to\cdots \to \pi_0(X)=\{*\}.
\]
such that $\pi_{i}(X)$ is a nilspace of degree $i$ and the fibers of the map $\pi_{k}(X)\to \pi_{k-1}(X)$
admit a free transitive action of a compact abelian group $A_{k}$ through which the cubespace structure on the
fiber can be identified by $\cD_k(A_k)$.

Before we can formulate these results more precisely we need to discuss how we can put a cubespace structure on
a quotient of a cubespace.

\begin{definition}\label{df:quotient}
  Let $(X, C^k(X))$ be a compact cubespace, and suppose $\sim$ is a closed equivalence relation on $X$.  The quotient space $X / \sim$ (with the quotient topology) is considered to be a cubespace\index{cubespace!quotient}, by setting $C^k(X/\sim)$ to be $C^k(X) / \sim$, i.e.~the image of $C^k(X)$ under pointwise application of $\sim$.
In other words this means that a configuration in $X/\sim$ is a cube if and only if there is some representative of it that is a cube in $X$.
\end{definition}

It is straightforward from the definition that this indeed defines a cubespace structure.
In the next definition, we specify the equivalence relation that gives rise to the quotients
$\pi_i(X)$.
This may be compared with \cite{CS12}*{Definition 2.3}.

\begin{definition}\label{df:canonical-rel}
  For any $k \ge 0$, we define $\sim_k$\index[nota]{$\sim_k$} on a (compact, ergodic) cubespace $X$ with $n$-completion for all $n$, by $x \sim_k y$ if and only if there exist $c, c' \in C^{k+1}(X)$ such that $c(\vec{1}) = x$, $c'(\vec{1}) = y$ and $c(\omega) = c'(\omega)$ for all $\omega \ne \vec{1}$.
We call $\sim_k$ the $k$-th canonical equivalence relation.\index{canonical equivalence relation}
We denote by $\pi_k$\index[nota]{$\pi_k$} the quotient map $X\to X/\sim_k$ and call it the $k$-th canonical projection\index{canonical projection}.
\end{definition}

If  $\sim$ is an equivalence relation such that the factor $X/\sim$ has $(s+1)$-uniqueness, then it must contain
$\sim_s$, which is the motivation for the definition.
However, it requires proof that $\sim_s$ is indeed an equivalence relation and the quotient is indeed a nilspace.
We will discuss these facts and further properties of $\sim_s$ in Section \ref{sec:canonical-factors}.

We are now ready to state the weak structure theorem, which is proved in Sections \ref{sec:weak-overview} and \ref{sec:elementary-v2}.
See also \cite{CS12}*{Theorem 1}, and also \cite{HK08}*{Section 5} for a related discussion.

\begin{theorem}[Weak Structure Theorem]
  \label{basic-cs-structure-theorem}
  Let $X$ be a compact ergodic nilspace of degree $s$ for some $s \ge 1$.
Then there exists a compact abelian group $A=A_s(X)$, the ``$s$-th structure group'' of $X$,
acting continuously on $X$, such that:
  \begin{enumerate}
    \item the action of $A$  is free on $X$, and its orbits are precisely the fibers of $X\to\pi_{s-1}(X)$;
    \item this induces a pointwise action of $C^k(\cD_s(A))$ on $C^k(X)$: again this action is free, and its orbits are precisely the fibers of the induced map $C^k(X) \to C^k(\pi_{s-1}(X))$.
  \end{enumerate}
\end{theorem}

We define the $k$-th structure group\index{structure group} $A_k(X)$\index[nota]{$A_k(X)$} of $X$ as the group $A_k(\pi_{k}(X))$ that
arises when we apply this theorem to the nilspace $\pi_{k}(X)$ (with $s=k$).
We will see in the proof that these groups are defined canonically in terms of the cubespace structure.

\begin{remark}\label{rem:cases}
We note that the weak structure theorem allows the following partial reconstruction of cubes on $X$.
Let $c\in C^k(\pi_{s-1}(X))$ be a cube.
If we know that $\wt c\in C^k(X)$ is a cube that projects to $c$ (that is $\pi_{s-1}(\wt c)=c$), then
we can find all cubes projecting to $\wt c$.
Indeed, let $c':\{0,1\}^k\to X$ be a configuration with $\pi_{s-1}(c')=c$, and write $a:\{0,1\}^k\to A_s(X)$
for the unique configuration such that $a(\omega).\wt c(\omega)=c'(\omega)$.
Then $c'$ is a cube if and only if $a\in C^k( \cD_s(A_s))$.
However, the weak structure theorem provides no insight for finding the first lift $\wt c$ of a cube in $C^k(\pi_{s-1}(X))$.
\end{remark}

\begin{remark}
We  note that no topological input is really used in this theorem, except to show topological conclusions.  For instance, if $X$ is a possibly infinite \emph{discrete} nilspace (i.e.~ignoring topology), the same proofs will apply, resulting in a discrete abelian group $A$ acting on the space and so on.

  We will not consider the result in such generality, as we have no need for such statements in applications.
\end{remark}

\subsection{Automorphisms and recovering a nilmanifold structure}

Stage II in our plan for proving Theorem \ref{main-structural-v1} was to establish a strong structural result, stating that $X$ actually is a nilmanifold $G/\Gamma$, under some additional technical assumptions.  We state this result, which is a variant of \cite{CS12}*{Theorem 7}, now.

\begin{theorem}
  \label{toral-structure-theorem}
  Suppose $X$ is a compact, ergodic, strongly connected nilspace of degree $s$.  Further suppose that the structure groups $A_t$ for $1 \le t \le s$ are all (compact abelian) Lie groups, i.e.~isomorphic to tori $(\RR/\ZZ)^d \times K$ for some $d=d(A_t)\ge 0$ and some finite $K$.

  Then $X$ is isomorphic to a nilmanifold $G/\Gamma$, equipped with its Host--Kra cubes.  That is, there exists a connected Lie group $G$ equipped with a filtration of degree $s$ (and such that $G_i$ is also connected for each $i$), and a discrete and co-compact subgroup $\Gamma$ of $G$ compatible with the filtration (cf. Definition \ref{defn:nilmanifold}), such that
  \begin{itemize}
    \item $X$ is homeomorphic to $G/\Gamma$, and
    \item under this homeomorphism, $C^k(X)$ is identified with $\HK^k(G_\bullet)/\Gamma$.
  \end{itemize}
\end{theorem}

A complete proof of this theorem can be found in another paper of the authors \cite{GMV2}*{Theorem 2.18}. 

\begin{remark}
  \label{rem:strong-connectivity}
The difference between Theorem \ref{toral-structure-theorem} and \cite{CS12}*{Theorem 7} is that the latter makes the assumption that the structure groups are tori $\RR^d/\ZZ^d$ -- i.e.~that the finite group $K$ is necessarily trivial, or equivalently that $A_k(X)$ is connected -- in place of strong connectivity.

  As we mentioned in Remark \ref{rem:connectivity}, these two conditions are equivalent, though exactly one of the implications is difficult.  For the purposes of the current sketch, we will assume the stronger hypothesis (that the structure groups are tori), as this makes the exposition somewhat simpler while still addressing the core of the argument.  The theorem is proved as stated in \cite{GMV2}*{Theorem 2.18}, and considerations arising from the weaker (strong connectivity) hypothesis are handled there.
\end{remark}

\begin{remark}
It is possible to formulate a version of Theorem \ref{toral-structure-theorem} without any reference to the structure groups. Then one needs to add some hypotheses on the topology of the space $X$, e.g. that it is locally connected and has finite Lebesgue covering dimension.
However, the proof of this statement requires the general structure theorem discussed in Section \ref{sec:structural-thm}.
For details, we refer to the Appendix of the companion paper \cite{GMV3}. 
\end{remark}

At the heart of the challenge in proving the above theorem is recovering the non-abelian group operation on $G$.  Although we were able to find a very explicit and combinatorial construction of the group operations on the structure groups $A_t$ just by exploiting concatenation and cube completion, there are seemingly fatal obstacles to such a simple approach working for the non-abelian group $G$.  Essentially this is because Theorem \ref{toral-structure-theorem} is simply not true, even in spirit, if the strong connectivity assumptions on $X$ are dropped, and hence any proof must have some non-elementary topological aspect.\footnote{An extreme case is to consider finite (so necessarily totally disconnected) nilspaces.  An example of one of these which is not of the form $G/\Gamma$ is given in \cite{HK08}*{Example 6}.}

Instead, the group operation is recovered in a completely different way: by considering the group of automorphisms of a nilspace $X$.

\begin{definition}
  Write $\Aut(X)$ for the (topological) group of cubespace automorphisms\index{cubespace!automorphism} of $X$; that is, of homeomorphisms $\phi \colon X \to X$ such that both $\phi$ and $\phi^{-1}$ are cubespace morphisms.
\end{definition}

The topology being used here is compact-open.  It turns out that $\Aut(X)$ has attached to it a canonical filtration,
\[
  \Aut(X) = \Aut_0(X) \supseteq \Aut_1(X) \supseteq \dots \supseteq \Aut_s(X) \supseteq \{\id\}
\]
of degree $s$, assuming $X$ is a nilspace of degree $s$.  The definition of the filtration is very explicit and depends only on the cubes of $X$, but we will not go into the details here; for a full description, see \cite{GMV2}*{Definition 2.7}.

It will typically be the case that the first containment $\Aut_0(X) \supseteq \Aut_1(X)$ is strict, and so we have no guarantee that $\Aut(X)$ is nilpotent -- in fact, often it will not be.\footnote{This can be seen even when $X = (\RR/\ZZ)^2$ with the degree $1$ filtration: $\Aut(X)$ contains a copy of $\SL_2(\ZZ)$.}  For this reason it is better to work with the subgroup $\Aut_1(X)$, sometimes called the group of \emph{$1$-translations}.  This inherits a filtration
\[
  \Aut_1(X) = \Aut_1(X) \supseteq \dots \supseteq \Aut_s(X) \supseteq \Aut_{s+1}(X) = \{\id\}\index[nota]{$\Aut_i(X)$}
\]
which guarantees that $\Aut_1(X)$ is nilpotent of nilpotency class at most $s$.

In general, the group $\Aut_1(X)$ need not be connected.  It turns out to be vital to work instead with the identity component of $\Aut_1(X)$, which we term $\Aut^\circ_1(X)$, and similarly with the revised filtration $\Aut^\circ_i(X)$, the identity component of $\Aut_i(X)$.  Note this need \emph{not} be the same as $\Aut^\circ_1(X) \cap \Aut_i(X)$, and hence one should check that this is actually a filtration, but this is not hard to show.

Given this filtration, we can consider the Host--Kra cube-groups $\HK^k(\Aut_1(X)_\bullet)$ or $\HK^k(\Aut_1^\circ(X)_\bullet)$.  Again these have a natural interpretation in terms of the cubes of $X$.  One consequence is that if $x_0 \in X$ and $(\phi_\omega)_{\omega \in \{0,1\}^k} \in \HK^k(\Aut_1(X)_\bullet)$ then the configuration $(\phi_\omega(x_0))_{\omega \in \{0,1\}^k}$ is a cube in $C^k(X)$.

In other words given $x_0 \in X$, the map
\begin{align*}
  \Aut_1(X) &\to X \\
       \phi &\mapsto \phi(x_0)
\end{align*}
is a cubespace morphism if $\Aut_1(X)$ is given its Host--Kra cubes (and hence this restricts to a cubespace morphism on $\Aut^\circ_1(X)$).  It will not be an isomorphism, as it is typically not injective: $x_0$ may have non-trivial stabilizer.  But we can define another map of cubespaces
\begin{align*}
  \Aut^\circ_1(X) / \stab(x_0) &\to X \\
       \phi \, \stab(x_0) &\mapsto \phi(x_0)
\end{align*}
which is well-defined and again a cubespace morphism.

The hope is that this is in fact surjective, and furthermore an isomorphism of cubespaces (which is a strictly stronger property).  If in addition $\Aut^\circ_1(X)$ were a Lie group and $\stab(x_0)$ a discrete and co-compact subgroup, we would have identified $X$ with a nilmanifold, as required.

However, it is not completely obvious \emph{a priori} that $\Aut(X)$ is not just the trivial group, in which case this map would be far from surjective.  Establishing that $X$ has ``enough'' automorphisms is the most difficult step in this strategy.

We will offer a few further clues as to the structure of this part of the proof.  Suppose for now that $s = 2$, and so $X$ is a nilspace of degree $2$.  One source of non-trivial automorphisms of $X$ is provided by the action of the top structure group $A_2(X)$ on $X$ described above.  Since this action fixes fibers of the quotient $X \to \pi_1(X)$ this still falls short of showing surjectivity.

However, we similarly have plenty of automorphisms of $\pi_1(X)$, given by the action of the next structure group $A_1(X)$ on $\pi_1(X)$.  In fact (assuming $X$ is ergodic) this action is simply transitive.

So, we know that we can move by elements of $\Aut^\circ_1(X)$ along fibers of $\pi_1$,\footnote{Note we have used here the extra assumption that $A_2(X)$ is connected.} and also that we can move between fibers by automorphisms of $\pi_1(X)$.  If we could show that these latter elements of $\Aut_1^\circ(\pi_1(X))$ can be lifted to elements of $\Aut_1^\circ(X)$, this would show that the action of $\Aut_1^\circ(X)$ is transitive, as desired.

With the formal definition of $\Aut_1(X)$ in place (which we deferred to \cite{GMV2}*{Definition 2.7}), it is not hard to show the following fact; this is done in \cite{GMV2}*{Proposition 3.2}.

\begin{lemma}
  Elements of $\Aut_1(X)$ commute with $\pi_{s-1}$, and so there is a natural group homomorphism $\Aut_1(X) \to \Aut_1(\pi(X))$.  Hence similarly we get a group homomorphism $\Aut^\circ_1(X) \to \Aut^\circ_1(\pi(X))$.
\end{lemma}

By the preceding discussion, the problem of showing surjectivity of the map $\Aut_1^\circ(X) \to X$ reduces to proving the following.

\begin{proposition}\label{prop:toral implies surjective}
  Suppose $X$ is an ergodic nilspace of degree $s$, whose structure groups $A_t$ are tori.  Then the map $\Aut^\circ_1(X) \to \Aut^\circ_1(\pi(X))$ from the previous proposition is surjective.
\end{proposition}

In the $s=2$ case, this would tell us that $\Aut^\circ_1(X)$ acts transitively on the fibers of $\pi_1$ (i.e.~given any two fibers, we can map some point on one fiber to some point on the other); and since we already know that $A_2$ acts transitively on each fiber, this would show that $\Aut_1^\circ(X)$ acts transitively on $X$ and so the map $\Aut_1^\circ(X) \to X$ is surjective.

Moving from surjectivity to isomorphism (and from the $s=2$ case to the general case) involves some extra work, but it is mainly technical.

To prove Proposition \ref{prop:toral implies surjective}, one begins by showing that if a $1$-translation on $\pi_1(X)$ is a small perturbation of the identity, then it lifts to a $1$-translation on $X$ (if $X$ is a nilspace of degree $2$).  The proof of this is roughly to relate failure of lifting to some kind of ``cocycle'', and therefore express obstructions to lifting in terms of some kind of ``cohomology'' (where we use the term in a very loose sense).  One then argues that this ``cohomology'' is somehow discrete, and so provided the $1$-translation is small enough, no obstructions arise.

These notions of ``cocycles'' and ``cohomology'' are useful tools, and we will allude to them again below.  For the relevant formal definitions of cocycles and coboundaries, see \cite{GMV2}*{Definition 4.8} and the subsequent discussion; for a more in-depth discussion of cohomology more generally, and its relation to extensions of cubespaces, see \cite{CS12}*{Section 2.10} (or \cites{Can1,Can2}).

The proof of Theorem \ref{toral-structure-theorem} is very closely modeled on the arguments of Antol\'\i n Camarena and Szegedy \cite{CS12}. The main difference is the following. In constructing the lift of a translation on the canonical quotient, we first construct (in the paper \cite{GMV2}) a continuous lift, which is not necessarily a translation. Then we use the action of the structure group to ``correct it'' to a translation. On the other hand in \cite{CS12}, a measurable lift is constructed first, which is shown to be continuous a posteriori.

\subsection{The inverse limit statement}

Given this structure theorem for (ergodic, compact, strongly connected) nilspaces whose structure groups are Lie, our remaining task is to deduce something for general (ergodic, compact, strongly connected) nilspaces.

Essentially this task reduces to showing some version of the following (see also \cite{CS12}*{Theorem 4}).

\begin{theorem}
  \label{thm:invlim}
  Let $X$ be an ergodic compact nilspace of degree $s$.

  Then there exists a sequence $X_n$ of ergodic compact nilspaces of degree $s$, such that $A_t(X_n)$ are Lie groups for all $n$ and all $1 \le t \le s$, and such that $X \cong \varprojlim X_n$ (with the inverse limit being in the sense of cubespaces).
\end{theorem}
In fact the result as proven in \cite{GMV3}*{Theorem 1.26} is marginally stronger, but in ways that use terminology we have not yet discussed.

Note that this gives some partial information about general (ergodic, compact) nilspaces without any connectivity assumption.

This is not quite enough to deduce Theorem \ref{main-structural-v1} as we do not get the precise description of the maps appearing in the inverse limit.  A complete proof of a suitably stronger result can be found in \cite{GMV3}*{Theorems 1.27 and 1.28}.  
We give an outline below, ignoring these finer points.

For simplicity, we again assume that the structure groups $A_k(X)$ are connected for all $k \ge 1$,
rather than just that $X$ itself is strongly connected

The starting point is again the weak structure theorem, together with the result that any compact connected abelian group is an inverse limit of tori.

Again consider first the case $s=2$.  Since $A_2(X)$ is a compact connected abelian group, we may write $A_2(X) = \varprojlim K_m$ where $K_m$ are tori (with given surjective maps between them).  Furthermore, it is straightforward to quotient $X$ by the action of a subgroup of $A_2$ (the kernel of one of the projections $A \to K_m$) to deduce the following:

\begin{fact*}
  Under these hypotheses, $X$ is an inverse limit of nilspaces $X_m$ of degree $2$, with the property that $A_2(X_m)$ is a torus for all $m$.
\end{fact*}

This will allow us to assume -- given a bit of work -- that $A_2(X)$ is itself a torus, as otherwise we can reduce to that case.

Meanwhile, $A_1(X)$ is also a connected compact abelian group so we can write $A_1(X) = \varprojlim M_r$ for some tori $M_r=(\RR/\ZZ)^{d_{r}}$.  We can deduce that $\pi_1(X)$ is an inverse limit of tori (in the sense of cubespaces).

The remaining challenge is to prove the following.

\begin{lemma}
  Let $X$ be an ergodic compact nilspace of degree $2$, such that $A_2(X)$ is a torus.  Suppose further that   $A_1(X) = \varprojlim M_r$ for some tori $M_r$. Then for all sufficiently large integers $r$, there exists a canonical (degree $2$, ergodic, compact) nilspace $X_r$ and a quotient map $X \to X_r$ such that $A_2(X_r) \cong A_2(X)$, $A_1(X_r) \cong M_r$ and the diagram
  \\[0.5\baselineskip]
  \[
  \begin{CD}
    X @>>> X_r \\
    @VV\pi_1V @VV\pi_1V \\
    \cD_1(A_1(X)) @>\phi>> \cD_1(M_r)
  \end{CD}
  \]
  \\[0.5\baselineskip]
  commutes, where $\phi \colon A_1(X) \to M_r$ is the map from the inverse limit.
\end{lemma}

The problem is essentially one of ``pushing forward'' the degree $2$ cubespace structure on $X$ along the map $\phi$.  Unfortunately this operation does not make sense in general.

Again, though, it will turn out that the obstructions to building such a push-forward space are measured by a kind of cohomology; in fact the same one as was used to prove the toral structure theorem.  Again, we use a kind of discreteness result for this cohomology to argue that, if $M_r$ is sufficiently close to $\pi_1(X)$ in some sense, then these obstructions do not arise.

One way of measuring this ``closeness'' is the statement that the fibers of the map $\pi_1(X) \to M_r$ should have small diameter with respect to the metric on $\pi_1(X)$.  Since $\pi_1(X) = \varprojlim M_r$ holds in the sense of metric spaces, this will be true provided $r$ is sufficiently large.

Again, the general case follows a very similar pattern, with this same argument repeated once for each $1 \le t \le s$.

As before our approach has much in common with that in \cite{CS12}, but differs in some important respects.  Notably, the argument in that paper proceeds by establishing a correspondence between extensions of a given nilspace $X$ by a given compact abelian group $A$, up to isomorphism, and classes of measurable cocycles on $X$.  An intricate argument is required to recover a topological object (the extended nilspace) from the measurable data of the cocycle. Some cocycle theory is then used to ``push forward'' the cocycle arising from $X \to \pi_1(X)$ onto $\cD_1(M_r)$ to build $X_r$.  By contrast, our argument realizes $X_r$ as an explicit quotient of $X$, although many related tools are needed in the process.

\subsection{A dictionary of statements}
\label{sec:comparison}

For clarity and convenience, we summarize which of the main statements in these papers correspond to which in \cite{CS12}.

\begin{itemize}
  \item The result concerning nilspaces whose structure groups are Lie, Theorem \ref{toral-structure-theorem}, is analogous to \cite{CS12}*{Theorem 7}; see Remark \ref{rem:strong-connectivity} for a detailed comparison.
  \item The inverse limit theorem, Theorem \ref{thm:invlim}, is stated identically to \cite{CS12}*{Theorem 4} (though the proof strategies differ, as remarked above).
  \item Our main structure theorem, Theorem \ref{main-structural-v1}, has no explicit counterpart in \cite{CS12}, but should be compared to what one obtains by concatenating \cite{CS12}*{Theorems 4 and 7}; see Remarks \ref{rem:inverse-limit} and \ref{rem:connectivity}.
\end{itemize}
Theorem \ref{main-structural-v1} also depends on \cite{GMV3}*{Theorem 1.27}, 
to obtain the sharper inverse limit statement (in the sense of the second paragraph of Remark \ref{rem:inverse-limit}).  The latter has no counterpart in \cite{CS12}.

\section{The weak structure theory}
\label{sec:weak-overview}

We now turn to detailed statements and proofs of the weak structure theory.

In fact, we will approach this in two stages.  This section will present the ``standard'' weak structure theorem \ref{basic-cs-structure-theorem} that appears in \cite{CS12}*{Theorem 1}, and give most of the proofs (although these may differ in some places from those in \cite{CS12}).  However, elsewhere in the project we will need a ``relative'' analogue of all of these statements.  This relative theory is developed in Section \ref{sec:elementary-v2}; many of the proofs there will be closely analogous to those from this section.

The relative versions are a strict generalization of the non-relative ones, so logically the results in Section \ref{sec:elementary-v2} suffice.  Sometimes, therefore, we omit part of a non-relative proof and refer to the corresponding result from Section \ref{sec:elementary-v2}.  Usually we do not, but this is for essentially pedagogical reasons.

\subsection{Glueing}

As a preliminary, we introduce one further property of cubespaces.

\begin{definition}
  \label{defn:glueing}
  We say a cubespace $X$ has the \emph{glueing property}\index{glueing} if ``glueing'' two cubes along a common face yields another cube.

  Formally, suppose $c, c' \in C^k(X)$, and $c(\omega 1) = c'(\omega 0)$ for all $\omega \in \{0,1\}^{k-1}$.  (Here we use $\omega 0$ to denote $(\omega_1, \dots, \omega_{k-1}, 0)$ and so on.)

  Then the configuration
  \[
    c'' \colon \omega \mapsto \begin{cases} c(\omega) &\colon \omega_k = 0 \\ c'(\omega) &\colon \omega_k = 1 \end{cases}
  \]
  is in $C^k(X)$.
\end{definition}

The reason we have not defined this previously is that it follows from $k$-completion (for all $k$).

\begin{proposition}
  \label{prop:fibrant-glueing}
  Suppose a cubespace $(X, C^k(X))$ has $k$-completion for all $k$.  Then it satisfies the glueing property.
\end{proposition}
\newcommand*{\glueingpic}[0]{
  \inlinetikz{
    \begin{scope}[shift={(-3,0)}]
      \begin{scope}[x={(1, 0)}, y={(0, 1)}, z={(0.352, 0.317)}, scale=2]
        \draw[fill=black!40,fill opacity=.4] (0,0,0) -- (0,0,1) -- (0,1,1) -- (0,1,0) -- cycle;
        \draw[fill=black!40,fill opacity=.4]  (0,0,0) -- (1,0,0) -- (1,0,1) -- (0,0,1) -- cycle;
      \end{scope}
    \end{scope}
    \draw[->] (0, 1.5) -- (1, 1.5);
    \begin{scope}[shift={(2,0)}]
      \begin{scope}[x={(1, 0)}, y={(0, 1)}, z={(0.352, 0.317)}, scale=2]
        \draw[fill=black!40,fill opacity=.4] (0,0,0) -- (0,0,1) -- (0,1,1) -- (0,1,0) -- cycle;
        \draw[fill=black!40,fill opacity=.4]  (0,0,0) -- (1,0,0) -- (1,0,1) -- (0,0,1) -- cycle;
        \draw[dashed] (1,0,1) -- (1,1,1) -- (0,1,1);
        \draw[dashed] (1,0,0) -- (1,1,0) -- (0,1,0);
        \draw[dashed] (1,1,1) -- (1,1,0);
        \draw[fill=black] (1,1,0) circle[radius=0.02cm];
        \draw[fill=black] (1,1,1) circle[radius=0.02cm];
      \end{scope}
    \end{scope}
    \draw[->] (5.5, 1.5) -- (6.5, 1.5);
    \begin{scope}[shift={(7.5,0)}]
      \begin{scope}[x={(1, 0)}, y={(0, 1)}, z={(0.352, 0.317)}, scale=2]
        \draw[fill=black!40,fill opacity=.4] (0,0,0) -- (0,0,1) -- (0,1,1) -- (0,1,0) -- cycle;
        \draw[fill=black!40,fill opacity=.4]  (0,0,0) -- (1,0,0) -- (1,0,1) -- (0,0,1) -- cycle;
        \draw (1,1,1) -- (0,1,1) -- (0,0,1) -- (1,0,1) -- cycle;
        \draw (1,1,1) -- (1,1,0) -- (1,0,0) -- (1,0,1) -- cycle;
        \draw (1,1,1) -- (0,1,1) -- (0,1,0) -- (1,1,0) -- cycle;
        \draw[fill=black!50, opacity=.8] (0,1,0) -- (0,1,1) -- (1,0,1) -- (1,0,0) -- cycle;
        \draw (0,0,0) -- (1,0,0) -- (1,1,0) -- (0,1,0) -- cycle;
      \end{scope}
    \end{scope}
  }
}
\begin{proof}[Proof sketch]
  We will defer a complete proof until Section \ref{sec:elementary-v2}, by which time we will have developed the machinery for a clean argument.  Meanwhile, we will draw some pictures in the case $k=2$.

  The idea is to place the two $2$-cubes to be glued as faces of a partial $3$-cube configuration; complete that (in two steps) to a $3$-cube; and then observe that the ``glued'' configuration is a sub-cube of this, and hence a cube by the axioms.
  \glueingpic{}
\end{proof}

\subsection{The Heisenberg example}
\label{sec:hberg-example}

As motivation, we will examine an example called the Heisenberg nilmanifold.
This is the nilspace $\HK(\cH_\bullet)/\Gamma$ constructed as in Appendix \ref{app:hk}, where
  \[
    \cH = \left\{ \heis{x}{y}{z} \colon x,y,z \in \RR \right\}
  \]
is the Heisenberg group equipped with the filtration
$\cH = \cH_0 = \cH_1 \supseteq \cH_2 \supseteq \{\id\}$ where $\cH_2$ is the center,
\[
  \cH_2 = \left\{ \heis{0}{0}{z} \colon z \in \RR \right\}
\]
and $\Gamma$ is the discrete co-compact subgroup that consists of the elements of $\cH$ with integral entries.

Note that, for any filtered group, $G_j$ is a normal subgroup of $G$ (since $[G_0, G_j] \subseteq G_j$ by the filtration property), and so we may consider the group quotient $G \to G / G_j$.  This comes with an induced filtration $(G_i / G_j)_{i \ge 0}$ of degree $j-1$.

For $\cH$, the non-trivial case is $\cH / \cH_2$.  This is isomorphic to the abelianization $\RR^2$ with the degree $1$ filtration $\RR^2 = \RR^2 \supseteq \{0\}$.

We also get an induced map on the nilmanifold $\cH / \Gamma$, i.e.~$\pi \colon \cH / \Gamma \to (\cH / \cH_2) / (\Gamma / (\Gamma \cap \cH_2)) = \RR^2 / \ZZ^2$.  This map corresponds to
\[
  \pi \colon \heis{x}{y}{z} \Gamma \mapsto (x, y) \bmod 1 \in \RR^2 / \ZZ^2
\]
and one can check this is well-defined.  Hence, the nilmanifold $\cH / \Gamma$ has a quotient isomorphic to a torus $\RR^2 / \ZZ^2$, sometimes called the ``horizontal torus''.  Moreover, it is essentially automatic that this map is well-behaved with respect to the Host--Kra cubes $\HK^k(\cH_\bullet)/\Gamma$: the image of these cubes is precisely $\HK^k(\RR^2/\ZZ^2)$.

The fibers of this quotient map are of the form
\[
  \left\{ \heis{x}{y}{z'} \Gamma \colon z' \in \RR / \ZZ \right\}
\]
i.e.~each fiber has an action by the center $\cH_2$, or more precisely a simply transitive action by $\cH_2 / (\cH_2 \cap \Gamma) = \RR / \ZZ$.  Again, this action respects cubes in the following sense: given an element
\[
  c(\omega) = \heis{x_\omega}{y_\omega}{z_\omega} \Gamma
\]
in $\HK^3(\cH_\bullet)/\Gamma$, and another configuration
\[
  c'(\omega) = \heis{x_\omega}{y_\omega}{z'_\omega} \Gamma
\]
so $\pi(c) = \pi(c')$, then one can check that $c'$ is a cube if and only if
\[
  z_{000} - z_{001} - \dots + z_{110} - z_{111} =
  z'_{000} - z'_{001} - \dots + z'_{110} - z'_{111} \ .
\]
Equivalently, $c'$ is a cube if and only if it is obtained by acting on $c$ pointwise by an element of $\HK^3(\RR/\ZZ)$, where $\RR/\ZZ$ is given the degree $2$ filtration $\RR/\ZZ = \RR/\ZZ = \RR/ \ZZ \supseteq \{0\}$.

In summary, for the Heisenberg nilmanifold we have maps
\[
  \cH / \Gamma \xrightarrow{\pi} \RR^2 / \ZZ^2 \to \{\ast\}
\]
where for each map the fibers have a simply transitive action by a compact abelian group ($\RR/\ZZ$ and $\RR^2/\ZZ^2$ respectively); and there is an induced map
\[
  \HK^k(\cH_\bullet) / \Gamma \xrightarrow{\pi} \HK^k(\RR^2 / \ZZ^2) \to \{\ast\}
\]
where the fibers are given by $\HK^k$ of the corresponding compact abelian groups, equipped with the ``trivial'' filtrations of degree $2$ and $1$ respectively.

It is not hard to argue that a similar property holds for any nilmanifold of any degree, not just the Heisenberg nilmanifold.  For a filtered group of degree $s$, one gets a sequence of maps
\[
  G \to G / G_s \to G/G_{s-1} \to \dots \to G/G_2 \to G/G_1
\]
where typically $G = G_1$ and so the last step is the trivial group.  The kernel of each of these maps is an abelian group.  Again, these maps induce maps of nilmanifolds
\[
  G/\Gamma \to (G / G_s) / (\Gamma / (\Gamma \cap G_s)) \to (G / G_{s-1}) / (\Gamma / (\Gamma \cap G_{s-1})) \to \dots
\]
where the fibers of each map are now compact abelian groups.  This is a tower of extensions of nilmanifolds.  A similar compatibility property of the Host--Kra cubes to that stated above also holds.

In this discussion, we have made heavy use of the group structure of $G$ and the precise description of Host--Kra cubes. It is a slightly surprising but very important fact that this sequence of quotient maps can be recovered in the completely abstract setting of general nilspaces.  Indeed, for any nilspace $X$ we can define a tower
\[
  X \to X / \sim_{s-1} \to X / \sim_{s-2} \to \dots
\]
where $\sim_{s-1}$ is the equivalence relation from Definition \ref{df:canonical-rel}.  When $X = G/\Gamma$ as above, it turns out this is the same tower that we just constructed.  For general $X$, many of the properties discussed above -- notably, that the fibers are compact abelian groups, which act compatibly on cubes in the sense described -- can be proven in the abstract setting.

In other words, we do not need to know about the global group operation on $G$ to recover the tower of  quotients of the nilmanifold $G/\Gamma$, and indeed the same conclusions hold even for nilspaces that are not of the form $G/\Gamma$ for any group $G$.

\subsection{The canonical factors}
\label{sec:canonical-factors}

As just discussed, for a general compact cubespace $X$ with $k$-completion we are interested in the quotient $X / \sim_{s-1}$ of $X$, which is the general analogue of the quotients $(G / G_s) / (\Gamma / (\Gamma \cap G_s))$ of a filtered nilmanifold $G/\Gamma$.   The space $X / \sim_{s-1}$ was defined in Section \ref{sc:outline-weak}: see Definition \ref{df:quotient} (quotient cubespace) and Definition \ref{df:canonical-rel} (canonical equivalence relation).

Our first task is to verify the properties of $X / \sim_s$ that we claimed in Section \ref{sc:outline-weak} immediately after the definition (and a few more).  This discussion is very similar to \cite{CS12}*{Section 2.4} and \cite{HK08}*{Section 3.3}.

\begin{proposition}
  \label{prop-canonical-factors}
  Let $(X, C^k(X))$ be a compact cubespace with the glueing property, and let $s \ge 0$ be an integer.  Then the canonical equivalence relation $\sim_s$  is indeed a closed equivalence relation and satisfies the following
{\em universal replacement property}.\index{universal replacement property}
If $k \le s+1$ and $c \in C^{k}(X)$, $c' \colon \{0,1\}^k \to X$, $c(\omega) = c'(\omega)$ for all $\omega \ne \vec{1}$ and $c(\vec{1}) \sim_s c'(\vec{1})$, then $c' \in C^{k}(X)$.

  Moreover, $X / \sim_s$ has $(s+1)$-uniqueness, and $\sim_s$ is the smallest equivalence relation with this property.  Finally, if $X$ has $k$-completion for all $k$ then so does $X / \sim_s$, and hence $X / \sim_s$ is a compact nilspace of degree $s$.
\end{proposition}

\begin{remark}
  Note we are using a convention of stating results using $\vec{1}$ as the ``special'' vertex of the discrete cube $\{0,1\}^k$, whenever one is required.  However, using the various morphisms of discrete cubes from the definitions of a cubespace, which act transitively on $\{0,1\}^k$, these properties hold just as well for any fixed vertex in $\{0,1\}^k$.
\end{remark}

\begin{remark}
  \label{rem:sim_0}
  Consider the case $s=0$.  Then $x \sim_0 y$ if and only if $[x, y]$ is a $1$-cube.  It is immediate from the glueing property that this is an equivalence relation.  If $X$ is ergodic, then $X / \sim_0 = \{\ast\}$ is a one-point space.  More generally, this identifies the ``ergodic components'' of $X$.
\end{remark}

The following more explicit characterization of $\sim_s$ will be very helpful.
See also \cite{CS12}*{Lemma 2.3} and \cite{HK08}*{Proposition 3}.

\begin{lemma}
  \label{lem:canonical-characterization}
  We have $x \sim_s y$ if and only if the configuration $c\colon \{0,1\}^{s+1} \to X$ given by
  \[
    c(\omega) = \begin{cases} y & \colon \omega = \vec{1} \\ x & \colon \omega \ne \vec{1} \end{cases}
  \]
  is an $(s+1)$-cube.
\end{lemma}
\begin{proof}[Proof sketch]
  The ``if'' direction is straightforward, since the constant configuration $c'(\omega) = x$ for all $\omega$ is in $C^{s+1}(X)$ by the axioms (see Remark \ref{rem:constant-cube}).  We illustrate the ``only if'' direction with a picture when $s=1$; the general case is similar but notationally awkward.

  We know there exist $c, c'$ with $c(\omega) = c'(\omega)$ for $\omega \ne \vec{1}$, $c(\vec{1}) = x$ and $c'(\vec{1}) = y$.  Consider the picture
  \inlinetikz
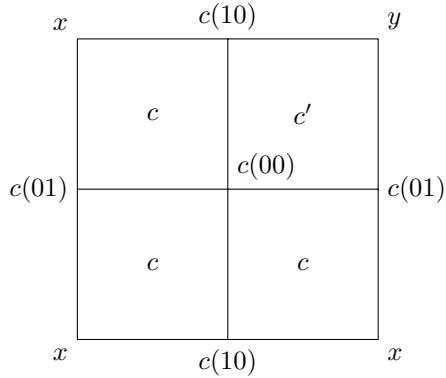
{
    \begin{scope}[scale=2]
      \draw[thin,black] (0,0) grid (2, 2);
      \node [below left] at (0, 0) {$x$};
      \node [above left] at (0, 2) {$x$};
      \node [below right] at (2, 0) {$x$};
      \node [above right] at (2, 2) {$y$};
      \node [left] at (0, 1) {$c(01)$};
      \node [right] at (2, 1) {$c(01)$};
      \node [below] at (1, 0) {$c(10)$};
      \node [above] at (1, 2) {$c(10)$};
      \node [above right] at (1, 1) {$c(00)$};
      \node at (0.5, 0.5) {$c$};
      \node at (0.5, 1.5) {$c$};
      \node at (1.5, 0.5) {$c$};
      \node at (1.5, 1.5) {$c'$};
    \end{scope}
  }
  It is clear each of the small cubes is a cube, as they are just rotated and reflected copies of $c$ or $c'$.  By two applications of the glueing property, the outer square is also a cube, as required.
\end{proof}

\begin{proof}[Proof sketch of Proposition \ref{prop-canonical-factors}]
  First we check that $\sim_k$ is an equivalence relation.  Symmetry and reflexivity are immediate from the original definition.  For transitivity, we apply the lemma; again we will give a pictorial sketch when $s=1$.  Suppose $x \sim_s y$, $y \sim_s z$ and consider
  \inlinetikz{
    \doublesquare{$x$}{$y$}{$z$}{$y$}{$y$}{$y$}{1}
  }
  where the left and right squares are cubes by the lemma and a reflection.  Then the outer rectangle is a cube by the glueing property.  But so is (say)
  \inlinetikz{
    \singlesquare{$y$}{$y$}{$x$}{$x$}
  }
  by a duplication operation (i.e.~using the cubespace axioms), and so $x \sim_s z$ from the definition, as required.

  Note that it is clear from the hypotheses about the spaces $C^k(X)$ being closed and $X$ being compact that $\sim_s$ is a closed relation.

  Next, we prove the universal replacement property.  First suppose $k=s+1$.  Take $c, c'$ as in the statement, and write $x = c(\vec{1})$, $y = c'(\vec{1})$.  Again we give a picture:
  \inlinetikz
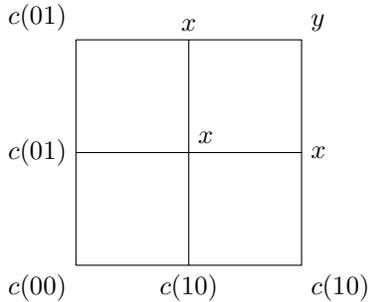
{
    \begin{scope}[scale=1.5]
      \draw[thin,black] (0,0) grid (2, 2);
      \node [below left] at (0, 0) {$c(00)$};
      \node [above left] at (0, 2) {$c(01)$};
      \node [below right] at (2, 0) {$c(10)$};
      \node [above right] at (2, 2) {$y$};
      \node [left] at (0, 1) {$c(01)$};
      \node [right] at (2, 1) {$x$};
      \node [below] at (1, 0) {$c(10)$};
      \node [above] at (1, 2) {$x$};
      \node [above right] at (1, 1) {$x$};
    \end{scope}
  }
  The bottom left square is just $c$, the top right is a cube by the lemma, and the other two small squares are cubes again by a duplication.  Hence the outer square is a cube by glueing, as required.

  Now suppose $k < s+1$.  Given a cube $c \in C^k(X)$, using appropriate morphisms of discrete cubes and the cubespace axioms, we may
  \begin{itemize}
    \item duplicate $c$ up to an $(s+1)$ cube $\tilde{c} \in C^{s+1}(X)$;
    \item change entries of $\tilde{c}$ repeatedly using the above; and
    \item restrict to some appropriate face of $\tilde{c}$ to obtain a cube of dimensions $k$ with the desired properties.
  \end{itemize}
  Hence we have universal replacement for all $k \le s+1$.

  We now argue the definition of $\sim_s$ and the universal replacement property imply $(s+1)$-uniqueness of the quotient.  Suppose $\tilde{c}$ and $\tilde{c}'$ are two cubes in $C^{s+1}(X/\sim_s)$ such that $\tilde{c}(\omega) = \tilde{c}'(\omega)$ for all $\omega \ne \vec{1}$.  By definition of the quotient cubespace, there are cubes $c, c' \in C^{s+1}(X)$ such that $\pi_s(c) = \tilde{c}$, $\pi_s(c') = \tilde{c}'$.  But by repeated application of the universal replacement property, \emph{any} configuration $c \colon \{0,1\}^{s+1} \to X$ such that $\pi_s(c) = \tilde{c}$ is a cube, and similarly for $c'$.  Hence we are free to choose $c, c'$ such that $c(\omega) = c'(\omega)$ for each $\omega \ne \vec{1}$. By definition of $\sim_s$, we now see that $c(\vec{1}) \sim_s c'(\vec{1})$, and so $\tilde{c}(\vec{1}) = \tilde{c}'(\vec{1})$ as required.

  As remarked after Definition \ref{df:canonical-rel}, the statement that $\sim_s$ is the smallest equivalence relation $\sim$ for which $X / \sim$ has $(s+1)$-uniqueness is clear.

  The final statement to prove is the completion property for $X/\sim_s$.  In fact this follows from more general statements about ``fibrations'', defined in Section \ref{sec:elementary-v2}, so we defer the proof to there (see Remark \ref{rem:completion for quotient}).
\end{proof}

Note it is clear that $X/\sim_s$ is also ergodic if $X$ is.

\begin{remark}
  We observe that this quotient $\pi_s$ as constructed is completely canonical: firstly in that it depends only on the cubespace $(X, C^k(X))$, and secondly that it has the following universal property: if $(Y, C^k(Y))$ is any other nilspace of degree $s$ then any cubespace morphism $X \to Y$ factors through $\pi(X)$.

  In particular, $\sim_s$ is trivial if and only if $X$ has $(s+1)$-uniqueness.
\end{remark}

The full strength of these canonical factors comes when they are chained together.  That is, as before we can construct a tower of maps
\[
  X \xrightarrow{\pi_s} \pi_s(X) \xrightarrow{\pi_{s-1}} \pi_{s-1}(X) \rightarrow \dots \xrightarrow{\pi_0} \pi_0(X)
\]
where $\pi_0(X) = \{\ast\}$, the one-point space, provided $X$ is ergodic.  To check this makes sense, we need to verify the following trivial observation.

\begin{proposition}
  \label{equiv-nested}
  The equivalence relations $\sim_s$ are nested, i.e.~if $t \ge s$ and $x \sim_t y$ then $x \sim_s y$.  Moreover, $\pi_s(\pi_t(X)) = \pi_s(X)$, i.e.~the definition of $\sim_s$ is not affected by first quotienting by $\sim_t$.
\end{proposition}
\begin{proof}
  If $c, c' \in C^{t+1}(X)$ are any cubes verifying $x \sim_t y$ then any sub-cube of dimension $(s+1)$ containing $\vec{1}$ will verify $x \sim_s y$.
\end{proof}

\subsection{Structure groups and the weak structure theorem}
\label{sec:structure-groups}

Having established the existence of the tower
\[
  \dots \xrightarrow{\pi_s} \pi_s(X) \xrightarrow{\pi_{s-1}} \pi_{s-1}(X) \rightarrow \dots \xrightarrow{\pi_0} \pi_0(X)
\]
in the abstract, to further the analogy with the Heisenberg case (or more generally, the case of any nilmanifold) we need to describe the \emph{fibers} of each map $\pi_t$.  Recall these are expected to have the structure of a compact abelian group; or more precisely, to have a free and transitive action by a compact abelian group.

As in Section \ref{sc:outline-weak} we let $\cD_s(A)$ denote the Host--Kra cubespace on an Abelian group
equipped with the filtration $A = A_0 = A_1 = \dots = A_s \supseteq \{0\}$, which is a nilspace of degree $s$.
Recall also the statement of the weak structure theorem (Theorem \ref{basic-cs-structure-theorem}).

\begin{remark}
  Suppose $X$ is an ergodic nilspace of degree $1$, and consider the weak structure theorem (Theorem \ref{basic-cs-structure-theorem}) in the case $s=1$.  Since $\pi_0(X) = \{\ast\}$ (see Remark \ref{rem:sim_0}), the theorem asserts precisely that $X$ is isomorphic to $\cD_1(A)$ for some compact abelian group $A$, i.e.~$X$ bijects with $A$ and the cubes of $X$ are identified with the Host--Kra cubes of $A$ with the degree $1$ filtration.

  (This isomorphism is found by fixing $x_0 \in X$ and identifying $A \leftrightarrow X$ by $a \leftrightarrow a(x_0)$.  In particular this is canonical only up to the choice of $x_0$.)
\end{remark}

A proof of this case is sketched above in Section \ref{sc:1step}.  However, that proof does not generalize entirely cleanly to the case $s > 1$.  We will now sketch a slightly different argument that does generalize, but working again in the case $s=1$ for ease of notation.

First we will sketch how the proof would look if we already knew that $X = \cD_1(A)$ for some $A$.  We consider all the edges \footnote{An edge is just another name for a $1$-cube.} $[a, b]$ of $X$, and associate to each one the group element $(b - a)$.  We define an equivalence relation on edges by $[a,b] \sim [a', b']$ if $b - a = b' - a'$; so $A$ is precisely the set of equivalence classes of edges under this relation.

Given edges $[a, b]$ and $[b, c]$ we can concatenate them to get $[a, c]$.  If the associated elements of $A$ are $r = b - a$ and $s = c - b$ then $[a,c]$ is associated to $c - a = r + s$.  Thus we can recover the group addition operation on $A$ in a combinatorial fashion by concatenation.  The other operations are similarly easy to define.

The task is now to show that these same constructions make sense without any \emph{a priori} assumptions on $X$ other than that it is a nilspace of degree $1$.

\begin{proof}[Sketch proof of Theorem \ref{basic-cs-structure-theorem}, $s=1$]
  Since $X$ is ergodic, the edges $C^1(X)$ are in bijection with $X \times X$. For two edges $[x, y]$ and $[x', y']$ we write $[x,y] \sim [x',y']$ whenever
    \inlinetikz{\singlesquare{$x$}{$y$}{$x'$}{$y'$}}
  is a $2$-cube.  This is an equivalence relation by the cubespace axioms and glueing.  As a set, we define $A := (X \times X) / \sim$.

  \begin{claim}\label{claim:uniqueness}
    Let $[x,y]$ be an edge.  For any fixed $e \in X$ there is an unique $a \in X$ such that $[x,y] \sim [e,a]$.  In other words, the class of $[x,y]$ in $A$ has an unique representative of the form $[e, a]$.
  \end{claim}
  \begin{proof}[Proof of claim]
    This follows trivially from unique completion of the corner
      \inlinetikz{\singlesquare{$x$}{$y$}{$e$}{$\ast$}}
    giving $a$ in the top right.
  \end{proof}

  We now define addition on $A$ by concatenation as suggested above.  More precisely, given $g, h \in A$, we fix an $e \in X$ and choose the unique representatives $g \sim [x, e]$ and $h \sim [e, y]$; then $g + h$ is defined to be the class of the concatenation $[x, y]$.

  \begin{claim}
    \label{well-defined-claim-v1}
    This definition of $g + h$ is well-defined; i.e.~the outcome does not depend on the choice of $e$.
  \end{claim}
  \begin{proof}[Proof of claim]
    Considering the diagram
    \inlinetikz{\doublesquare{$x$}{$e$}{$y$}{$x'$}{$e'$}{$y'$}{1.0}}
    where $[x, e] \sim [x', e'] \sim g$ and $[e, y] \sim [e', y'] \sim h$, since the outer square is a $2$-cube by glueing we have $[x, y] \sim [x', y']$ as required.
  \end{proof}

  The identity $0 \in A$ is given by the class of constant edges $[x, x]$ (by $2$-uniqueness this is indeed a class of $\sim$), and inversion by the operation $[x, y] \mapsto [y, x]$.

  \begin{claim}
    The set $A$ with these operations forms a (topological) abelian group.
  \end{claim}
  \begin{proof}[Proof of claim]
    That the proposed identity and inverse operation do what they claim is a trivial check.  Associativity follows from associativity of concatenation, together with Claim \ref{well-defined-claim-v1}.  For commutativity, suppose $[x, e] \sim g$, $[e, y] \sim h$; by completing the corner to get $e'$ in
    \inlinetikz{\singlesquare{$e$}{$y$}{$x$}{$e'$}}
    we have $[e', y] \sim g$, $[x, e'] \sim h$ and hence $h + g \sim [x, y]\sim g+h$ as required.

    Issues of continuity are not difficult to justify.  Note that the equivalence relation $\sim$ is closed (as $C^2(X)$ is a closed subspace of $X^4$) and so $A$ is a compact metric space.  We now sketch the proof that $+:A\times A\rightarrow A$ is continuous.  For a fixed $e \in X$, the map $r_e \colon X \to A$ given by $x \mapsto [e,x]$ is clearly continuous (by the definition of the product and quotient topologies) and a bijection (by Claim \ref{claim:uniqueness}) so is a homeomorphism.  The same holds for $\ell_e \colon x \mapsto [x,e]$.  Hence, the composite
    \begin{align*}
      A \times A &\to X \times X \to A \\
      (g,h) &\mapsto (\ell_e^{-1}(g),r_e^{-1}(h)) \mapsto [\ell_e^{-1}(g),r_e^{-1}(h)] / \sim
    \end{align*}
    is also continuous; but this is the definition of $+$, as required.
  \end{proof}
  \vspace{\baselineskip}
  We now define the group action of $A$ on $X$.  Given $g \in A$ and $x \in X$ we take $g(x)$ to be the unique element of $X$ such that $g \sim [x, g(x)]$.  It is clear $0 \in A$ acts trivially; looking at the diagram
  \inlinetikz{\doublesquare{$x$}{$g(x)$}{$h(g(x))$}{$c$}{$e$}{$d$}{1.0}}
  where $g \sim [c, e]$ and $h \sim [e, d]$, we conclude this is indeed a group action by considering the outer square.  Moreover, it is trivially simply transitive.

  Finally we must investigate the cubes of $X$ in terms of $A$, i.e.\ prove (ii) from Theorem \ref{basic-cs-structure-theorem}.  For $0$ and $1$-cubes there is nothing to say.  Suppose $c$ is the configuration
  \inlinetikz{\singlesquare{$x$}{$y$}{$z$}{$w$}}
  in $X^{\{0,1\}^2}$.  By definition of $\sim$, this is a cube if and only if $[x, y] \sim [z, w]$.  Furthermore it is obtained by acting on the constant cube
  \inlinetikz{\singlesquare{$e$}{$e$}{$e$}{$e$}}
  by the elements of $A$ represented by $r := [e, x]$, $s := [e, y]$, $t := [e, z]$, $u := [e, w]$ respectively.

  By considering the concatenation of $[x,e]$, $[e,y]$ and using the definitions of the group operation on $A$, we find that $[x, y] \sim s - r$, and similarly $[z, w] \sim u - t$.   Hence, $c$ is a cube if and only if $[x,y] \sim [z,w]$, if and only if $s - r = u - t$; i.e.~if and only if $c$ is obtained by acting on a constant cube by an element of $C^2(\cD_1(A))$.  This is precisely what is required by (ii) in the case $k=2$.

  It is straightforward to deduce the cases $k \ge 3$ of Theorem \ref{basic-cs-structure-theorem}(ii) from the $k=2$ case, but we defer this argument to the more general setting of Section \ref{sec:elementary-v2} (see Theorem \ref{relative-structure-thm} and especially the proof of Lemma \ref{lem:relative cubes}).
\end{proof}

Several adjustments are required to make this argument work in the full case of Theorem \ref{basic-cs-structure-theorem}.  However, we will address these at the same time as stating and proving a version of Theorem \ref{basic-cs-structure-theorem} in greater generality.

Although the statement of Theorem \ref{basic-cs-structure-theorem} refers only to the map $\pi_{s-1} \colon X \to \pi_{s-1}(X)$, we can exploit the fact that $\pi_{s-1}(X)$ is itself a nilspace of degree $(s-1)$ to apply the theorem repeatedly, and obtain a tower
\[
  X \xrightarrow{\pi_{s-1}} \pi_{s-1}(X) \to \dots \xrightarrow{\pi_0} \pi_0(X) = \{\ast\}
\]
as promised.  We write $A_t$ (for $1 \le t \le s$, if $X$ is a nilspace of degree $s$) for the compact abelian groups that arise from the application of Theorem \ref{basic-cs-structure-theorem} to $\pi_t(X)$ at each stage; collectively, these are referred to as the \emph{structure groups} of $X$.

\section{The relative weak structure theory}
\label{sec:elementary-v2}

With these overviews concluded, we will now return to a formal account of the weak structure theory.  In this second pass, we will both give complete proofs, and also introduce a slightly greater degree of generality which will be useful to us in future.

\subsection{Fibrations}

One key definition missing from our initial treatment of cubespaces and nilspaces is the following notion of a \index{fibration}\emph{fibration} \footnote{This term was chosen by analogy with the notion of a Kan fibration in simplicial homotopy theory.  A simplicial set is called a \emph{Kan complex} or \emph{fibrant} if for all $k \ge 1$, any collection of $k$ compatible $(k-1)$-simplices (called a ``horn'') can be completed (or ``filled'') to a $k$-simplex with these as faces, analogously to completion of $k$-corners in a cubespace.  The relative analogue of this for a map between two simplicial sets is known as a \emph{Kan fibration}.  The survey \cite{friedman} is a very approachable introduction to these ideas.  There are other resemblances between these two theories which may not be entirely superficial, but we will not pursue this here.}, which is a particular kind of cubespace morphism.  This can be thought of as a relative version of the notion of the corner completion property.  An alternative heuristic is that fibrations are morphisms which are ``properly'' surjective in the category of cubespaces.

\begin{definition}
  \label{def:fibration}
  Let $f \colon X \to Y$ be a morphism between cubespaces $(X, C^k(X))$ and $(Y, C^k(Y))$.  We say $f$ is a \emph{fibration} if the following holds for any $k \ge 0$.

  Suppose $c \in C^k(Y)$ is a $k$-cube, and $\lambda \colon \{0,1\}^k \setminus \{\vec{1}\} \to X$ is a $k$-corner (see Definition \ref{defn:k-completion}) such that $f(\lambda(\omega)) = c(\omega)$ for all $\omega \in \{0,1\}^k \setminus \{\vec{1}\}$.  Then there exists $x \in X$ such that
  \begin{itemize}
    \item the configuration
      \begin{align*}
        \tilde{c} \colon \{0,1\}^k &\to X \\
                              \omega &\mapsto \begin{cases} \lambda(\omega) &\colon \omega \ne \vec{1} \\ x &\colon \omega = \vec{1} \end{cases}
      \end{align*}
      is in $C^k(X)$; and
    \item $f(x) = c(\vec{1})$ (and hence $f(\tilde{c}) = c$).
  \end{itemize}
\end{definition}

Informally, this definition can be stated as follows: given any $k$-corner $\lambda$ of $X$ and any completion of $f(\lambda)$ to a cube of $Y$, we can complete $\lambda$ to a cube in a compatible fashion.

\begin{remark}
  By taking $k=0$, we find that any fibration is surjective as a map $X \to Y$.  However, this is strictly weaker, as can be seen by taking $X = \cD_1(\RR/\ZZ)$, $Y = \cD_2(\RR/\ZZ)$ and $f \colon \RR/\ZZ \to \RR/\ZZ$ the identity map.  This is a surjective cubespace morphism, but taking $k=2$ and choosing $c$ to be any configuration in $(\RR/\ZZ)^4$ that is not a cube of $C^2(X)$ (but is in $C^2(Y) = (\RR/\ZZ)^4$) it is seen not to be a fibration.
\end{remark}

\begin{remark}
  Let $X$ be a cubespace and consider the unique morphism $f \colon X \to \{\ast\}$, the one-point cubespace.  Then $f$ is a fibration if and only if $X$ has $k$-completion for all $k$.

  Hence, we refer to spaces that have $k$-completion for all $k$ as \emph{fibrant}.\index{fibrant}
\end{remark}

By the previous remark, whenever we prove a statement about fibrant cubespaces, it is reasonable -- and sometimes useful -- to ask for a relative version that holds for fibrations.

\begin{remark}
  \label{rem:fiber-surjective}
  The notion of a fibration is intimately related to that of a \emph{fiber-surjective morphism} which appears in \cite{CS12}*{Section 2.8}.  There, a cubespace morphism $f \colon X \to Y$ is called fiber-surjective if for each $k \ge 0$, the image of any $\sim_k$ class in $X$ is a $\sim_k$ class of $Y$.  It is assumed in this context that both $X$ and $Y$ are nilspaces.

  It is not hard to check that if $X$ and $Y$ are nilspaces and $f \colon X \to Y$ is a cubespace morphism, then $f$ is a fibration if and only if it is fiber-surjective.  If $X$ and $Y$ are general cubespaces, it is not clear how to extend the definition of fiber-surjectivity in general, so no comparison is possible.  If one did extend it to a wider class of cubespaces, we note that the map $X \to \{\ast\}$ will surely be vacuously fiber-surjective for any $X$, whereas it is a fibration if and only if $X$ is fibrant, and hence the notions would be inequivalent whenever $X$ is not assumed to be fibrant.

  Our reasons for working with the definition of a fibration given here rather than with fiber-surjective morphisms are twofold.
  \begin{itemize}
    \item We will have cause later in the project to work with fibrations $f \colon X \to Y$ where $X$ and $Y$ are not nilspaces. Here the distinction matters, and the notion of a fibration is the correct one for what we need.
    \item The fact that the definition of a fibration cleanly extends the completion axiom means that many proofs that work for nilspaces tend to transfer without significant modification to relative versions for fibrations.  In the authors' view, this makes the theory cleaner in places, and suggests the fibration definition is perhaps the more natural one.
  \end{itemize}
  However, we stress that the difference is a technical one, and logically unimportant in almost all cases.
\end{remark}

We will now show a couple of lemmas, together with their proofs, to illustrate this concept in action.

\begin{lemma}[Lifting of ``partial cubes'']
  \label{extension-property}
  Let $f \colon X \to Y$ be a fibration between cubespaces $X$ and $Y$.

  Let $S \subseteq \{0,1\}^k$ be a downwards-closed subset, i.e.~if $\omega \in S$ and $\omega' \subseteq \omega$ then $\omega' \in S$.  Let $T$ be another such subset with $S \subseteq T$.

  Suppose we have configurations $A \colon S \to X$, $B \colon T \to Y$ such that
  \begin{enumerate}
    \item for each $\omega \in S$, we have $f(A(\omega)) = B(\omega)$;
    \item for each $\omega \in S$, the configuration
      \begin{align*}
        \{ \omega' \colon \omega' \subseteq \omega \} &\to X \\
                                              \omega' &\mapsto S(\omega')
      \end{align*}
      is a cube of $X$; and
    \item similarly for $Y$ and $T$.
  \end{enumerate}
  Then we can extend $A$ to a function $\tilde{A} \colon T \to X$ such that $f(\tilde{A}) = B$ and $\tilde{A}$ now has property (ii) with respect to $T$.
\end{lemma}

Here we again used the identification of $\{0,1\}^k$ with the set of subsets of $[k]$ as in Section \ref{sec:hk}.
This rather general statement has the following more natural corollaries.

\begin{corollary}
  If $f \colon X \to Y$ is a fibration then the induced map $f \colon C^k(X) \to C^k(Y)$ is surjective for all $k \ge 0$.
\end{corollary}
\begin{proof}
  Set $S = \emptyset$, $T = \{0,1\}^k$ in the lemma.
\end{proof}

\begin{corollary}
  \label{cor:fibrant-extension-property}
  If $X$ is fibrant, $S \subseteq \{0,1\}^k$ is downward closed and $A \colon S \to X$ is a configuration satisfying property (ii) from the lemma, then $A$ extends to a cube, i.e.~there is some cube $c \in C^k(X)$ with $c|_S = A$.
\end{corollary}
\begin{proof}
  Apply the lemma with $Y = \{\ast\}$, $T = \{0,1\}^k$.
\end{proof}

\begin{proof}[Proof of Lemma \ref{extension-property}]
  If $T = S$ there is nothing to do.  Suppose $\omega_0$ is a minimal element of $T \setminus S$.  If we can extend $A$ to a configuration on $S \cup \{\omega_0\}$ that still satisfies (i) and (ii), we will be done by iterating this process.

  An example of this set-up is shown in the following diagram, where $T = \{0,1\}^3 \setminus \{\vec{1}\}$, $S = \{000,001,010,011,100\}$ and $\omega_0 = 101$:
  \inlinetikz
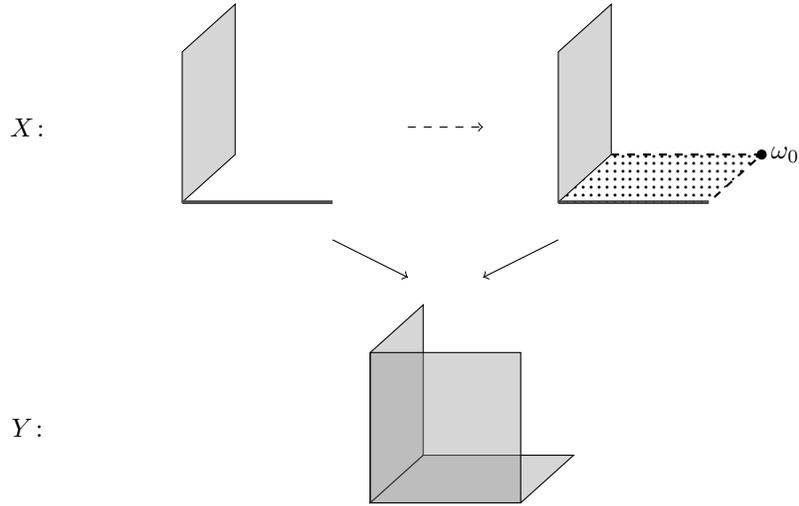
{
    \begin{scope}[shift={(-3,0)}]
      \begin{scope}[x={(1, 0)}, y={(0, 1)}, z={(0.352, 0.317)}, scale=2]
        \draw[fill=black!40,fill opacity=.4] (0,0,0) -- (0,0,1) -- (0,1,1) -- (0,1,0) -- cycle;
        \draw[line width=0.05cm,black!70]  (0,0,0) -- (1,0,0);
      \end{scope}
    \end{scope}
    \draw[->,dashed] (0, 1) -- (1, 1);
    \begin{scope}[shift={(2,0)}]
      \begin{scope}[x={(1, 0)}, y={(0, 1)}, z={(0.352, 0.317)}, scale=2]
        \draw[fill=black!40,fill opacity=.4] (0,0,0) -- (0,0,1) -- (0,1,1) -- (0,1,0) -- cycle;
        \draw[line width=0.05cm,black!70]  (0,0,0) -- (1,0,0);
        \draw[thick,dashed] (0,0,1) -- (1,0,1) -- (1,0,0);
        \draw[pattern=dots,draw=none] (0,0,0) -- (0,0,1) -- (1,0,1) -- (1,0,0) -- cycle;
        \draw[fill=black] (1,0,1) circle[radius=0.03cm] node[right] {$\omega_0$};
      \end{scope}
    \end{scope}
    \draw[->] (-1,-0.5) -- (0, -1);
    \draw[->] (2,-0.5) -- (1, -1);
    \begin{scope}[shift={(-0.5,-4)}]
      \begin{scope}[x={(1, 0)}, y={(0, 1)}, z={(0.352, 0.317)}, scale=2]
        \draw[fill=black!40,fill opacity=.4] (0,0,0) -- (0,0,1) -- (0,1,1) -- (0,1,0) -- cycle;
        \draw[fill=black!40,fill opacity=.4]  (0,0,0) -- (1,0,0) -- (1,0,1) -- (0,0,1) -- cycle;
        \draw[fill=black!40,fill opacity=.4]  (0,0,0) -- (1,0,0) -- (1,1,0) -- (0,1,0) -- cycle;
      \end{scope}
    \end{scope}
    \draw (-5, 1) node {$X\colon$};
    \draw (-5, -3) node {$Y\colon$};
  }

  Let $V = \{ \omega \in \{0,1\}^k \colon \omega \subseteq \omega_0 \}$; clearly this is a discrete sub-cube of dimension $k' = |\omega_0|$ (shown as the dotted square in the diagram, where $k'=2$).  Then $A|_{V \setminus \{\omega_0\}}$ is a $k'$-corner, and furthermore a partial lift  of the corresponding cube $B|_V$ of $Y$.

  By the fibration definition we can choose $x \in X$ such that $f(x) = B(\omega_0)$, and such that setting $A(\omega_0) = x$ we have $A|_V \in C^{k'}(X)$.  Hence this extended $A$ satisfies (i) and (ii) as required.
\end{proof}

We can also use this to give a full proof of Proposition \ref{prop:fibrant-glueing}; recall this stated that a fibrant space satisfies the glueing property.

\begin{proof}[Proof of Proposition \ref{prop:fibrant-glueing}]
  Essentially we formalize the diagram given above when $k=2$; for convenience we reproduce it now.
  \glueingpic
Formally: given $c, c' \in C^k(X)$ with $c(\omega 1) = c'(\omega 0)$ for all $\omega \in \{0,1\}^{k-1}$, let $S = \{0,1\}^{k+1} \setminus \left(\{0,1\}^{k-1} \times \{(1, 1)\}\right)$ and consider the configuration
  \begin{align*}
    d \colon S &\to X \\
                  \omega 0 0 & \mapsto c(\omega 1) = c'(\omega 0) \\
                  \omega 0 1 & \mapsto c'(\omega 1) \\
                  \omega 1 0 & \mapsto c(\omega 0) \ .
  \end{align*}
  One can verify this satisfies the hypotheses of Corollary \ref{cor:fibrant-extension-property}, and so there is a cube $\tilde{d} \in C^{k+1}(X)$ whose restriction to $S$ is $d$.  But then the ``glued'' configuration $c''$ as in Definition \ref{defn:glueing} is a sub-cube of $\tilde{d}$ as shown, and so is in $C^k(X)$ by the cubespace axioms.
\end{proof}

It is clear from the definition that a composition of fibrations is a fibration.
Another useful fact is the following universal property.

\begin{lemma}[``Universal property'']
\label{lem:universal1}
Let $f_{YX}:X\to Y$ and $f_{ZX}:X\to Z$ be fibrations between compact cubespaces.
Suppose that for every $y\in Y$ there is $z\in Z$ such that  $f_{YX}^{-1}(y)\subseteq f_{ZX}^{-1}(z)$.
Then there is a unique fibration $f_{ZY}: Y\to Z$ such that $f_{ZX}=f_{ZY}\circ f_{YX}$.

Equivalently, the following holds.
Let $f:X\to Y$ be a fibration and $g: Y\to Z$ be a map between two compact cubespaces.
If $g\circ f$ is a fibration then so is $g$.
\end{lemma}
\begin{proof}
Since $X$, $Y$ and $Z$ are compact metric spaces, $f_{YX}$ and $f_{ZX}$ are quotient maps.
Hence the map $f_{ZY}$ (which is uniquely defined thanks to the condition imposed on the fibres
of $f_{YX}$ and $f_{ZX}$) is continuous.
It remains to show that it is also a fibration, which is precisely the second statement.

  For the second statement, we first show that $g$ is a cubespace morphism.
To this end, fix a cube $c\in C^k(Y)$ and let $\wt c\in C^k(X)$ be a cube such that $f(\wt c)=c$.
This exists since $f$ is a fibration.
This implies that $g(c)=g\circ f(\wt c)\in C^{k}(Z)$, hence $g$ is indeed a cubespace morphism.

Suppose $\lambda$ is a $k$-corner in $Y$ and $c \in C^k(Z)$ a compatible $k$-cube.  By Lemma \ref{extension-property} on $f$ with $S = \emptyset$, $T = \{0,1\}^k \setminus \{\vec{1}\}$, we can choose $\tilde{\lambda}$ a $k$-corner of $X$ such that $f(\tilde{\lambda}) = \lambda$ and so in particular $g \circ f(\tilde{\lambda})$ is compatible with $c$.

  Since $g \circ f$ is a fibration, we may extend $\tilde{\lambda}$ to a cube $c' \in C^k(X)$ such that $g \circ f(c') = c$.  So, $f(c') \in C^k(Y)$ has the required property.
\end{proof}

\begin{remark}
Taking $Z=\{*\}$ in the lemma, we see that if $f:X\to Y$ is a fibration and $X$ is fibrant then so is $Y$.

It also holds that the property of $s$-uniqueness is inherited by the image of a cubespace under a fibration, hence the image of a nilspace under a fibration is a nilspace.

Indeed, let $f:X\to Y$ be a fibration and suppose that $X$ has $s$-uniqueness.
Suppose that $x\sim_s y$ for some $x,y\in Y$.
We aim to show that $x=y$.
We first note that the configuration $c:\{0,1\}^s\to Y$ defined by
$\omega\mapsto x$ for $\omega\neq\vec1$ and $\vec1\mapsto y$ is a cube in $C^s(Y)$.
Let $\wt x\in f^{-1}(x)$ be arbitrary and let $\lambda:\{0,1\}^s\backslash\{\vec1\}\to X$ be the constant
corner configuration $\omega\mapsto \wt x$.

Since $f$ is a fibration, $\lambda$ can be completed to a cube $\wt c$ such that $f(\wt c(\vec 1))=y$.
Since $X$ has $s$-uniqueness, we have $\wt x=\wt c(\vec 1)$, which implies $x=f(\wt x)=f(\wt c(\vec 1))=y$,
as required.
\end{remark}

We also need to record the correct ``relative'' version of uniqueness.

\begin{definition}
  We say a cubespace morphism $f \colon X \to Y$ has \emph{$k$-uniqueness}\index{relative(ly)!$k$-uniqueness} if the following holds: if $c, c' \in C^k(X)$ are two cubes such that $f(c) = f(c')$ and $c(\omega) = c'(\omega)$ for all $\omega \ne \vec{1}$, then in fact $c = c'$.

  If $f$ is a fibration and $k$ is the smallest number such that $f$ has $k$-uniqueness, we say that $f$ has \emph{degree $(k-1)$}.
\index{fibration!degree}
\end{definition}

Again it is clear that a space $X$ has $k$-uniqueness if and only if the map $X \to \{\ast\}$ does.

\subsection{The structure theory in terms of fibrations}

We will now modify -- and prove -- statements of the weak structure theory discussed above, in ``relative form'', i.e.~in terms of general fibrations.

First we consider the \emph{canonical factors} $\pi_s$.  Everything here is a reasonably straightforward generalization of the corresponding arguments in Section \ref{sec:canonical-factors}.

\begin{definition}
  Let $f \colon X \to Y$ be a fibration, and let $k \ge 0$.  Define an equivalence relation $\sim_{f, s}$\index[nota]{$\sim_{f, s}$} on $X$ as follows: $x \sim_{f, s} x'$ if there exist two $(s+1)$-cubes $c, c'$ in $X$ such that $f(c)=f(c')$, $c(\vec{1}) = x$, $c'(\vec{1}) = x'$ and $c(\omega) = c'(\omega)$ at all other vertices $\omega \in \{0,1\}^{s+1} \setminus \{\vec{1}\}$.
\end{definition}

\begin{proposition}
  \label{relative-canonical-factor}
  Suppose $f \colon X \to Y$ is a morphism between compact cubespaces $X$, $Y$ that have the glueing property.

  Then the relation $\sim_{f, s}$ is a closed equivalence relation.  Moreover a ``universal replacement''\index{relative(ly)!universal replacement property} property holds: if $x \sim_{f, s} x'$, $k \le s+1$ and $c \in C^{k}(X)$ is a cube with $c(\vec{1}) = x$, then the configuration $c'$ given by
  \[
    c'(\omega) = \begin{cases} x' &\colon \omega = \vec{1} \\ c(\omega) &\colon \omega \ne \vec{1} \end{cases}
  \]
  is a cube.

  Now suppose further that $f$ is a fibration.  Writing $\pi  = \pi_{f,s} \colon X \to X / \sim_{f,s}$\index[nota]{$\pi_{f,s}$} for the projection map, we have that $\pi$ is a fibration and $f$ factors as $f \colon X \xrightarrow{\pi} X / \sim_{f, s} \xrightarrow{g} Y$ where $g$ is a fibration of degree at most $s$.
\end{proposition}

\begin{remark}\label{rem:completion for quotient}
  Again we can consider the case $Y = \{\ast\}$ and $f \colon X \to \{\ast\}$ is the unique map.  So, $f$ is a fibration if and only if $X$ is a fibrant cubespace.  In this case, Proposition \ref{relative-canonical-factor} exactly restates Proposition \ref{prop-canonical-factors}.

  Note in particular that we now have a proof that $X / \sim_s$ is fibrant (since $g \colon X / \sim_s \to \{\ast\}$ is a fibration), which was previously omitted.

  For general $X$ and $Y$, since $x \sim_{f,s} y$ precisely if $x \sim_s y$ and $f(x) = f(y)$, the equivalence relation and universal replacement statements in Proposition \ref{relative-canonical-factor} are easily deducible from the non-relative case $Y = \{\ast\}$; i.e.\ these follow from Proposition \ref{prop-canonical-factors}.  However, since we gave only a proof sketch of these parts in the non-relative case, we lose nothing by starting over in full generality.
\end{remark}

Notwithstanding the increased generality, all the key ideas required for this result are contained in the sketches in Section \ref{sec:canonical-factors}.

We will first introduce some symbology for manipulating high-dimensional cubes.

\begin{definition}
  If $c$ and $c'$ are $k$-cubes, we will denote by $[c, c']$\index[nota]{$[c, c']$} the $(k+1)$-configuration
  \[
    \omega \mapsto \begin{cases} c(\omega_1 \dots \omega_k) &\colon \omega_{k+1} = 0 \\ c'(\omega_1 \dots \omega_k) &\colon \omega_{k+1} = 1 \ . \end{cases}
  \]
  Given an element $x \in X$, the notation $\square^k(x)$\index[nota]{$\square^k(x)$} denotes the constant $k$-cube $(\omega \mapsto x)$.  Given $x, y \in X$, we denote by $\llcorner^k(x;y)$\index[nota]{$\llcorner^k(x;y)$} the configuration
  \begin{align*}
    \{0,1\}^k &\to X \\
       \omega &\mapsto \begin{cases} x &\colon \omega \ne \vec{1} \\ y &\colon \omega = \vec{1} \ . \end{cases}
  \end{align*}
\end{definition}

We may combine these pieces of notation freely with each other and also with pictorial representations of cubes.  For instance, the notation $[\llcorner^2(x;y), [\square^1(z), \square^1(w)]]$ is a synonym for
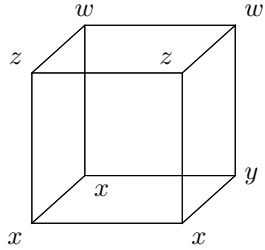
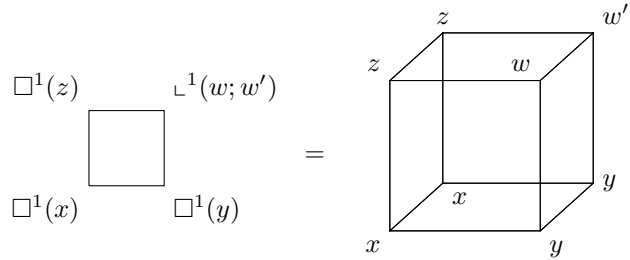
\inlinetikz{
  \threecube{$x$}{$x$}{$x$}{$y$}{$z$}{$z$}{$w$}{$w$}{1}
}
and
\inlinetikz{
  \begin{scope}[shift={(-3,0.6)}]
    \singlesquare{$\square^1(x)$} {$\square^1(y)$} {$\square^1(z)$} {$\llcorner^1(w;w')$}
  \end{scope}
  \draw (0,1) node {$=$};
  \begin{scope}[shift={(1,0)}]
    \threecube{$x$}{$y$}{$x$}{$y$}{$z$}{$w$}{$z$}{$w'$}{1}
  \end{scope}
}
although we will rarely need to denote anything quite so unpleasant.

We will be generally cavalier about the ordering of the indices $\{1, \dots, k\}$ implied by successive applications of this notation, since it should always be clear from context what is meant, and because (thanks to the axioms) it makes little or no difference most of the time.

Finally, we will need to introduce some notation about \emph{tri-cubes} of the form appearing in the proof of Proposition \ref{prop-canonical-factors}, i.e.~a collection of $2^k$ $k$-cubes that all glue together at the middle. These are very useful gadgets introduced by Antol\'\i n Camarena and Szegedy in \cite{CS12}*{Section 2.3}.

\begin{definition}
  We say a collection $t = (t_\nu)_{\nu \in \{0,1\}^k}$\index[nota]{$(t_\nu)_{\nu \in \{0,1\}^k}$} constitutes a \emph{tri-cube}\index{tri-cube} if $t_\nu \in C^k(X)$ for each $\nu$, and if for each $\nu, \nu', \omega \in \{0,1\}^k$ such that for each $i \in \{1,\dots,k\}$ either $\nu_i = \nu'_i$ or $\omega_i = 1$, we have $t_\nu(\omega) = t_{\nu'}(\omega)$.

  We say the \emph{outer cube}\index{outer cube} of $t$ is the configuration $\omega \mapsto (t_\omega(\vec{0}))$.
\end{definition}
Note this corresponds to a picture (when $k=2$):
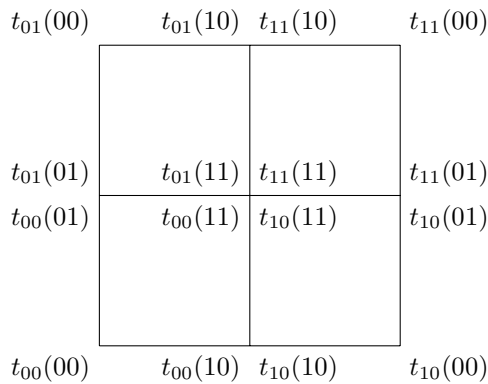
\inlinetikz{
  \begin{scope}[scale=2]
    \draw[thin,black] (0,0) grid (2, 2);
    \node [below left] at (0, 0) {$t_{00}(00)$};
    \node [above left] at (0, 2) {$t_{01}(00)$};
    \node [below right] at (2, 0) {$t_{10}(00)$};
    \node [above right] at (2, 2) {$t_{11}(00)$};
    \node [below left] at (0, 1) {$t_{00}(01)$};
    \node [above left] at (0, 1) {$t_{01}(01)$};
    \node [below right] at (2, 1) {$t_{10}(01)$};
    \node [above right] at (2, 1) {$t_{11}(01)$};
    \node [below left] at (1, 0) {$t_{00}(10)$};
    \node [below right] at (1, 0) {$t_{10}(10)$};
    \node [above left] at (1, 2) {$t_{01}(10)$};
    \node [above right] at (1, 2) {$t_{11}(10)$};
    \node [below left] at (1, 1) {$t_{00}(11)$};
    \node [below right] at (1, 1) {$t_{10}(11)$};
    \node [above left] at (1, 1) {$t_{01}(11)$};
    \node [above right] at (1, 1) {$t_{11}(11)$};
  \end{scope}
}
and hence some of the $t_\nu$ are reflected relative to the standard orientation.

\begin{proposition}
  If $X$ is a cubespace with the glueing property and $(t_\nu)$ is a tri-cube in $X$, then the outer cube is in $C^k(X)$.
\end{proposition}
\begin{proof}
  This follows from repeated application of the glueing property.  By assumption we can glue together each pair $t_{\eta 0}$ and $t_{\eta 1}$ for $\eta \in \{0,1\}^{k-1}$, and then repeat for each of the other $(k-1)$ coordinates in turn.
\end{proof}

\begin{proof}[Proof of Proposition \ref{relative-canonical-factor}]
  The proof is strongly analogous to that from Section \ref{sec:canonical-factors}.  We first prove an analogue of Lemma \ref{lem:canonical-characterization}.

  \begin{lemma}
    \label{lem:relative-canoncial-characterization}
    We have that $x \sim_{f,s} y$ if and only if $f(x) = f(y)$ and also $\llcorner^{s+1}(x; y) \in C^{s+1}(X)$.
  \end{lemma}
  \begin{proof}
The strategy is to construct a tri-cube whose outer cube is $\llcorner^{s+1}(x; y)$.

    By assumption we are given cubes $c, c' \in C^{s+1}(X)$ such that $f(c) = f(c') = d$, $c(\omega) = c'(\omega)$ for all $\omega \ne \vec{0}$ and $c(\vec{0}) = x$, $c'(\vec{0}) = y$.  (Note we have switched $\vec{1}$ for $\vec{0}$ for convenience.)  From this it is immediate that $f(x) = f(y)$.

    Now define the family $(t_\nu)_{\nu \in \{0,1\}^{s+1}}$ by
    \[
      t_\nu = \begin{cases} c &\colon \nu \ne \vec{1} \\ c' &\colon \nu = \vec{1} \ . \end{cases}
    \]
    It is clear from the hypotheses that this is a tri-cube with outer cube $\llcorner^{s+1}(x; y)$ as required.
  \end{proof}

  \begin{claim*}
    The relation $\sim_{f,s}$ is an equivalence relation.
  \end{claim*}
  \begin{proof}[Proof of claim]
    Again, symmetry and reflexivity are immediate.  For transitivity, assume $x \sim_{f,s} y \sim_{f,s} z$.  Then by the lemma, $f(x) = f(y) = f(z)$ and $\llcorner^{s+1}(y; x)$, $\llcorner^{s+1}(y; z)$ are cubes.  By glueing along the common face $\square^s(y)$, we find that $\llcorner^s(\square^1(y); [x,z])$ is an $(s+1)$-cube.  But so is $\llcorner^s(\square^1(y); [x,x])$, by restricting to a face of $\llcorner^{s+1}(y; x)$ and then duplicating.  Since these latter two cubes have the same image under $f$, we have $x \sim_{f,s} z$ by the definition.
  \end{proof}

  Again, the fact that this relation is closed is straightforward from the definition, the hypothesis that $C^{s+1}(X)$ is closed, continuity of $f$ and compactness.

  We turn to the ``universal replacement'' statement.
  \begin{claim*}
    If $x \sim_{f, s} x'$, and $c \in C^{s+1}(X)$ is a cube with $c(\vec{1}) = x$, then the configuration $c'$ given by
    \[
      c'(\omega) = \begin{cases} x' &\colon \omega = \vec{1} \\ c(\omega) &\colon \omega \ne \vec{1} \end{cases}
    \]
    is a cube.
  \end{claim*}
  As in the proof of Proposition \ref{prop-canonical-factors}, this easily implies the same statement for all $k < s+1$.
  \begin{proof}[Proof of claim]
We construct a tri-cube whose outer cube is $c'$.  We set
    \[
      t_\nu(\omega) = \begin{cases} c((\min(\nu_i + \omega_i, 1))_{i=1}^{s+1}) &\colon \nu \ne \vec{1} \\ x &\colon \nu = \vec{1}, \omega \ne \vec{0} \\ x' &\colon \nu = \vec{1}, \omega = \vec{0} \ . \end{cases}
    \]
    which is a very elaborate way of stating the generalization of the diagram from the proof of Proposition \ref{prop-canonical-factors}.  Now, $t_{\vec{0}}$ is just $c$, the other $t_\nu$ except $t_{\vec{1}}$ are obtained from $c$ by duplication of an upper face, and $t_{\vec{1}}$ is a reflected copy of $\llcorner^{s+1}(x; x')$.  So, all the $t_\nu$ are cubes, and a tedious check verifies that $(t_\nu)$ is a tri-cube with outer cube $c'$ as required.
  \end{proof}

  So, we have a projection $\pi_{f,s} \colon X \to X / \sim_{f,s}$, and by construction $f$ factors through $\pi_{f,s}$, i.e.~$f = g \circ \pi_{f,s}$ for some well-defined function $g \colon X / \sim_{f,s} \to Y$.

  \begin{claim*}
  The map $g$ has $(s+1)$-uniqueness.
  \end{claim*}
  \begin{proof}[Proof of claim]
    Indeed, given two $(s+1)$-cubes $c, c'$ of $X / \sim_{f,s}$ with $c(\omega) = c'(\omega)$ for all $\omega \ne \vec{1}$ and $g(c) = g(c')$, we know there are cubes $\tilde{c}, \tilde{c}'$ of $X$ such that $\pi_{f,s}(\tilde{c}) = c$, $\pi_{f,s}(\tilde{c}') = c'$.  By repeated application of universal replacement we may assume $\tilde{c}(\omega) = \tilde{c'}(\omega)$ for each $\omega \ne \vec{1}$.  But then $\tilde{c}(\vec{1}) \sim_{f,s} \tilde{c}'(\vec{1})$ by the definition of $\sim_{f,s}$, and so $c = c'$.
  \end{proof}

  \begin{claim*}
    The map $\pi_{f,s}$ is a fibration.
  \end{claim*}
  \begin{proof}[Proof of claim]
    Suppose we are given a $k$-corner $\lambda$ in $X$ and a compatible $c \in C^k(X) / \sim_{f,s}$.

    First consider the case $k < s+1$.  By definition there is a cube $\tilde{c}$ of $X$ with $\pi_{f,s}(\tilde{c}) = c$, and furthermore by repeated application of universal replacement, \emph{any} configuration $\tilde{c} \colon \{0,1\}^k \to X$ with $\pi_{f,s} \circ \tilde{c} = c$ is a cube.  Hence taking $\lambda(\vec{1})$ to be an arbitrary point of $\pi_{f,s}^{-1}(c(\vec{1}))$ must work.

    If $k \ge s+1$, we may use the fact that $f$ is a fibration to complete $\lambda$ to a cube $\tilde{c}$ such that $f(\tilde{c}) = g(c)$.  But now $\pi_{f,s}(\tilde{c})$ and $c$ are two cubes of $X / \sim_{f,s}$ whose image under $g$ is the same, and which are equal except possibly at the vertex $\vec{1}$.  Now we may restrict $\pi_{f,s}(\tilde{c})$ and $c$ to an upper face $F$ of $\{0,1\}^k$ of dimension $(s+1)$ (if necessary) and use $(s+1)$-uniqueness of $g$ to deduce that $\pi_{f,s}(\tilde{c}(\vec1)) = c(\vec1)$, and the claim follows.
  \end{proof}

  We now note that $g$ is therefore a fibration by the universal property (Lemma \ref{lem:universal1}).  This completes the proof of Proposition \ref{relative-canonical-factor}.
\end{proof}

\subsection{The (relative) structure groups}

With the theory of relative canonical factors in place, we are in a position to attack the relative analogue of Theorem \ref{basic-cs-structure-theorem}.  First, we need one further ``relative'' definition.

\begin{definition}
  A morphism $f \colon X \to Y$ is called \emph{(relatively) $k$-ergodic} if for any $1 \le \ell \le k$ and $c \in C^\ell(Y)$, any configuration in $f^{-1}(c)$ is a cube of $X$.

  We say it is \emph{relatively ergodic}\index{relative(ly)!ergodic} if it is relatively $1$-ergodic.
\end{definition}

The direct analogue of Theorem \ref{basic-cs-structure-theorem} is as follows.

\begin{theorem}
  \label{relative-structure-thm}
  Suppose $X$ and $Y$ are compact ergodic cubespaces that obey the glueing condition.  Let $f \colon X \to Y$ be a relatively $s$-ergodic fibration of degree at most $s$.
  Then there exists a compact Abelian group $A=A(f)$\index[nota]{$A(f)$} acting continuously on $X$ such that:
  \begin{enumerate}
    \item the action of $A$ on $X$ is free, and its orbits are precisely the fibers of $f$;
    \item this induces a (free) pointwise action of $C^k(\cD_s(A))$ on $C^k(X)$, whose orbits are precisely the fibers of the map $f \colon C^k(X) \to C^k(Y)$.
  \end{enumerate}
\end{theorem}

We will see that the proof is constructive and $A$ is defined canonically in terms of the cube structures and $f$.

The key point about these hypotheses is that they hold whenever $f$ is the canonical projection map $\pi_{g,s-1} \colon \pi_{g,s}(X) \to \pi_{g,s-1}(X)$ defined with respect to some other fibration $g \colon X \to Y$.  In this case, we write $A = A_s(g)$, the ``$s$-th structure group of the fibration''.  In this way we can iterate this theorem to obtain:

\begin{corollary}
  \label{relative-tower}
  Let $f\colon X \to Y$ be a fibration of degree at most $s$ between compact ergodic cubespaces $X$, $Y$ that obey the glueing condition.  Then we have a tower
  \[
    X = \pi_{f,s}(X) \xrightarrow{\pi_{f,s-1}} \pi_{f,s-1}(X) \to \dots \xrightarrow{\pi_{f,0}} \pi_{f,0}(X) = Y
  \]
  where the fibers of $\pi_{f,i-1}$ are identified with the compact abelian group $A_i(f)$ for each $i$, in the sense of Theorem \ref{relative-structure-thm}.
\end{corollary}
\begin{proof}[Proof of Corollary \ref{relative-tower} from Theorem \ref{relative-structure-thm}]
  It suffices to know that each map $\pi_{f,t} \colon \pi_{f, t+1}(X) \to \pi_{f, t}(X)$ is well-defined, a fibration, has $(t+2)$-uniqueness and is $(t+1)$-ergodic.

  Well-definedness follows from the nested nature of the relations $\sim_{f,t}$; the proof of this is unchanged from Proposition \ref{equiv-nested}.

  The fact that it is a fibration follows from the fact that $\pi_{f,t+1} \colon X \to \pi_{f,t+1}(X)$ and $\pi_{f,t} \colon X \to \pi_{f,t}(X)$ are fibrations, and the universal property (Lemma \ref{lem:universal1}).

  The $(t+2)$-uniqueness statement is supplied by Proposition \ref{relative-canonical-factor} applied to $\pi_{f,t+1}(X)$.

  Finally, the $(t+1)$-ergodicity statement is precisely what is encoded by the ``universal replacement'' property from that same proposition as long as $t\geq 1$. For $t=0$ it does not follow from Proposition \ref{relative-canonical-factor} (indeed in general it does not hold). Instead for $t=0$ we use that assumption that $X$ is ergodic.
\end{proof}

Before embarking on the proof of Theorem \ref{relative-structure-thm}, we note a relative analogue of Proposition \ref{prop:high-cubes-boring} that will be useful for discussing high-dimensional cubes.

\begin{proposition}
  \label{prop:relative-boring-high-dimension}
  Let $s \ge 1$ be an integer, $f \colon X \to Y$ a fibration of degree at most $s$ and $k \ge s +1$.  Further let $c \colon X \to \{0,1\}^k$ be given.  Then $c \in C^k(X)$ if and only if $f(c)$ is a cube of $Y$ and further every face of $c$ of dimension $(s+1)$ is a cube of $X$.
\end{proposition}
\begin{proof}
  The ``only if'' part is clear.  For the converse, again we argue by induction on $k$, the case $k=s+1$ being trivial.  Given $c$, we consider the restriction to $\{0,1\}^k \setminus \{\vec{1}\}$, which by inductive hypothesis is a $k$-corner.  Hence we may complete this corner relative to $f(c)$; i.e.~there is a cube $c'$ such that $c'(\omega) = c(\omega)$ for all $\omega \ne \vec{1}$ and $f(c'(\vec{1})) = f(c(\vec{1}))$.  Now by restricting to any upper face of $c$ and $c'$ of dimension $(s+1)$ and invoking $(s+1)$-uniqueness, we find $c(\vec{1}) = c'(\vec{1})$ also and hence $c = c'$ is a cube.
\end{proof}

The remainder of this section is spent proving Theorem \ref{relative-structure-thm}.

\begin{proof}[Proof of Theorem \ref{relative-structure-thm}]
  Again, this account is very similar to the proof of Theorem \ref{basic-cs-structure-theorem}, and the reader will find it significantly more comprehensible if they are already familiar with the sketch proof of that theorem given above in Section \ref{sec:structure-groups}.

  Before, we defined the structure group in terms of equivalence classes of edges in $C^1(X)$, under the equivalence relation $[x,y] \sim [z,w]$ if and only if $[[x,y],[z,w]] \in C^2(X)$.

  One way to rationalize this construction is in terms of the \emph{edge cubespace} $\cE(X)$, which we define now.

  \begin{definition}  \label{defn:edge-cubespace}
    Let $X$ be a cubespace.  The \emph{edge cubespace}\index{edge cubespace} of $X$, denoted $\cE(X)$\index[nota]{$\cE(X)$}, is the cubespace whose base space consists of the edges $C^1(X) \subseteq X \times X$, and whose $k$-cubes are the configurations $c \in (X \times X)^{\{0,1\}^k} \cong X^{\{0,1\}^{k + 1}}$ that correspond to $(k+1)$-cubes of $X$.
  \end{definition}

  It is easy to check that this does indeed define a cubespace.  Moreover, if $X$ has the glueing property, this is inherited by $\cE(X)$.  It follows that we may consider the canonical equivalence relations $\sim_{k}$ on $\cE(X)$; by Proposition \ref{relative-canonical-factor} (in the case $Y = \{\ast\}$; see Remark \ref{rem:completion for quotient}) these are equivalence relations and have the universal replacement property.

  It follows that $[x,y] \sim [z,w]$ as above if and only if $[[x,y],[z,w]]$ is an edge of $\cE(X)$, hence if and only if $[x,y] \sim_0 [z,w]$ where $\sim_0$ denotes the canonical equivalence relation on $\cE(X)$.

  In order to generalize the proof of Theorem \ref{basic-cs-structure-theorem} to $s > 1$, it is therefore tempting to try replacing $\sim_0$ by $\sim_{s-1}$ in the above and attempting to re-run the argument.  A little further thought shows this is not quite the right thing to do, since
  \begin{itemize}
    \item the space $\cE(X)$ is too large when $s > 1$: we want a pair to represent the element of the structure group that takes one to the other, but this only makes sense for pairs $[x,y]$ that lie in the same fiber of $f$;
    \item the condition $[x,y] \sim_{s-1} [x',y']$ turns out to be too restrictive on $Y$, since it implies that $x$ and $x'$ lie in the same fiber of $f$, and hence we cannot relate pairs lying over different fibers.
  \end{itemize}

  However, a second attempt at this does work, and generalizes Theorem \ref{basic-cs-structure-theorem} fairly cleanly.

  We define
  \[
   \cM= \cM_f(X) = \{ [x,y] \in X\times X \colon f(x) = f(y) \}
  \]
  i.e.~the space of pairs that lie in the same fiber of $f$, and specify an equivalence relation $\sim$ on $\cM$ by $[x,x'] \sim [y, y']$ if and only if $[x,y] \sim_{s-1} [x',y']$ in $\cE(X)$, where $\sim_{s-1}$ denotes the canonical equivalence relation on that space.  Note that this definition involves a ``$90^\circ$ rotation'', e.g.~pairing up $[x,y]$ rather than $[x,x']$; this is not an error and is necessary for the definition to make sense.

  Note that, by Lemma \ref{lem:relative-canoncial-characterization}, this is equivalent to saying $[x,x'] \sim [y,y']$ if and only if $[\llcorner^s(x;x'), \llcorner^s(y;y')]$ is a cube of $X$.

  We will define the structure group $A = A_f$ to be $\cM / \sim$.  As above, addition on $A$ will be given by concatenation of edges, negation by $[x,y] \mapsto [y,x]$ and the action on $X$ by a cube completion.

  We now turn to the details.

  \begin{claim*}
    The relation $\sim$ is a closed equivalence relation on $\cM$.
  \end{claim*}
  \begin{proof}[Proof of claim]
    For reflexivity we are asked to verify that for all $[x, x'] \in \cM$ the configuration $[\llcorner^s(x; x'), \llcorner^s(x; x')]$ is an $(s+1)$-cube of $X$; equivalently by duplication, that $\llcorner^s(x; x')$ is an $s$-cube.  Since $f(x) = f(x')$ this is immediate from the assumption of relative $s$-ergodicity.

    Symmetry is clear: if $[\llcorner^s(x; x'), \llcorner^s(y; y')]$ is a cube then so is $[\llcorner^s(y; y'), \llcorner^k(x; x')]$.

    For transitivity, suppose $[x, x'], [y, y'], [z, z'] \in \cM$ and $[x, x'] \sim [y, y']$, $[y, y'] \sim [z, z']$.  Equivalently, $[\llcorner^s(x; x'), \llcorner^s(y; y')]$ and $[\llcorner^s(y; y'), \llcorner^s(z; z')]$ are cubes.  It is immediate by glueing that $[\llcorner^s(x;x'), \llcorner^s(z;z')]$ is a cube, and hence $[x,x'] \sim [z,z']$ as required.

    Once again, the topological properties are an easy check.
  \end{proof}

  Hence, as a set, $A = \cM / \sim$ is well-defined.  As in the proof of Theorem \ref{basic-cs-structure-theorem}, we have considerable freedom to choose representatives for $a \in A$.
  \begin{claim*}
    Given $a \in A$ and $x \in X$, there exists an unique $x' \in X$ such that $a \sim [x,x']$.
  \end{claim*}
  \begin{proof}[Proof of claim]
    Suppose $[y,y']$ is any representative of $a$.  Necessarily $f(y) = f(y')$.  Hence the configuration $[\llcorner^s(y; y'), \llcorner^s(x; -)]$ is an $(s+1)$-corner of $X$ lying above the cube $[\square^s(f(y)), \square^s(f(x))]$ of $Y$ (which is a cube by ergodicity of $Y$ and duplication): the fact that it is a corner is trivial by $s$-ergodicity of $f$.  Since $f$ is a fibration of degree at most $s$, there is an unique $x' \in f^{-1}(f(x))$ completing this $(s+1)$-corner to a cube of $X$, as required.
  \end{proof}

  Hence we may now define addition on $A$ by concatenation of edges as before.  Given $a, b \in A$ we may fix $e \in X$ and choose representatives $[x, e] \sim a$, $[e, y] \sim b$ and define $a + b$ to be the class of $[x,y]$.  The identity is the class of $[e,e]$ for any $e$, and negation is done by sending $[x,y] \mapsto [y,x]$.

  \begin{claim*}
    These operations are well-defined, and give $A$ the structure of a (topological) abelian group.
  \end{claim*}
  \begin{proof}[Proof of claim]
    We first check well-definedness: that is, that these operations do not depend on the choice of representatives, or -- in the case of addition -- on the choice of $e$.

    It is clear from the definition of $\sim$ (in terms of $\sim_s$ in $\cE(X)$) that $[x,x] \sim [y,y]$ for any $x,y \in X$ and hence the identity class really is a class.  Similarly, if $[x,y] \sim [x',y']$ then $[y,x] \sim [y',x']$ and hence negation is well-defined.

    Now suppose $[x,y] \sim [x',y']$ and $[y,z] \sim [y',z']$; we wish to show $[x,z] \sim [x',z']$ as this will show addition is well-defined.  But this says $[x,x'] \sim_{s-1} [y,y']$ and $[y,y'] \sim_{s-1} [z,z']$ in $\cE(X)$, and hence $[x,x'] \sim_{s-1} [z,z']$ in $\cE(X)$ by transitivity and so $[x,z] \sim [x',z']$ as required.

    Associativity of addition is now clear from associativity of concatenation. Similarly, that the identity and inverses behave as they should is straightforward.

    We now argue commutativity.  We will first need a lemma.
    \begin{lemma}
      If $x,y,z,w \in X$ all lie in the same fiber of $f$ and $[x,y] \sim [z,w]$ then $[x,z] \sim [y,w]$.
    \end{lemma}
    \begin{proof}[Proof of lemma]
      Note if $s=1$ this is trivial: both are equivalent to
      \inlinetikz{\singlesquare{$x$}{$y$}{$z$}{$w$}}
      being a cube.

 For $s>1$, we argue as follows.
Choose $u$ in the same fiber arbitrarily.
By relative $s$-ergodicity,
   \inlinetikz{\singlesquare{$\llcorner^{s-1}(u;x)$}{$\llcorner^{s-1}(u;x)$}{$\llcorner^{s-1}(u;z)$}{$\llcorner^{s-1}(u;z)$}}
      is an $(s+1)$-cube in $X$, or equivalently an $s$-cube in $\cE(X)$.

      Note that $[x,y]\sim[z,w]$ if and only if $[x,z]\sim_{s-1}[y,w]$ in $\cE(X)$.  In the previous diagram, interpreted as an $s$-cube in $\cE(X)$, the edge $[x,z]$ appears twice.  If we replace one of these edges by $[y,w]$, giving
      \inlinetikz{\singlesquare{$\llcorner^{s-1}(u;x)$}{$\llcorner^{s-1}(u;y)$}{$\llcorner^{s-1}(u;z)$}{$\llcorner^{s-1}(u;w)$}}
      then this lies in $C^s(\cE(X))$ if and only if $[x,z] \sim [y,w]$.  Indeed, ``only if'' follows from the definition of $\sim_{s-1}$ on $\cE(X)$, and for ``if'' we apply the universal replacement property for $\cE(X)$ (which has the glueing property, as was mentioned after Definition \ref{defn:edge-cubespace}).

      However, thinking of this configuration now as an element of $X^{\{0,1\}^{s+1}}$, it is unchanged, up to a morphism of discrete cubes, on exchanging $y$ and $z$.  Hence, we see that it is a cube in $C^{s+1}(X)$ if and only if $[x,z]\sim[y,w]$, and also if and only if $[x,y] \sim [z,w]$, as required.
\end{proof}

    So, now suppose we have $a, b \in A$, and pick representatives $a \sim [x,e]$, $b \sim [e,y]$ and hence $a+b \sim [x,y]$.  Further pick $e'$ such that $[x,e'] \sim b$.  Note $x,y,e,e'$ all lie in the same fiber and $[x,e'] \sim [e,y]$, so by the lemma $[x,e] \sim [e',y]$.  But then by concatenation,
    \[
      b + a \sim [x,e'] + [e',y] \sim [x,y] \sim a + b
    \]
    as required.

    Once again, arguing continuity of these operations is straightforward from closedness of the equivalence relations, closedness of $C^k(X)$, continuity of $f$ and compactness.
This concludes the proof of the claim.
  \end{proof}

  We now define the group action of $A$ on $X$.  Given $a \in A$ and $x \in X$ we take $a(x)$ (also notated $a . x$) to be the unique element of $X$ such that $a \sim [x, a(x)]$.  It is clear $0 \in A$ acts trivially.  To confirm it is an action, we consider that $a \sim [x, a(x)]$ and $a' \sim [a(x), a'(a(x))]$ but then by definition of addition $a+a' \sim [x, a'(a(x))]$.

  Moreover, it is clear that this action respects $f$ (i.e.~$x$ and $a(x)$ lie in the same fiber), and that it is simply transitive on fibers of $f$.  Again, continuity of the action is straightforward.  Hence we have shown part (i) of Theorem \ref{relative-structure-thm}.

  We now turn to part (ii), which describes the cubes $C^k(X)$ in terms of $C^k(Y)$ and $A$.  Recall that we wanted to show that the action of $A$ on $X$ induces a pointwise action of $C^k(\cD_s(A))$ on $C^k(X)$ whose orbits are precisely the fibers of $f \colon C^k(X) \to C^k(Y)$.  This can be rephrased as follows.

  \begin{lemma}\label{lem:relative cubes}
    Suppose $k\ge 0$, $c \in C^k(X)$ and $c' \colon \{0,1\}^k \to X$ is another configuration such that $f(c) = f(c')$.  Define a configuration $a \colon \{0,1\}^k \to A$ by
    \[
      a(\omega) \sim [c(\omega), c'(\omega)] \ .
    \]
    Then $c'$ is a cube if and only if $a \in C^{k}(\cD_s(A))$.
  \end{lemma}

  Recall again that the cubespace $\cD_s(A)$ was defined by the Host--Kra construction applied to the compact abelian group $A$ with the filtration $A_0 = \dots = A_s \supseteq A_{s+1} = \{0\}$.  We note the following features:
  \begin{itemize}
    \item it is $s$-ergodic, i.e.~every element of $A^{\{0,1\}^k}$ is a cube for $0 \le k \le s$; and
    \item it is a nilspace of degree $s$.
  \end{itemize}

  \begin{proof}[Proof of Lemma \ref{lem:relative cubes}]
    We will consider three cases.
    \setcounter{case}{0}
    \begin{case}
      $0 \le k \le s$.  By relative $s$-ergodicity, we have that $c'$ is always a cube of $X$.  But by $s$-ergodicity of $\cD_s(A)$, $a$ is always a $k$-cube, so this is consistent.
    \end{case}
    \begin{case}
      $k = s+1$.  For fixed $c$, let $T \subseteq A^{\{0,1\}^{s+1}}$ denote the set of all elements $a \colon \{0,1\}^{s+1} \to A$ such that $c' = a . c$ is a cube; so our goal is to show $T = C^{s+1}(\cD_s(A))$.  Then we claim
    \begin{enumerate}
      \item for $a \in T$, and for $F \subseteq \{0,1\}^{s+1}$ a face of dimension $1$ and $b \in A$, we have $a + [b]_F \in T$; \footnote{Recall (Definition \ref{def:cube-group}) that $[b]_F$ denotes a configuration that is equal to $b$ on $F$ and zero elsewhere.}
      \item if $a \in T$ and $a(\omega) = 0$ for all $\omega \ne \vec{1}$, then $a = 0$.
      \end{enumerate}
      To see (i), note that we can think of any cube $c'$ of dimension $(s+1)$ as a cube of dimension $s$ in the edge cubespace $\cE(X)$.  Let $[x,y]$ denote the edge corresponding to $F$ in $c'$.  Then $[x,b.x] \sim [y,b.y] \sim b$ in $\cM$ and so $[x,y] \sim_{s-1} [b.x,b.y]$ in $\cE$.  By the universal replacement property of $\sim_{s-1}$ (Proposition \ref{relative-canonical-factor}) applied to $\cE(X)$ (which certainly has the glueing property), replacing the vertex $[x,y]$ of the $s$-cube $c'$ of $\cE(X)$ by $[b.x,b.y]$ yields another $s$-cube, as required.

      Part (ii) is immediate from relative $(s+1)$-uniqueness.

      Now, from (i) it follows that $T$ is a union of cosets of $C^{s+1}(\cD_s(A))$, since such $[b]_F$ generate $C^{s+1}(\cD_s(A))$.  In particular this includes $C^{s+1}(\cD_s(A))$ itself as trivially $0 \in T$.  Now if $a \in T$, by completing the $(s+1)$-corner $a|_{\{0,1\}\setminus \{\vec{1}\}}$ in the nilspace $\cD_s(A)$, we can find an element $a' \in \cD_s(A)$ such that $a'(\omega) = a(\omega)$ for all $\omega \ne \vec{1}$.  So $a - a' \in T$; but by (ii) this means $a = a'$ and hence $a \in C^{s+1}(\cD_s(A))$.
    \end{case}
    \begin{case}
      $k > s+1$.  Given $c$, $c'$, $a$ as in the statement, we note that $c'$ is a cube if and only if every face of dimension $(s+1)$ is a cube (by Proposition \ref{prop:relative-boring-high-dimension}), if and only if every face of $a$ of dimension $(s+1)$ is a cube (by the previous case), if and only if $a$ is a cube of $\cD_s(A)$ (by Proposition \ref{prop:high-cubes-boring}).
    \end{case}
  \end{proof}
This concludes the proof of Theorem \ref{relative-structure-thm}.
\end{proof}

\appendix

\section{An extended exposition of Host--Kra cubes}
\label{app:hk}

Here, we provide a fairly detailed exposition of the theory of the Host--Kra construction of cubes (or if you prefer, parallelepipeds) in nilpotent groups and nilmanifolds.  Much of this is modelled closely on \cite{GT10}*{Appendix E}; however, we provide some further results and closely worked examples.

\subsection{Some examples of filtered groups}

We recall from Section \ref{sec:hk} the definition of a filtered group.  The following are some examples that will be referred to several times in what follows.

\begin{example}
  \label{ex-A}
  The simplest case is something like $G = \RR$ with $G_0 = G_1 = \RR$ and $G_i = \{0\}$ for $i \ge 2$.  This corresponds to the lower central series filtration on the abelian group $\RR$.
\end{example}

\begin{example}
  \label{ex-B}
  Again take $G = \RR$ and $G_0 = G_1 = G_2 = \RR$ and $G_i = \{0\}$ for $i \ge 3$.  This is now a (somewhat trivial) degree $2$ filtration on $\RR$.
\end{example}

\begin{example}
  \label{ex-C}
  Take $G = \RR^2$, and $G_0 = G_1 = \RR^2$, $G_2 = \{0\} \times \RR$ and $G_i = \{0\}$ for $i \ge 3$.  This is again a degree $2$ filtration on an abelian group.
\end{example}

\begin{example}
  \label{ex-D}
  For a non-commutative case, we consider the Heisenberg group $\cH$, defined as the group of upper diagonal $3 \times 3$ matrices with ones on the diagonal, i.e.
  \[
    \cH = \left\{ \heis{x}{y}{z} \colon x,y,z \in \RR \right\} \ .
  \]
  We give this its lower central series filtration.  Note that
  \[
    [\cH,\cH] = Z(\cH) = \left\{ \heis{0}{0}{z} \colon z \in \RR \right\}
  \]
  and so $\cH_0 = \cH_1 = \cH$, $\cH_2 = Z(\cH)$ as above and $\cH_i = \{\id\}$ for $i \ge 3$.  So again we have a filtration of degree $2$.
\end{example}

\subsection{More on $\HK^k(G_\bullet)$}

The definition of $\HK^k(G_\bullet)$ in Section \ref{sec:hk} is fairly natural, but at the same time computationally unhelpful for some purposes.  There, we specify $\HK^k(G_\bullet)$ by giving an explicit generating set as a subgroup of $G^{\{0,1\}^k}$; but because we could in principle take arbitrarily long products in these generators, this does not give an algorithm for determining whether a given configuration $\{0,1\}^k \to G$ lies in the Host--Kra cube group or not.

It turns out that both these problems can be solved fairly straightforwardly: we can obtain a completely explicit description of $\HK^k(G_\bullet)$ in terms of bounded length products, and at the same time obtain such an algorithm.

We recall some notation.
If $S\subseteq [k]$, we write $F_S$ for the face of $\{0,1\}^k$ that consists of the vertices $\omega$ with
$S\subseteq\omega$.
If $x\in G$ and $F\subseteq \{0,1\}^k$, we write $[x]_F$ for the element of
$G^{\{0,1\}^k}$ defined by $\omega\mapsto x$ if $\omega\in F$
and $\omega\mapsto \id$ otherwise.

\begin{proposition}
  \label{prop:hk-decomposition}
  Take any integer $k \ge 0$, and fix an ordering $S_1=\emptyset, \dots, S_{2^k}=\{0,1\}^k$ for the subsets of $[k]$ that respects inclusion, i.e.~if $S_i \subseteq S_j$ then $i \le j$.  Then any configuration $(g_\omega)_{\omega \in \{0,1\}^k} \in G^{\{0,1\}^k}$ has a unique representation as an ordered product
  \[
    \prod_{i=1}^{2^k} [x_i]_{F_{S_i}}=[x_1]_{F_{S_1}}\cdots[x_{2^k}]_{F_{S_{2^k}}}
  \]
  for parameters $x_i \in G$, $1 \le i \le 2^k$.  The value of $x_i$ is a function only of $\{g_\omega \colon \omega \subseteq S_i\}$, and moreover can be expressed as a fixed word of bounded length in these elements.  Similarly $g_\omega$ depends only on $\{x_i \colon S_i \subseteq \omega\}$.

  Finally, we have that $(g_\omega) \in \HK^k(G_\bullet)$ if and only if $x_i \in G_{|S_i|}$ for all $i$.  Hence, there is a homeomorphism
  \begin{align*}
    \HK^k(G_\bullet) &\cong \prod_{i=1}^{2^k} G_{|S_i|} \\
                     &\cong \prod_{r=1}^k G_r^{k \choose r } \ .
  \end{align*}
  However, this is \emph{not} an isomorphism of groups.
\end{proposition}
\begin{proof}[Proof sketch]
  We will give the key ideas of the proof; for any missing details, see \cite{ben-book}.

  The first part is fairly straightforward.  Given $(g_\omega)$, we can solve for the variables $x_i$ recursively in order from $1$ to $2^k$.  Indeed, suppose $1 \le i \le 2^k$ and $x_j$ have already been chosen for $j < i$.  It is clear that we must have
  \[
    g_{S_i} = \left(\prod_{j \in \{1, \dots, i-1\},\ S_j \subsetneq S_i} x_j\right) \cdot x_i
  \]
  and we can achieve this uniquely by setting
  \[
    x_i := \left(\prod_{j \in \{1, \dots, i-1\},\ S_j \subsetneq S_i} x_j\right)^{-1} g_{S_i}  \ .
  \]
  It is clear inductively from this construction that $x_i$ only depends on $\{g_\omega \colon \omega \subseteq S_i\}$.

  Now we consider the statement about $\HK^k(G_\bullet)$.  Certainly any element of the stated form with $x_i \in G_{|S_i|}$ is in $\HK^k(G_\bullet)$; conversely, since all the generators of $\HK^k(G_\bullet)$ have this form, we will be done if we can show that the set of such elements is closed under multiplication.

  In particular, it would suffice to show that for any $1 \le j \le 2^k$ and $y \in G_{|S_j|}$, we have
  \[
    \left(\prod_{i=1}^{2^k} [x_i]_{F_{S_i}} \right) [y]_{F_{S_j}} = \prod_{i=1}^{2^k} [x'_i]_{F_{S_i}}
  \]
  for some new choice of variables $x'_i$.

  We note the obvious product identity
  \[
    [x]_F [y]_F = [x y]_F
  \]
  and the slightly less obvious commutator identity
  \[
    [x]_F [y]_{F'} [x]_{F}^{-1} [y]_{F'}^{-1} = [ x y x^{-1} y^{-1} ]_{F \cap F'} \ .
  \]
  If $F = F_S$ and $F' = F_T$ then $F \cap F' = F_{S \cup T}$.  By the filtration property, if $x \in G_{|S|}$ and $y \in G_{|T|}$ then $x y x^{-1} y^{-1} \in G_{|S| + |T|} \subseteq G_{|S \cup T|}$ and so the left hand side is again a generator of the Host--Kra cube group.  Similarly if $x, y \in G_{|S|}$ then so is $x y$, and so $[x y]_{F_S}$ is again a generator.

  This means that we can move the new term $[y]_{F_{S_j}}$ to the left in the product, generating new error terms as we go, but ones of strictly ``lower order'' as measured by the larger size of the set $S$ concerned.  Arguing by induction on $|S_j|$, we can iteratively clean up these errors, and the process clearly terminates when $|S_j| = k$.
\end{proof}

\begin{remark}
  One can be completely explicit about the various expressions that arise in this proof, e.g.~the expressions for $x_i$ in terms of $(g_\omega)$.  These turn out to be products with alternating inverse signs, traversing ``Gray codes'' around the vertices of $\{0,1\}^k$.  This observation was first made in \cite{CS12}*{Section 1.2}; the reader could also consult \cite{ben-book} for an exposition.
\end{remark}

This Proposition allows us to compute some spaces of Host--Kra cubes in the case of our example filtered groups.

\begin{example}
  \label{ex-A-group-cubes}
  Consider Example \ref{ex-A}.  Then $\HK^0(G) = G$, $\HK^1(G) = G^2$ and
  \[
    \HK^2(G) = \{ (x,y,z,w) \in G^4 \colon x - y - z + w = 0 \}
  \]
  i.e.~the Host--Kra $2$-cubes correspond exactly to parallelograms.  Similarly, $\HK^k(G)$ consists precisely of parallelepipeds of dimension $k$ in the usual sense.
\end{example}

\begin{example}
  \label{ex-B-group-cubes}
  In Example \ref{ex-B}, we have $\HK^0(G) = G$, $\HK^1(G) = G^2$ and also $\HK^2(G) = G^4$.  The interesting case is $\HK^3(G)$ which is given by
  \[
    \HK^3(G) = \left\{ (x_{000}, \dots, x_{111}) \in \RR^{\{0,1\}^3} \colon x_{000} - x_{001} - x_{010} + x_{011} - \dots - x_{111} = 0 \right\}
  \]
  i.e.~all tuples whose alternating sum is zero.  Indeed, if we write $(x_\omega)$ as
  \[
    \sum_{i=1}^{8} [y_i]_{F_{S_i}}
  \]
  as in the proposition, the conditions on $y_i$ just reduce to $y_8 = 0$.  It is not hard to extract the explicit formula (in the abelian case)
  \[
    y_8 = x_{000} - x_{001} - x_{010} + x_{011} - \dots - x_{111}
  \]
  from the proof of the proposition.
\end{example}

\begin{example}\label{ex:Ds}
We consider the degree $s$ filtration on an Abelian group
\[
A=A_0=A_1=\ldots=A_s\supseteq A_{s+1}=\{0\},
\]
because of the important role it plays in the weak structure theory.

Specifically, we verify the claim made in Proposition \ref{pr:Ds} that
a configuration $c:\{0,1\}^{s+1}\to A$ is a cube in $\cD_s(A)$
(i.e. $c\in \HK^{s+1}(A_\bullet)$) if and only if \footnote{Here $|\omega|$ denotes the number of $1$ entries in $\omega$.}
\[
L(c):=\sum_{\omega\in\{0,1\}^{s+1}}(-1)^{|\omega|}c(\omega)=0.
\]

We first observe that $L(c)$ is an additive function on configurations, since
the group is commutative.
If we write $c=\sum[x_i]_{F_{S_i}}$ as in Proposition \ref{prop:hk-decomposition},
we get
\[
L(c)=\sum L([x_i]_{F_{S_i}}).
\]

We note that $L([x]_F)=0$ for any face $F$ of dimension at least $1$ and
$L([x]_F)=\pm x$ for a face $F$ of dimension $0$, as follows directly from the
definition of $L$.
Thus $L(c)=\pm x_{2^n}$.
Now the claim follows from  Proposition \ref{prop:hk-decomposition}, which
asserts that $c\in \HK^s(A_\bullet)$ if and only if $x_{2^n}\in A_{s+1}=\{0\}$.
\end{example}

\begin{example}
  For Example \ref{ex-C}, the Host--Kra $k$-cubes are the direct product of the $k$-cubes from Examples \ref{ex-A-group-cubes} and \ref{ex-B-group-cubes}.  In particular, $\HK^3(G)$ consists of tuples $(x_\omega, y_\omega)$ such that $(x_\omega)$ form a parallelepiped and the alternating sum of $(y_\omega)$ is zero.
\end{example}

\begin{example}
  \label{ex-D-cubes}
  For the Heisenberg group (Example \ref{ex-D}), despite the non-abelian group law the situation is similar to the previous example.  Certainly $\HK^0 = \cH$, $\HK^1 = \cH^2$ as usual.  A configuration
  \[
    \omega \mapsto \heis{x_\omega}{y_\omega}{z_\omega}
  \]
  is in $\HK^2(\cH_\bullet)$ if and only if it has the form
  \[
    [g_1]_{F_{\emptyset}} [g_2]_{F_{\{1\}}} [g_3]_{F_{\{2\}}} [g_4]_{F_{\{1,2\}}}
  \]
  for some $g_1,g_2,g_3 \in \cH$ and $g_4 \in \cH_2$.  In the ``abelian'' variables $x_\omega$, $y_\omega$ this condition becomes
  $x_{00} - x_{01} - x_{10} + x_{11} = 0$ and $y_{00} - y_{01} - y_{10} + y_{11} = 0$; and for $z_\omega$ we have complete freedom.

  Now we consider
  \[
    \omega \mapsto \heis{x_\omega}{y_\omega}{z_\omega}
  \]
  which it turns out is in $\HK^3(\cH_\bullet)$ if and only if:
  \begin{itemize}
    \item $(x_\omega)$ is a parallelepiped;
    \item $(y_\omega)$ is a parallelepiped;
    \item the alternating sum $\sum_{\omega \in \{0,1\}^3} (-1)^{|\omega|} z_\omega$ is zero.
  \end{itemize}
  The first two points are fairly easy to justify: since this calculation really takes place on the abelianization of $\cH$ which is just $\RR^2$, this strongly resembles Example \ref{ex-A-group-cubes}.

  For the third, we again use a decomposition
  \[
    \prod_{i=1}^8 [g_i]_{F_{S_i}}
  \]
  and consider the constraint $g_8 = \id$ and in particular its $z$-coordinate.  As in Example \ref{ex-B-group-cubes}, the contribution from $z_\omega$ is just this alternating sum.  It is not immediately obvious that this expression should not also involve terms of the form $x_\omega y_{\nu}$ arising from the cross-terms in the group operation, but a calculation shows that these do indeed drop out.
\end{example}

\subsection{Some topological properties of the Host--Kra groups}

In the second paper of this project \cite{GMV2}, we will use some topological properties
of the Host--Kra groups.
These are easy to deduce from Proposition \ref{prop:hk-decomposition},
so we record them now for convenient reference.

\begin{lemma}\label{lm:hk-homomorphisms}
Let $G_\bullet$, $H_\bullet$ be two filtered topological groups.
Let $\tau:G\to H$ be a homomorphism such that
$\tau(G_i)\subseteq\tau(H_i)$.
Then $\tau$ induces a homomorphism $\tau: \HK^n(G_\bullet)\to\HK^n(H_\bullet)$
for each $n$ by pointwise application on the vertices.

If $\tau:G_i\to H_i$ is open for each $i$, then so is the induced homomorphism
 $\tau: \HK^n(G_\bullet)\to\HK^n(H_\bullet)$
for each $n$.
\end{lemma}

\begin{lemma}\label{lm:hk-subgroups}
Let $G_\bullet$, $H_\bullet$ be two filtered topological groups.
Suppose $G_i\subseteq H_i$ for each $i$.
If $G_i$ are open in $H_i$ (resp., connected) for each $i$, then
$\HK^n(G_\bullet)$ is also open in $\HK^n(H_\bullet)$ (resp., connected) for each $n$.
\end{lemma}

We recall from Proposition \ref{prop:hk-decomposition} that each element
$(g_\omega)$ can be written uniquely in the form
$\prod [x_i]_{F_{S_i}}$.
For convenient reference, we call $(g_\omega)$ the {\em vertex coordinates}
and $(x_i)$ the {\em face coordinates}.
We also recall that the two set of coordinates can be expressed from each other
by word maps of finite length.

\begin{proof}[Proof of Lemma \ref{lm:hk-homomorphisms}]
   That $\tau(\HK^n(G_\bullet)) \subseteq \HK^n(H_\bullet)$ is clear by considering the generators.

We observe that the face coordinates give rise to the same homomorphism by pointwise application, since the two sets of coordinates can be expressed from each other using word maps.
The last claim follows from the description with face coordinates, since a product of open maps is open.
\end{proof}

\begin{proof}[Proof of Lemma \ref{lm:hk-subgroups}]
We use the description with face coordinates and the fact that the direct product
of open (resp., connected) sets is open (resp., connected).
\end{proof}

\subsection{The Host--Kra group $\HK^k(G_\bullet)$ is a nilspace}

Another crucial corollary of Proposition \ref{prop:hk-decomposition} is a version of the \emph{corner constraint}.

\begin{corollary}
  \label{hk-group-uniqueness}
  Suppose $G_\bullet$ is a degree $s$ filtered group, set $k = s+1$ and suppose $g, g' \in \HK^k(G_\bullet)$ have the property that $g(\omega) = g'(\omega)$ for all $\omega \ne \vec{1}$.  \footnote{As ever, we write $\vec{1}$ for the element $(1, \dots, 1) \in \{0,1\}^k$.}

  Then $g(\vec{1}) = g'(\vec{1})$.
\end{corollary}
\begin{proof}
  Apply Proposition \ref{prop:hk-decomposition} to $g$ and $g'$ to obtain parameters $x_i$, $x'_i$.  Every $x_i$ other than the last one $x_{2^k}$ does not depend on $g(\vec{1})$, and similarly for $g'$.  But $G_k = \{\id\}$ by assumption, so $x_{2^k} = x'_{2^k} = \id$ in the decomposition.  Hence $g = g'$.
\end{proof}

We consider the \emph{corner completion} property discussed at length above.

\begin{corollary}
  \label{hk-group-completion}
  Suppose we have a configuration $(g_\omega)_{\omega \in \{0,1\}^k \setminus \{\vec{1}\}}$ such that every lower face is in $\HK^{k-1}(G_\bullet)$.  Then there exists $g \in G$ such that setting $g_{\vec{1}} = g$ we have $(g_\omega)_{\omega \in \{0,1\}^k} \in \HK^k(G_\bullet)$.
In other words, the cubespace $(G, \HK^k(G_\bullet))$ is fibrant.
\end{corollary}
Note the previous corollary says precisely that the choice of $g$ is unique when $k=s+1$.
\begin{proof}
  By Proposition \ref{prop:hk-decomposition} we may define all the variables $x_i$ for $1 \le i < 2^k$ without knowing $g_{\vec{1}}$, such that
  \[
    (g_\omega) = \prod_{i=1}^{2^k-1}[x_i]_{ F_{S_i}}
  \]
  whenever this is defined.

We show that  $x_j \in G_{|S_j|}$.
To this end, we choose a lower face $E$ that contains the vertex $S_j$ and consider the product
\[
\prod_{i\in\{1,\ldots, 2^k-1\}:S_i\in E}^{2^k-1}[x_i]_{ F_{S_i}}|_E,
\]
where $\cdot|_E$ denotes restriction to $E$.
We observe that this is a product decomposition of a cube in $\HK^{\dim E}(G_\bullet)$ of the form that appears
in Proposition \ref{prop:hk-decomposition}.
The second part of that proposition now implies that  $x_j \in G_{|S_j|}$.

If we set $x_{2^k} = \id$, then
  \[
    \prod_{i=1}^{2^k-1} [x_i]_{F_{S_i}}
  \]
  is an element of $\HK^k(G_\bullet)$ extending $(g_\omega)$ as required.
\end{proof}

\subsection{Properties of cubes on nilmanifolds}

We recall from Definition \ref{defn:nilmanifold} that the notion of Host--Kra cubes on $G_\bullet$ induces one on nilmanifolds $G/\Gamma$.

Many of the nice properties of the Host--Kra cubes on $G$ are inherited on the nilmanifold $G/\Gamma$, even though the latter typically has no compatible group operation.  We will verify some instances of this that are of particular importance in the abstract setting.
We begin by showing that the Host--Kra cubes arising from a filtration of degree $s$ satisfy $(s+1)$-uniqueness.

\begin{proposition}
  \label{prop:nilmanifold-uniqueness}
  Take $G_\bullet$, $\Gamma$ as in the definition, where the filtration $G_\bullet$ has degree $s$.  Set $k = s+1$, and suppose $c, c' \in \HK^k(G_\bullet)/\Gamma$ have the property that $c(\omega) = c'(\omega)$ for all $\omega \ne \vec{1}$.  Then $c = c'$.
\end{proposition}
\begin{proof}
  We write $\pi \colon G \to G/\Gamma$ for the projection map.  Let $g, g' \in \HK^k(G_\bullet)$ be such that $\pi(g) = c$, $\pi(g') = c'$.  Hence $g^{-1} g'$ is in $\HK^k(G_\bullet)$, and all of its entries lie in $\Gamma$ except possibly the $\vec{1}$ entry.  Applying Proposition \ref{prop:hk-decomposition}, we find that all the coefficients $x_i$ in the expansion
  \[
    g^{-1} g' = \prod_{i=1}^{2^k} [x_i]_{F_{S_i}}
  \]
  lie in $\Gamma$, except possibly the last one $x_{2^k}$; but in fact $x_{2^k} \in G_k = \{\id\}$ and so $x_{2^k}$ is in $\Gamma$ also.  Hence, $(g^{-1} g')(\vec{1}) \in \Gamma$, and so $g(\vec{1}) \Gamma = g'(\vec{1}) \Gamma$ as required.
\end{proof}

We now turn to the completion property.
\begin{proposition}
  \label{prop:nilmanifold-completion}
  Take $G_\bullet$, $\Gamma$ as above, and suppose $(y_\omega)_{\omega \in \{0,1\}^k \setminus \{\vec{1}\}}$ is a configuration on $G/\Gamma$ such that every lower face is in $\HK^{k-1}(G_\bullet)/\Gamma$.  Then there exists $y \in G/\Gamma$ such that setting $y_{\vec{1}} = y$ results in an element of $\HK^k(G_\bullet)/\Gamma$.
\end{proposition}

\begin{proof}
  Since $(G, \HK^k(G_\bullet))$ is fibrant (Corollary \ref{hk-group-completion}), it suffices by the universal property (Lemma \ref{lem:universal1}) to verify that the projection $G \to G/\Gamma$ is a fibration.

  Unpacking the definitions, this says the following.  Suppose we are given a cube $c \in \HK^k(G_\bullet)/\Gamma$; by definition this has the form $\omega \mapsto g(\omega) \Gamma$ where $\omega \mapsto g(\omega)$ is a Host--Kra cube of $G_\bullet$.  Suppose we are given another $k$-corner
  \[
    h \colon \{0,1\}^k \setminus \{\vec{1}\} \to G
  \]
  such that $h(\omega) \Gamma = g(\omega) \Gamma$ for all $\omega \ne \vec{1}$.  Then we need to show there is an $h(\vec{1}) \in G$ such that $h(\vec{1}) \Gamma = g(\vec{1}) \Gamma$, and furthermore this completes $h$ to a cube of $G$.

  Consider the $k$-corner $b \colon \omega \mapsto g(\omega)^{-1} h(\omega)$.  By assumption this takes values in $\Gamma$.  By Proposition \ref{prop:hk-decomposition} we may write $b$ as an ordered product
  \[
    \prod_{i=1}^{2^k} [x_i]_{F_{S_i}}
  \]
  for some $x_i \in G_{|S_i|}$, where all but $x_{2^k}$ are determined.  Moreover, all of these elements lie in $\Gamma$.  Setting $x_{2^k} = \id$ (say) we complete $b$ to a cube in $\HK^k(G_\bullet)$ with entries in $\Gamma$, and then we can set $h(\vec{1}) = g(\vec{1}) b(\vec{1})$ as required.
\end{proof}

\subsection{Examples of cubes on nilmanifolds}

Our next task is to explore the consequences of Definition \ref{defn:nilmanifold} concerning cubes on nilmanifolds, in the setting of our favourite examples.

The abelian cases here are all fairly straightforward.

\begin{example}
  Suppose we take $\Gamma = \ZZ$ in Examples \ref{ex-A} or \ref{ex-B}, or $\Gamma = \ZZ^2$ in Example \ref{ex-C}.  Then $G/\Gamma$ actually does have a group structure in this cases; specifically, it is the torus $\RR/\ZZ$ or $\RR^2/\ZZ^2$ respectively.  Moreover, this compact group inherits a filtration from that on $\RR$ or $\RR^2$.

  In this case, the Host--Kra cubes on the ``nilmanifold'' $G/\Gamma$ agree with the Host--Kra construction on the group $G/\Gamma$.
\end{example}

In the case of the Heisenberg group (Example \ref{ex-D}) and the corresponding nilmanifolds, it is much harder to obtain a completely explicit description of the Host--Kra cubespaces $\HK^k(G_\bullet) / \Gamma$.  We do so anyway, both as an instructive exercise in its own right, and to emphasize some of the difficulties that arise in trying to obtain a simple description of nilspaces in general.

\begin{example}
  \label{ex-D-nilmanifold-cubes}
  Consider again the Heisenberg case (Example \ref{ex-D}) with $\Gamma$ the subgroup of $\cH$ consisting of matrices with integer entries.  A fundamental domain for the quotient $\cH/\Gamma$ consists of all matrices
  \[
    \left\{ \heis{x}{y}{z} \colon x,y,z \in [0,1) \right\}
  \]
  and the reduction map to the fundamental domain is
  \[
    \heis{x}{y}{z} \mapsto \heis{x}{y}{z} \heis{-\lfloor x \rfloor}{-\lfloor y \rfloor}{-\left\lfloor z - x \lfloor y \rfloor \right\rfloor} = \heis{\{x\}}{\{y\}}{\left\{z - x \lfloor y \rfloor \right\}}
  \]
  where $\lfloor\cdot\rfloor$, $\{\cdot\}$ denote the integer and fractional parts respectively.

  Let
  \[
    c(\omega) = \heis{x_\omega}{y_\omega}{z_\omega}
  \]
  be the representatives in the fundamental domain of an element of $(G/\Gamma)^{\{0,1\}^3}$.  We wish to find necessary and sufficient conditions for this to represent an element of $\HK^3(\cH_\bullet)/\Gamma$.

  By definition, this holds if and only if there is a configuration
  \[
    c'(\omega) = \heis{r_\omega}{s_\omega}{t_\omega}
  \]
  in $\HK^3(\cH_\bullet)$ such that $c(\omega)$ are precisely the representatives of $c'(\omega)$ in the fundamental domain; i.e.~if
  \[
    \heis{x_\omega}{y_\omega}{z_\omega} = \heis{\{r_\omega\}} {\{s_\omega\}} {\left\{t_\omega - r_\omega \lfloor s_\omega \rfloor \right \}}
  \]
  for all $\omega$.

  By the description of $\HK^3(\cH_\bullet)$ (see Example \ref{ex-D-cubes}) this implies that $(x_\omega \ \bmod 1)_{\omega \in \{0,1\}^3}$ and $(y_\omega \ \bmod 1)_{\omega \in \{0,1\}^3}$ must be three-dimensional parallelepipeds over $\RR/\ZZ$.  One might expect a corresponding identity
  \[
    z_{000} - z_{001} - z_{010} + z_{011} - \dots - z_{111} \equiv 0 \pmod{1}
  \]
  as was the case for $\HK^3(\cH_\bullet)$; but in fact this gets ``twisted'' by the non-abelian group action.
  For instance, taking the cube
  \[
    c(\omega) = \heis{1/2}{(\omega_1 + \omega_2 + \omega_3)/3}{0}
  \]
  in $\HK^3(\cH_\bullet)$, its projection to the fundamental domain is
  \[
    c'(\omega) = \begin{cases} \heis{1/2}{(\omega_1 + \omega_2 + \omega_3)/3}{0} &\colon \omega \ne \vec{1} \\ \heis{1/2}{0}{1/2} &\colon \omega = \vec{1} \end{cases}
  \]
  and so the alternating sum over the $z$ coordinate is $1/2 \bmod 1$. However, once this value $1/2$ as a function of $x_\omega, y_\omega$ is fixed, it \emph{is} true that e.g.
  \[
    c''(\omega) = \heis{1/2}{1/3(\omega_1 + \omega_2 + \omega_3) \bmod 1}{z_\omega}
  \]
  is in $\HK^3(\cH_\bullet)/\Gamma$ if and only if $z_{000} - z_{001} - \dots + z_{110} - z_{111} = 1/2 \pmod{1}$.

  It is possible to do this calculation in general.  We compute
  \begin{align*}
    \sum_\omega (-1)^{|\omega|} z_\omega &\equiv \sum_\omega (-1)^{|\omega|} (t_\omega - r_\omega \lfloor s_\omega \rfloor) \pmod{1} \\
                                   &\equiv - \sum_\omega (-1)^{|\omega|} r_\omega \lfloor s_\omega \rfloor \pmod{1}
  \end{align*}
  since the alternating sum over $t_\omega$ is known to be zero. Since $r_\omega - x_\omega$ is an integer, we can rewrite this as
  \[
    - \sum_\omega (-1)^{|\omega|} x_\omega \lfloor s_\omega \rfloor \pmod{1}
  \]
  and by exploiting the fact that $s_\omega \equiv y_\omega \pmod{1}$ together with the parallelepiped conditions for $(s_\omega)$ and $(x_\omega \ \bmod 1)$ and a little calculation, we can derive a rather inelegant formula
  \begin{align*}
    \sum_\omega (-1)^{|\omega|} z_\omega \equiv &x_{110} (y_{110} - y_{100} - y_{010} + y_{000}) \\
                                        + &x_{101} (y_{101} - y_{100} - y_{001} + y_{000}) \\
                                        + &x_{011} (y_{011} - y_{010} - y_{001} + y_{000}) \\
                                        + &x_{111} (y_{111} + 2 y_{000} - y_{001} - y_{010} - y_{100}) \ .
  \end{align*}
  Together with the previous conditions, this gives necessary and sufficient conditions for a configuration $(x_\omega, y_\omega, z_\omega)$ to be a Host--Kra cube of $\cH / \Gamma$.

  The specific form of this explicit ``cocycle'' on the right hand side  is not terribly important.  What is useful to bear in mind is that these conditions have the structure:
  \begin{itemize}
    \item the ``abelianization'' $(x_\omega \bmod 1, y_\omega \bmod 1) \in \RR^2 / \ZZ^2$ must be a parallelepiped; and
    \item the alternating sum $\sum_\omega (-1)^{|\omega|} z_\omega \pmod{1}$ is given by some function only of the $x$'s and $y$'s.
  \end{itemize}
\end{example}

\subsection{Polynomial sequences}

We finally give a brief discussion of \emph{polynomial sequences}.

One of the key features of filtered groups is that they admit a natural definition of a \emph{polynomial sequence}.  These are a class of functions $\ZZ \to G$, generalizing the notion of a polynomial $\ZZ \to \RR$.

\begin{definition}
  \label{defn:old-poly}
  For any $f \colon \ZZ \to G$ and $h, x \in \ZZ$ write $\partial_h f \colon x \mapsto f(x+h) f(x)^{-1}$.  A map $f \colon \ZZ \to G$ is called a \emph{polynomial sequence} if $\partial_{h_i} \dots \partial_{h_1} f (x) \in G_i$ for each $i \ge 0$ and $h_1, \dots, h_i \in \ZZ$.
\end{definition}

To justify the term ``polynomial'', we consider this definition in the context of our examples.

\begin{example}
  In Example \ref{ex-A}, the polynomial sequences are precisely the polynomials of degree at most $1$ on $\ZZ$, i.e.~maps of the form $n \mapsto a n + b$ for some $a,b \in \RR$.  Indeed, the condition translates to the functional equation $p(x) - p(x+h_1) - p(x+h_2) + p(x+h_1+h_2) = 0$ for all $x,h_1,h_2 \in \ZZ$, which characterizes affine-linear functions.
\end{example}

\begin{example}
  Similarly, in Example \ref{ex-B}, we get precisely the polynomials of degree at most $2$ on $\ZZ$, i.e.~maps of the form $n \mapsto a n^2 + b n + c$ for some $a,b,c \in \RR$.  As before, the definition expands to a functional equation for $p$,
  \[
    p(x) - p(x+h_1) - p(x+h_2) + p(x+h_1+h_2) - p(x+h_3) + p(x+h_1+h_3) + p(x+h_2+h_3) - p(x+h_1+h_2+h_3) = 0
  \]
  which has the required properties.  Alternatively, the definition holds if and only if $\partial_h p$ is a polynomial of degree $1$ for all $h \in \ZZ$.
\end{example}

\begin{example}
  Example \ref{ex-C} is a hybrid of these: the polynomial sequences are of the form $n \mapsto (p_1(n), p_2(n))$ where $p_1$ has degree at most $1$ and $p_2$ has degree at most $2$.
\end{example}

\begin{example}
  The polynomial sequences for Example \ref{ex-D} have the form
  \[
    n \mapsto \heis{a_1 n + a_0}{b_1 n + b_0}{c_2 n^2 + c_1 n + c_0}
  \]
  for any choice of $a_i, b_i, c_i \in \RR$.  This is no longer quite so clear, due to the non-abelian group operation implicit in the definition of a polynomial sequence.  However, it follows from a short computation, or from Proposition \ref{prp:poly-characterize} below.
\end{example}

These sequences are fairly hard to study directly from this definition.  For instance, it is true that if $p(n)$ and $p'(n)$ are polynomial sequences then so is $p(n) p'(n)$, but a direct argument is fairly involved.

The theory becomes much cleaner, however, if we invoke the notion of Host--Kra cube groups defined in Definition \ref{def:cube-group}.  The key result is the following.

\begin{proposition}\label{prp:poly-characterize}
  Let $G_\bullet$ be a filtered group and $p \colon \ZZ \to G$ be given.  Then $p$ is a polynomial sequence if and only if the following holds: for any $k \ge 0$ and $x, h_1, \dots, h_k \in \ZZ$, the configuration
  \begin{align*}
    \{0,1\}^k &\to G \\
       \omega &\mapsto p\left(x + \sum_{i=1}^k h_i \omega_i \right)
  \end{align*}
  lies in $\HK^k(G_\bullet)$.

Equivalently, $p$ is a polynomial sequence if and only if it maps $\HK^k(\ZZ_\bullet)$ into $\HK^k(G_\bullet)$ for each $k$, where $\ZZ$ is given the usual filtration $\ZZ_0 = \ZZ_1 = \ZZ$, $\ZZ_i = \{0\}$ for $i \ge 2$.
\end{proposition}
For a proof, see \cite{green-tao-quantitative}*{Section 6}.  Hence, the notion of a polynomial map $\ZZ \to G_\bullet$ (which is the common term in the literature) coincides with the concept from the introduction of a map that sends cubes to cubes; or with the formal notion of a cubespace morphism.

In the case of maps $\ZZ \to G/\Gamma$, no criterion along the lines of \ref{defn:old-poly} is available.  The most common convention in the literature is to refer to a map $p \colon \ZZ \to G/\Gamma$ as polynomial if it is the projection of a polynomial map $p \colon \ZZ \to G$ in the existing sense.

Alternatively, one can simply use the notion of a cubespace morphism $\ZZ \to G/\Gamma$; i.e., a map that sends Host--Kra cubes of $\ZZ$ (i.e.~parallelepipeds) to Host--Kra cubes of $G/\Gamma$.

It turns out that these notions coincide.  However, this is not the case if $\ZZ$ is replaced by another group; e.g.~in the case of maps $\ZZ/N\ZZ \to G/\Gamma$, we find that there are no non-constant polynomial maps $\ZZ/N\ZZ \to G$ but plenty of maps $\ZZ/N\ZZ \to G/\Gamma$ that send cubes to cubes.  The reader is encouraged always to have in mind the latter notion.

\printindex

\printindex[nota]

\bibliography{higher-intro-bib}{}

\end{document}